%% file: main.tex
\documentclass[10pt]{amsart}

\usepackage{latexsym,exscale,enumerate,amsfonts,amssymb,xparse,mathtools}

\usepackage[normalem]{ulem}
\usepackage{amsmath,amsthm,amsfonts,amssymb,amscd, stmaryrd,textcomp}
\usepackage[usenames,dvipsnames]{xcolor}
\usepackage{enumitem}
\usepackage{verbatim}
\usepackage{ytableau}
\usepackage{euscript}

\usepackage{tikz,mathtools,extarrows, tikz-cd}
\usetikzlibrary{decorations.markings,decorations.pathreplacing, decorations.pathmorphing, calligraphy,positioning}
\usepackage[margin=2in]{geometry}
\usepackage[bookmarks=false]{hyperref}

\numberwithin{equation}{section}

\colorlet{darkblue}{blue!70!black}
\colorlet{darkred}{red!70!black}
\colorlet{darkgreen}{green!70!black}
\colorlet{darkmagenta}{magenta!70!black}
\colorlet{darkwhite}{white!65!black}

\newenvironment{tric}
    {\begin{tikzpicture}[scale= 0.5,line width=0.45mm,draw=darkblue,double distance=0.45mm,
        baseline={([yshift=-.8ex]current bounding box.center)}] }
    {\end{tikzpicture}}

\newenvironment{Gtric}
    {\begin{tikzpicture}[scale=  0.5, 
    line width=0.55mm, draw=darkwhite, double distance=0.55mm,
        baseline={([yshift=-.8ex]current bounding box.center)}] }
    {\end{tikzpicture}}

\newenvironment{FGtric}
    {\begin{tikzpicture}[scale=  0.5, 
    line width=0.55mm, draw=Apricot, double distance=0.55mm,
        baseline={([yshift=-.8ex]current bounding box.center)}] }
    {\end{tikzpicture}}

\newenvironment{tricc}
    {\begin{tikzpicture}[scale= 0.4,line width=0.45mm,,draw=darkgreen,double distance=0.45mm,
        baseline={([yshift=-.8ex]current bounding box.center)}] }
    {\end{tikzpicture}}

\newtheorem{thm}{Theorem}[section]
\newtheorem{lemma}[thm]{Lemma}
\newtheorem{proposition}[thm]{Proposition}
\newtheorem{corollary}[thm]{Corollary}
\newtheorem{cor}[thm]{Corollary}

\newtheorem{notation}[thm]{Notation}
\newtheorem{example}[thm]{Example}

\theoremstyle{definition}
\newtheorem{definition}[thm]{Definition}
\newtheorem{defn}[thm]{Definition}

\theoremstyle{remark}
\newtheorem{remark}[thm]{Remark}
\newtheorem{rmk}[thm]{Remark}

\DeclareMathOperator{\id}{id}

\DeclareMathOperator{\tr}{trace}

\DeclareMathOperator{\wt}{\rm{wt}}

\DeclareMathOperator{\Hom}{\rm{Hom}}
\DeclareMathOperator{\End}{\rm{End}}

\DeclareMathOperator{\RepC}{\textbf{Rep}(U_{\kk}(\om))}
\DeclareMathOperator{\StdRep}{\textbf{StdFund}(U_{\R}(\om))}
\DeclareMathOperator{\StdRepK}{\textbf{StdFund}(U_{\K}(\om))}
\DeclareMathOperator{\StdRepA}{\textbf{StdFund}(U_{\A}(\om))}
\DeclareMathOperator{\StdRepC}{\textbf{StdFund}(U_{\kk}(\om))}
\DeclareMathOperator{\Fund}{\textbf{Fund}(U_{\R}(\om))}
\DeclareMathOperator{\FundK}{\textbf{Fund}(U_{\K}(\om))}
\DeclareMathOperator{\FundA}{\textbf{Fund}(U_{\A}(\om))}
\DeclareMathOperator{\Fundk}{\textbf{Fund}(U_{\kk}(\om))}

\DeclareMathOperator{\Web}{\textbf{Web}_{\R}(O(m))}
\DeclareMathOperator{\WebK}{\textbf{Web}_{\K}(O(m))}
\DeclareMathOperator{\WebA}{\textbf{Web}_{\A}(O(m))}
\DeclareMathOperator{\WebC}{\textbf{Web}_{\kk}(O(m))}
\DeclareMathOperator{\WebB}{\textbf{StdWeb}_{\R}^{\beta}(O(m))}

\DeclareMathOperator{\StdWeb}{\textbf{StdWeb}_\R(O(m))}
\DeclareMathOperator{\StdWebK}{\textbf{StdWeb}_{\K}(O(m))}
\DeclareMathOperator{\StdWebA}{\textbf{StdWeb}_{\A}(O(m))}
\DeclareMathOperator{\StdWebC}{\textbf{StdWeb}_{\kk}(O(m))}
\DeclareMathOperator{\BMW}{\textbf{BMW}_\R(O(m))}
\DeclareMathOperator{\BMWA}{\textbf{BMW}_{\A}(O(m))}
\DeclareMathOperator{\BMWK}{\textbf{BMW}_{\K}(O(m))}
\DeclareMathOperator{\BMWC}{\textbf{BMW}_{\kk}(O(m))}

\DeclareMathOperator{\smbase}{\mathfrak{b}}
\DeclareMathOperator{\base}{\mathfrak{B}}

\DeclareMathOperator{\wta}{{\mathbf{a}}}
\DeclareMathOperator{\wtb}{\mathbf{b}}

\DeclareMathOperator{\wtk}{\gamma}
\DeclareMathOperator{\wtl}{\delta}

\DeclareMathOperator{\brweb}{{^\R\beta}_{1,1}}
\DeclareMathOperator{\brwebi}{{^\R\beta}^{-1}_{1,1}}

\DeclareMathOperator{\brwebK}{{^{\K}\beta}_{1,1}}
\DeclareMathOperator{\brwebC}{{^{\kk}\beta}_{1,1}}
\DeclareMathOperator{\brrep}{{\beta}_{V_\R,V_\R}}

\DeclareMathOperator{\brrepK}{{\beta}_{V_{\K},V_{\K}}}
\DeclareMathOperator{\brrepC}{{\beta}_{V_{\kk},V_{\kk}}}
\DeclareMathOperator{\brrepi}{{\beta^{-1}}_{V_\R,V_\R}}

\DeclareMathOperator{\brrepKi}{{\beta^{-1}}_{V_{\K},V_{\K}}}

\DeclareMathOperator{\one}{\textbf{1}}
\DeclareMathOperator{\detR}{\mathrm{det}_{\R}}

\DeclareMathOperator{\K}{\mathbb{F}}
\DeclareMathOperator{\A}{\mathbb{A}}
\DeclareMathOperator{\kk}{\mathbb{C}}
\DeclareMathOperator{\om}{\mathfrak{o}_m}
\DeclareMathOperator{\som}{\mathfrak{so}_m}

\DeclareMathOperator{\rk}{rk}
\newcommand{\R}{\textbf{R}}

\newcommand{\EQ}{\text{Equation }}
\newcommand{\RL}{\text{Relation }}
\newcommand{\EQS}{\text{Equations }}

\author{Elijah Bodish}
\address[E.B.]{
  Department of Mathematics, Indiana University Bloomington, Bloomington, IN, USA}
\email{ebodish@iu.edu}
\author{Haihan Wu}
\address[H.W.]{
  Department of Mathematics, Johns Hopkins University, Baltimore, MD, USA}
\email{hwu125@jh.edu}

\date{\today}
\title{Webs for the quantum orthogonal group}
    
\begin{document}

\maketitle

\begin{abstract}
We give a generators and relations presentation for the full monoidal subcategory of representations of the quantum orthogonal group generated by the quantum exterior powers of the defining representation.
\end{abstract}

\input{introduction.tex} 
\input{webs.tex}
\input{fund.tex}

\input{proof.tex}

\bibliographystyle{alpha}

\bibliography{mastercopy}

\end{document}

%% file: introduction.tex
\section{Introduction}

With the goal of calculating Reshetkhin-Turaev invariants in mind, Kuperberg proposed the problem of giving a generators and relations presentation of the $q$-analogue of the category of fundamental representations for any simple Lie algebra. In \cite{Kupe}, he gives such presentations, where the generating morphisms are trivalent graphs, for each rank two simple Lie algebra. Thus, the study of ``web categories" was born. Since then, the main results are analogous presentations for $\mathfrak{sl}_n$ \cite{CKM} and $\mathfrak{sp}_{2n}$ \cite{bodish2021type}. 

However, there has been limited progress in the case of special orthogonal Lie algebras. The monoidal category generated by the vector representation and the spin representation for $U_q(\mathfrak{so}_7)$ is studied using diagrammatics in \cite{West-Spin}. However, the second fundamental representation which corresponds to the exterior square of the vector representation is not included as a generating object. Webs describing intertwiners for the co-ideal subalgebra $U_q'(\mathfrak{so}_m)\subset U_q(\mathfrak{gl}_m)$\footnote{Precisely, restricting the $U_q(\mathfrak{gl}_m)$ coproduct to the subalgebra $U_q'(\som)\subset U_q(\mathfrak{gl}_m)$ one obtains a map $U_q'(\som)\rightarrow U_q'(\som)\otimes U_q(\mathfrak{gl}_m)$.} which is a non-standard quantization of $U(\mathfrak{so}_m)$, appear in \cite{ST-BCDHowe}. When $q=1$, this non-standard quantization specializes to $U(\som)$. However, since $U'_q(\som)$ is not a Hopf algebra, the category of modules over $U'_q(\som)$ is not a monoidal category. Thus, over $\mathbb{C}(q)$, there will not be a relation between the usual $q$-analogue of the category of representations of $\som$, i.e. type-$\textbf{1}$ modules over the quantum group $U_q(\mathfrak{so}_m)$, and the webs in \cite{ST-BCDHowe}.

The present work is a step in the direction of defining webs for $\som$, in the presence of some simplifying assumptions. First, we ignore spin representations, i.e. we work with $SO(m)$ instead of $\mathfrak{so}(m)$. Second, we focus on $O(m)$, which is a $\mathbb{Z}/2$ extension of $SO(m)$. The complex orthogonal group, or its $q$-analogue, acts on each exterior power of the defining representation irreducibly, and these representations are pairwise non-isomorphic. In this paper we give a generators and relations presentation for the monoidal category generated by these exterior powers.

\subsection{Results}

\def\Skeina
{\begin{tric}
\draw (0.7,0) circle (0.7);
\end{tric}
}

\def\Skeinc
{\begin{tric}
\draw[scale= 0.8] (0,0.7)--(0,1.5) (0,-0.7)--(0,-1.5)
      (0,0.7)..controls(-0.5,0.7)and(-0.5,-0.7)..(0,-0.7)  (0,0.7)..controls(0.5,0.7)and(0.5,-0.7)..(0,-0.7);
\end{tric}
}

\def\Skeinca
{\begin{tric}
\draw [scale=0.8] (0,1.5)--(0,-1.5);
\end{tric}
}

\def\Skeind
{\begin{tric}
\draw [scale=0.5] (-1,0)--(1,0)--(0,1.732)--cycle
                  (-1,0)--(-2,-0.577) (1,0)--(2,-0.577) (0,1.732)--(0,2.887);
\end{tric}
}

\def\Skeinda
{\begin{tric}
\draw[scale=0.5]
    (0,0.577)--(-2,-0.577) (0,0.577)--(2,-0.577) (0,0.577)--(0,2.887);
\end{tric}
}

\def\Skeindb
{\begin{tric}
\draw [scale=0.65] (0,0.7)..controls(0.5,0.7)and(0.5,1.732)..(0,1.732)  (0,0.7)..controls(-0.5,0.7)and(-0.5,1.732)..(0,1.732)  
(-1,-0.5)--(0,0)--(1,-0.5) (0,1.732)--(0,2.887);
\draw[double,thin,scale=0.65]  (0,0)--(0,0.7) ;
\end{tric}
}

\def\Skeindc
{\begin{tric}
\draw [scale=0.65] (0,0.7)..controls(0.5,0.7)and(0.5,1.732)..(0,1.732)  (0,0.7)..controls(-0.5,0.7)and(-0.5,1.732)..(0,1.732)  
(-1,-0.5)--(0,0)--(1,-0.5) (0,0)--(0,0.7) (0,1.732)--(0,2.887)  ;
\end{tric}
}

\def\Skeindd
{\begin{tric}
\draw [scale=0.65] (0,0.7)..controls(0.5,0.7)and(0.5,1.732)..(0,1.732)  (0,0.7)..controls(-0.5,0.7)and(-0.5,1.732)..(0,1.732)  
(-1,-0.5)..controls(0,0)..(1,-0.5)  (0,1.732)--(0,2.887)  ;
\end{tric}
}

\def\Skeine
{\begin{tric}
\draw [scale=0.4](-1,-1)--(-1,1)--(1,1)--(1,-1)--cycle;
\draw [scale=0.4](-1,-1)--(-2,-2) (-1,1)--(-2,2) (1,1)--(2,2) (1,-1)--(2,-2);
\end{tric}
}

\def\Skeinea
{\begin{tric}
\draw[scale=0.35] (-2,-1.732)--(-1,0)--(-2,1.732) (2,1.732)--(1,0)--(2,-1.732) (-1,0)--(1,0);
\end{tric}
}

\def\Skeineb
{\begin{tric}
\draw [scale=0.35]  (-1.732,2)--(0,1)--(1.732,2) (1.732,-2)--(0,-1)--(-1.732,-2) (0,-1)--(0,1);
\end{tric}
}

\def\Skeinec
{\begin{tric}
\draw [scale=0.3] (-2,-2)..controls(-1,-1)and(-1,1)..(-2,2) (2,2)..controls(1,1)and(1,-1)..(2,-2);
\end{tric}
}

\def\Skeined
{\begin{tric}
\draw  [scale=0.3] (-2,2)..controls(-1,1)and(1,1)..(2,2) (2,-2)..controls(1,-1)and(-1,-1)..(-2,-2) ;
\end{tric}
}

\def\Skeinef
{\begin{tric}
\draw [scale=0.3](2,2)--(-2,-2)  (-2,2)--(2,-2);
\end{tric}
}

\def\Skeinf
{\begin{tric}
\draw [scale=0.5] (18:1)--(90:1)--(162:1)--(234:1)--(306:1)--cycle
      (18:1)--(18:2) (90:1)--(90:2) (162:1)--(162:2) (234:1)--(234:2) (306:1)--(306:2);
\end{tric}
}

\def\Skeinfa
{\begin{tric}
\draw [scale=0.5] (18:1)--(90:1)--(162:1)
      (18:1)--(18:2) (90:1)--(90:2) (162:1)--(162:2) (162:1)..controls(234:1)..(234:2) (18:1)..controls(306:1)..(306:2);
\end{tric}
}

\def\Skeinfb
{\begin{tric}
\draw [scale=0.5] (90:1)--(162:1)--(234:1)
      (90:1)..controls(18:1)..(18:2) (90:1)--(90:2) (162:1)--(162:2) (234:1)--(234:2) (234:1)..controls(306:1)..(306:2);
\end{tric}
}

\def\Skeinfc
{\begin{tric}
\draw [scale=0.5] (162:1)--(234:1)--(306:1)
      (306:1)..controls(18:1)..(18:2) (162:1)..controls(90:1)..(90:2) (162:1)--(162:2) (234:1)--(234:2) (306:1)--(306:2);
\end{tric}
}

\def\Skeinfd
{\begin{tric}
\draw [scale=0.5] (234:1)--(306:1)--(18:1)
      (18:1)--(18:2) (18:1)..controls(90:1)..(90:2) (234:1)..controls(162:1)..(162:2) (234:1)--(234:2) (306:1)--(306:2);
\end{tric}
}

\def\Skeinfe
{\begin{tric}
\draw [scale=0.5] (306:1)--(18:1)--(90:1)
      (18:1)--(18:2) (90:1)--(90:2) (90:1)..controls(162:1)..(162:2) (306:1)..controls(234:1)..(234:2) (306:1)--(306:2);
\end{tric}
}

\def\Skeinff
{\begin{tric}
\draw  [scale=0.5] (90:1)..controls(18:1)..(18:2) (90:1)--(90:2)   
      (90:1)..controls(162:1)..(162:2) (234:2)..controls(234:1)and(306:1)..(306:2);
\end{tric}
}

\def\Skeinfg
{\begin{tric}
\draw    [scale=0.5] (162:1)..controls(90:1)..(90:2)
        (162:1)--(162:2) (162:1)..controls(234:1)..(234:2) (18:2)..controls(18:1)and(306:1)..(306:2);
\end{tric}
}

\def\Skeinfh
{\begin{tric}
\draw    [scale=0.5]  (18:2)..controls(18:1)and(90:1)..(90:2) 
         (234:1)..controls(162:1)..(162:2) (234:1)--(234:2) (234:1)..controls(306:1)..(306:2);
\end{tric}
}

\def\Skeinfi
{\begin{tric}
\draw  [scale=0.5] (306:1)..controls(18:1)..(18:2)  
      (90:2)..controls(90:1)and(162:1)..(162:2)
      (306:1)..controls(234:1)..(234:2) (306:1)--(306:2);
\end{tric}
}

\def\Skeinfj
{\begin{tric}
\draw      [scale=0.5]  (18:1)--(18:2) (18:1)..controls(90:1)..(90:2)
           (234:2)..controls(234:1)and(162:1)..(162:2)  (18:1)..controls(306:1)..(306:2);
\end{tric}
}

\def\Skeinfk
{\begin{tric}
\draw      [scale=0.5]  (90:-0.7)--(90:2) 
           (306:2)--(90:-0.7)--(234:2) (162:2)..controls(162:0.5)and(18:0.5)..(18:2);
\end{tric}
}

\def\Skeinfl
{\begin{tric}
\draw      [scale=0.5]  (162:-0.7)--(162:2) 
           (18:2)--(162:-0.7)--(306:2) (234:2)..controls(234:0.5)and(90:0.5)..(90:2);
\end{tric}
}

\def\Skeinfm
{\begin{tric}
\draw      [scale=0.5]  (234:-0.7)--(234:2) 
           (90:2)--(234:-0.7)--(18:2) (306:2)..controls(306:0.5)and(162:0.5)..(162:2);
\end{tric}
}

\def\Skeinfn
{\begin{tric}
\draw      [scale=0.5]  (306:-0.7)--(306:2) 
           (162:2)--(306:-0.7)--(90:2) (18:2)..controls(18:0.5)and(234:0.5)..(234:2);
\end{tric}
}

\def\Skeinfo
{\begin{tric}
\draw      [scale=0.5]  (18:-0.7)--(18:2) 
           (234:2)--(18:-0.7)--(162:2) (90:2)..controls(90:0.5)and(306:0.5)..(306:2);
\end{tric}
}

\def\Skeing
{\begin{tric}
\draw[scale=0.8] (0,-1.5)--(0,0) (0,0)..controls(0.7,0.5)and(0.7,1.5)..(0,1.5)..controls(-0.7,1.5)and(-0.7,0.5)..(0,0);
\end{tric}
}

\def\Skeinh
{\begin{tric}
\draw[scale=0.7,double,thin] (0,0.75)--(0,-0.75);
\draw [scale=0.7] (-0.5,1.616)--(0,0.75)--(0.5,1.616) (-0.5,-1.616)--(0,-0.75)--(0.5,-1.616);
\end{tric}
}

\def\Verta
{\begin{tric}
\draw [scale=0.8] (0,0)--(90:1) (0,0)--(210:1) (0,0)--(330:1);
\draw (90:1)node[black,anchor=south,scale=0.7]{$\varpi$}
      (210:1)node[black,anchor=north,scale=0.7]{$\varpi$}
      (330:1)node[black,anchor=north,scale=0.7]{$\varpi$}; 
\end{tric}
}

\def\Vertb
{\begin{tric}
\draw (0,0)--(1,1) (0,1)--(1,0);
\draw (0,0)node[black,below,scale=0.7]{$\varpi$}
      (1,0)node[black,below,scale=0.7]{$\varpi$}
      (0,1)node[black,above,scale=0.7]{$\varpi$}
      (1,1)node[black,above,scale=0.7]{$\varpi$};
\end{tric}
}

\def\Skeini{
\begin{tric}
\draw [double,thin](0,0.7)--(0,1.5) (0,-0.7)--(0,-1.5);
\draw (0,0.7)..controls(-0.5,0.7)and(-0.5,-0.7)..(0,-0.7)  
      (0,0.7)..controls(0.5,0.7)and(0.5,-0.7)..(0,-0.7);
\end{tric}
}

\def\Skeinia{
\begin{tric}
\draw[double,thin] (0,1.5)--(0,-1.5);
\end{tric}
}

\def\Skeinj{
\begin{tric}
\draw[double,thin](0,-1.5)--(0,0);
\draw (0,0)..controls(0.7,0.5)and(0.7,1.5)..(0,1.5)..controls(-0.7,1.5)and(-0.7,0.5)..(0,0);
\end{tric}
}

\def\Skeinja{
\begin{tric}
\draw[double,thin](0,-0.5)--(0,0) (0,-1.4)--(0,-2);
\draw (0,0)..controls(0.4,0.5)and(0.4,1)..(0,1)..controls(-0.4,1)and(-0.4,0.5)..(0,0);
\draw(0,-1.4)..controls(0.4,-1.4)and(0.4,-0.5)..(0,-0.5)
     (0,-1.4)..controls(-0.4,-1.4)and(-0.4,-0.5)..(0,-0.5);
\end{tric}
}

\def\Skeinjb{
\begin{tric}
\draw[double,thin] (0,-1.4)--(0,-2);
\draw (0,0)..controls(0.4,0.5)and(0.4,1)..(0,1)..controls(-0.4,1)and(-0.4,0.5)..(0,0)  (0,-0.5)--(0,0);
\draw(0,-1.4)..controls(0.4,-1.4)and(0.4,-0.5)..(0,-0.5)
     (0,-1.4)..controls(-0.4,-1.4)and(-0.4,-0.5)..(0,-0.5);
\end{tric}
}

\def\Skeinjc{
\begin{tric}
\draw[double,thin] (0,-1.4)--(0,-2);
\draw (0,0.5)ellipse(0.3 and 0.5);
\draw(0,-1.4)..controls(0.4,-1.4)and(0.4,-0.5)..(0,-0.5)
     (0,-1.4)..controls(-0.4,-1.4)and(-0.4,-0.5)..(0,-0.5);
\end{tric}
}

\def\Skeinjd{
\begin{tric}
\draw[double,thin](0,-1.5)--(0,-0.5);
\draw (0,-0.5)..controls(0.4,0.5)and(0.4,1.5)..(0,1.5)..controls(-0.4,1.5)and(-0.4,0.5)..(0,-0.5);
\end{tric}
}

\def\Skeinje{
\begin{tric}
\draw[double,thin] (0,-1.5)--(0,-0.5);
\draw (0.4,0.5)..controls(0.4,1.5)..(0,1.5)..controls(-0.4,1.5)..(-0.4,0.5)
      (0.4,0.5)--(0,-0.5)--(-0.4,0.5) (0.4,0.5)--(-0.4,0.5);
\end{tric}
}

\def\Skeink{
\begin{tric}
\draw [double,thin](0,-0.7)--(0,-1.5);
\draw (0,0.7)--(0,1.5) 
      (0,0.7)..controls(-0.5,0.7)and(-0.5,-0.7)..(0,-0.7)  
      (0,0.7)..controls(0.5,0.7)and(0.5,-0.7)..(0,-0.7);
\end{tric}
}

\def\Skeinl
{\begin{tric}
\draw [scale=0.65] (-1,0)--(1,0)--(0,1.732)--cycle
                   (0,1.732)--(0,2.887);
\draw[double,thin,scale=0.65] (-1,0)--(-2,-0.577) (1,0)--(2,-0.577);
\end{tric}
}

\def\Skeinm
{\begin{tric}
\draw [scale=0.65] (-1,0)--(1,0)--(0,1.732)--cycle
                  (-1,0)--(-2,-0.577) (1,0)--(2,-0.577) ;
\draw[double,thin,scale=0.65] (0,1.732)--(0,2.887);
\end{tric}
}

\def\Skeinmaa
{\begin{tric}
\draw[scale=0.65]
    (0,0.577)--(-2,-0.577) (0,0.577)--(2,-0.577) ;
\draw[double,thin,scale=0.65] (0,0.577)--(0,2.887);
\end{tric}
}

\def\Skeinma
{\begin{tric}
\draw[scale=0.65]
    (0,0.577)--(-1,-0.577) (0,0.577)--(1,-0.577) ;
\draw[double,thin,scale=0.65] (0,0.577)--(0,2.887);
\end{tric}
}

\def\Skeinmb
{\begin{tric}
\draw [scale=0.65] (0,0.7)..controls(0.5,0.7)and(0.5,1.732)..(0,1.732)  (0,0.7)..controls(-0.5,0.7)and(-0.5,1.732)..(0,1.732)  
(-1,-0.5)--(0,0)--(1,-0.5);
\draw[double,thin,scale=0.65] (0,1.732)--(0,2.887) (0,0)--(0,0.7) ;
\end{tric}
}

\def\Skeinmc
{\begin{tric}
\draw [scale=0.65] (0,0.7)..controls(0.5,0.7)and(0.5,1.732)..(0,1.732)  (0,0.7)..controls(-0.5,0.7)and(-0.5,1.732)..(0,1.732)  
(-1,-0.5)--(0,0)--(1,-0.5) (0,0)--(0,0.7);
\draw[double,thin,scale=0.65] (0,1.732)--(0,2.887)  ;
\end{tric}
}

\def\Skeinmd
{\begin{tric}
\draw [scale=0.65] (0,0.7)..controls(0.5,0.7)and(0.5,1.732)..(0,1.732)  (0,0.7)..controls(-0.5,0.7)and(-0.5,1.732)..(0,1.732)  
(-1,-0.5)..controls(0,0)..(1,-0.5) ;
\draw[double,thin,scale=0.65] (0,1.732)--(0,2.887)  ;
\end{tric}
}

\def\Skeinn
{\begin{tric}
\draw[scale=0.65] (1,0)--(0,1)--(-1,0)--(0,-1)--cycle 
      (1,0)--(2,0) (0,1)--(0,2) (-1,0)--(-2,0) ;
\draw[double,thin,scale=0.65] (0,-1)--(0,-2);
\end{tric}
}

\def\Skeinna
{\begin{tric}
\draw [scale=0.65](1,0)..controls(0,1)..(0,2)    (0,-1)..controls(-1,0)..(-2,0) 
      (0,-1)--(1,0)--(2,0)   ;
\draw[double,thin,scale=0.65] (0,-1)--(0,-2);
\end{tric}
}

\def\Skeinnb
{\begin{tric}
\draw [scale=0.65](-1,0)..controls(0,1)..(0,2)    (0,-1)..controls(1,0)..(2,0) 
      (0,-1)--(-1,0)--(-2,0) ;
\draw[double,thin,scale=0.65] (0,-1)--(0,-2);
\end{tric}
}

\def\Skeinnc
{\begin{tric}
\draw[scale=0.63] (1,0)--(0,-1)--(-0.2,0.2)--cycle  
      (1,0)--(2,0)  (0,2)--(-0.7,0.7)--(-2,0) ;
\draw[double,thin,scale=0.63] (0,-1)--(0,-2)  (-0.7,0.7)--(-0.2,0.2);
\end{tric}
}

\def\Skeinnd
{\begin{tric}
\draw[scale=0.63] (1,0)--(0,-1)  (1,0)..controls(-0.2,0.2)..(0,-1)  
      (1,0)--(2,0)  (0,2)..controls(-0.7,0.7)..(-2,0) ;
\draw[double,thin,scale=0.63] (0,-1)--(0,-2) ;
\end{tric}
}

\def\Skeinne
{\begin{tric}
\draw[scale=0.63] (1,0)--(0,-1)--(-0.2,0.2)--cycle  
      (1,0)--(2,0)  (0,2)--(-0.7,0.7)--(-2,0)  (-0.7,0.7)--(-0.2,0.2) ;
\draw[double,thin,scale=0.63] (0,-1)--(0,-2)  ;
\end{tric}
}

\def\ReidemAa
{\begin{tric}
\draw (0.1,-1)..controls (-0.1,0.3) and (-0.2,0.5) .. (-0.4,0.5) 
       (-0.4,0.5)..controls(-0.5,0.5)and(-0.7,0.3)..(-0.7,0);
\draw (0.1,1)..controls  (-0.1,-0.3)and (-0.2,-0.5) .. (-0.4,-0.5) 
       (-0.4,-0.5)..controls(-0.5,-0.5)and(-0.7,-0.3)..(-0.7,0);
\end{tric}
}

\def\ReidemAb
{\begin{tric}
\draw (0,-1)--(0,1);
\end{tric}
}

\def\ReidemBa
{\begin{tric}
\draw (0,-1)..controls(1,-0.3) and (1,0.3) .. (0,1)
      (1,-1)..controls(0,-0.3) and (0,0.3) .. (1,1);
\end{tric}
}

\def\ReidemBb
{\begin{tric}
\draw (0,-1)-- (0,1)
      (1,-1)-- (1,1);
\end{tric}
}

\def\ReidemCa
{\begin{tric}
\draw[scale=0.75] (-1,-0.866)--(0.5,1.732) (1,-0.866)--(-0.5,1.732) (-1.5,0)--(1.5,0);
\end{tric}
}

\def\ReidemCb
{\begin{tric}
\draw[scale=0.75] (-1,0.866)--(0.5,-1.732) (1,0.866)--(-0.5,-1.732) (-1.5,0)--(1.5,0);
\end{tric}
}

\def\ReidemDa
{\begin{tric}
\draw      [scale=0.55]  (90:0)--(90:2) 
           (306:2)--(90:0)--(234:2) (162:2)--(18:2);
\end{tric}
}

\def\ReidemDb
{\begin{tric}
\draw       [scale=0.55]  (90:0)--(90:2) 
           (306:2)--(90:0)--(234:2) (162:2)..controls(234:1.7)and(306:1.7)..(18:2);
\end{tric}
}

\def\ReidemSa
{\begin{tric}
\draw    (0,0.2)--(0,1) (-0.5,-1)--(0,0.2)--(0.5,-1) ;
\end{tric}
}

\def\ReidemSb
{\begin{tric}
\draw   (0,0.2)--(0,1)
        (0,0.2)..controls(-0.5,-0.4)..(0.5,-1)
        (0,0.2)..controls(0.5,-0.4)..(-0.5,-1);
\end{tric}
}

\def\IHXXX
{\begin{tric}
\draw   (-0.4,0)--(0.4,0) (-0.7,-0.5)--(-0.4,0) (0.4,0)--(0.7,-0.5) 
        (-0.4,0)--(0.7,1) (0.4,0)--(-0.7,1);
\end{tric}
}

\def\FourTa
{\begin{tric}
\draw[scale=0.7]  (-50:2) arc (-50:10:2) (90:2) arc (90:130:2) (190:2) arc (190:240:2);
\draw[scale=0.7] (-30:2)--(110:2) (-10:2)--(220:2);
\draw[scale=0.7] (-25:2.1)node[below,scale=0.7]{$k$} (-10:2)node[right,scale=0.7]{$k+1$} (110:2)node[above,scale=0.7]{$f(k)$}(225:2.1)node[left,scale=0.7]{$f(k+1)$};
\filldraw[scale=0.7] (35:2) circle (1pt) (50:2) circle (1pt) (65:2) circle (1pt);
\filldraw [scale=0.7](145:2) circle (1pt) (160:2) circle (1pt) (175:2) circle (1pt);
\filldraw[scale=0.7] (260:2) circle (1pt) (275:2) circle (1pt) (290:2) circle (1pt);
\end{tric}
}

\def\FourTb
{\begin{tric}
\draw[scale=0.7]  (-50:2) arc (-50:10:2) (90:2) arc (90:130:2) (190:2) arc (190:240:2);
\draw[scale=0.7] (-30:2)--(220:2) (-10:2)--(110:2);
\draw [scale=0.7](-25:2.1)node[below,scale=0.7]{$k$} (-10:2)node[right,scale=0.7]{$k+1$} (110:2)node[above,scale=0.7]{$f(k)$}(225:2.1)node[left,scale=0.7]{$f(k+1)$};
\filldraw [scale=0.7](35:2) circle (1pt) (50:2) circle (1pt) (65:2) circle (1pt);
\filldraw [scale=0.7](145:2) circle (1pt) (160:2) circle (1pt) (175:2) circle (1pt);
\filldraw[scale=0.7] (260:2) circle (1pt) (275:2) circle (1pt) (290:2) circle (1pt);
\end{tric}
}

\def\FourTc
{\begin{tric}
\draw  (-50:2) arc (-50:10:2) (90:2) arc (90:130:2) (190:2) arc (190:240:2);
\draw (-30:1)--(220:2) (-30:1)--(110:2) (-30:1)--(-30:2); 
\draw (-25:2.1)node[below,scale=0.7]{$k$} (110:2)node[above,scale=0.7]{$f(k)-1$}(225:2.1)node[left,scale=0.7]{$f(k+1)-1$};
\filldraw (35:2) circle (1pt) (50:2) circle (1pt) (65:2) circle (1pt);
\filldraw (145:2) circle (1pt) (160:2) circle (1pt) (175:2) circle (1pt);
\filldraw (260:2) circle (1pt) (275:2) circle (1pt) (290:2) circle (1pt);
\end{tric}
}

\def\FourTd
{\begin{tric}
\draw [scale=0.7] (-50:2) arc (-50:10:2) (90:2) arc (90:150:2) (190:2) arc (190:240:2);
\draw [scale=0.7](110:2)--(-30:2)  (130:2)--(220:2) ;
\draw [scale=0.7](-25:2.1)node[below,scale=0.7]{$k$} (90:2)node[above,scale=0.7]{$f(k)-1$} 
(150:2.4)node[above,scale=0.7]{$f(k)$}  (225:2.1)node[left,scale=0.7]{$f(k+1)$};
\filldraw[scale=0.7] (35:2) circle (1pt) (50:2) circle (1pt) (65:2) circle (1pt);
\filldraw[scale=0.7](160:2) circle (1pt) (170:2) circle (1pt) (180:2) circle (1pt);
\filldraw[scale=0.7](260:2) circle (1pt) (275:2) circle (1pt) (290:2) circle (1pt);
\end{tric}
}

\def\FourTe
{\begin{tric}
\draw [scale=0.7] (-50:2) arc (-50:10:2) (90:2) arc (90:150:2) (190:2) arc (190:240:2);
\draw [scale=0.7](110:2)--(220:2)  (130:2)--(-30:2) ;
\draw [scale=0.7](-25:2.1)node[below,scale=0.7]{$k$} (90:2)node[above,scale=0.7]{$f(k)-1$} 
(150:2.4)node[above,scale=0.7]{$f(k)$}  (225:2.1)node[left,scale=0.7]{$f(k+1)$};
\filldraw [scale=0.7](35:2) circle (1pt) (50:2) circle (1pt) (65:2) circle (1pt);
\filldraw [scale=0.7](160:2) circle (1pt) (170:2) circle (1pt) (180:2) circle (1pt);
\filldraw [scale=0.7](260:2) circle (1pt) (275:2) circle (1pt) (290:2) circle (1pt);
\end{tric}
}

\def\FourTf
{\begin{tric}
\draw  (-50:2) arc (-50:10:2) (90:2) arc (90:130:2) (190:2) arc (190:260:2);
\draw (240:2)--(-30:2) (220:2)--(110:2); 
\draw (-25:2.1)node[below,scale=0.7]{$k$} (110:2)node[above,scale=0.7]{$f(k)-1$}(225:2.1)node[left,scale=0.7]{$f(k+1)-1$} (250:2.1)node[left,scale=0.7]{$f(k+1)$};
\filldraw (35:2) circle (1pt) (50:2) circle (1pt) (65:2) circle (1pt);
\filldraw (145:2) circle (1pt) (160:2) circle (1pt) (175:2) circle (1pt);
\filldraw (270:2) circle (1pt) (285:2) circle (1pt) (300:2) circle (1pt);
\end{tric}
}

\def\FourTg
{\begin{tric}
\draw  (-50:2) arc (-50:10:2) (90:2) arc (90:130:2) (190:2) arc (190:260:2);
\draw (240:2)--(110:2) (220:2)--(-30:2); 
\draw (-25:2.1)node[below,scale=0.7]{$k$} (110:2)node[above,scale=0.7]{$f(k)-1$}(225:2.1)node[left,scale=0.7]{$f(k+1)-1$} (250:2.1)node[left,scale=0.7]{$f(k+1)$};
\filldraw (35:2) circle (1pt) (50:2) circle (1pt) (65:2) circle (1pt);
\filldraw (145:2) circle (1pt) (160:2) circle (1pt) (175:2) circle (1pt);
\filldraw (270:2) circle (1pt) (285:2) circle (1pt) (300:2) circle (1pt);
\end{tric}
}

\def\Vertc
{\begin{tric}
\draw [scale=0.8] (0,0)--(90:1) (0,0)--(210:1) (0,0)--(330:1);
\draw (90:1)node[black,anchor=south,scale=0.7]{$k+1$}
      (210:1)node[black,anchor=north,scale=0.7]{$1$}
      (330:1)node[black,anchor=north,scale=0.7]{$k$}; 
\end{tric}
}

\def\Vertd
{\begin{tric}
\draw [scale=0.8] (0,0)--(90:1) (0,0)--(210:1) (0,0)--(330:1);
\draw (90:1)node[black,anchor=south,scale=0.7]{$k+1$}
      (210:1)node[black,anchor=north,scale=0.7]{$k$}
      (330:1)node[black,anchor=north,scale=0.7]{$1$}; 
\end{tric}
}

\def\Skeinak
{\begin{tric}
\draw (0.7,0) circle (0.7);
\draw (0,0)node[left,black,scale=0.7]{$k$}; 
\end{tric}
}

\def\Skeinaone
{\begin{tric}
\draw (0.7,0) circle (0.7);
\draw (0.7,0.7)node[above,black,scale=0.7]{$1$}; 
\end{tric}
}

\def\Skeingg
{\begin{tric}
\draw[scale=0.8] (0,-1)--(0,0) (0,0)..controls(0.7,0.5)and(0.7,1.5)..(0,1.5)..controls(-0.7,1.5)and(-0.7,0.5)..(0,0);
\draw (0,-0.8)node[below,black,scale=0.7]{$2$};
\draw (0,1.2)node[above,black,scale=0.7]{$1$};
\end{tric}
}

\def\Skeinbigon
{\begin{tric}
\draw[scale= 0.8] (0,0.7)--(0,1.5) (0,-0.7)--(0,-1.5)
        (0,1.5)node[above,black,scale=0.7]{$k$}
        (0,-1.5)node[below,black,scale=0.7]{$k$}
      (0,0.7)..controls(-0.5,0.7)and(-0.5,-0.7)..(0,-0.7) node[left,midway,black,scale=0.7]{$1$} (0,0.7)..controls(0.5,0.7)and(0.5,-0.7)..(0,-0.7)
      node[right,midway,black,scale=0.7]{$k-1$} ;
\end{tric}
}

\def\Skeinbigona
{\begin{tric}
\draw [scale=0.8] (0,1.5)--(0,-1.5)  node[below,black,scale=0.7]{$k$} ;
\end{tric}
}

\def\SkeinIH
{\begin{tric}
\draw[scale=0.7] (-1.7,-1.732)--(-1,0)--(-1.7,1.732) (1.7,1.732)--(1,0)--(1.7,-1.732) 
(-1.7,-1.732)node[below,black,scale=0.7]{$k$}
(-1.7,1.732)node[above,black,scale=0.7]{$1$}
(1.7,1.732)node[above,black,scale=0.7]{$1$}
(1.7,-1.732) node[below,black,scale=0.7]{$k$}
(-1,0)--(1,0)node[below,midway,black,scale=0.7]{$k+1$};
\end{tric}
}

\def\GrSkeinIH
{\begin{tric}
\draw[darkgreen,scale=0.7] (-2,-1.732)--(-1,0)--(-2,1.732) (2,1.732)--(1,0)--(2,-1.732) 
(-2,-1.732)node[below,black,scale=0.7]{$k$}
(-2,1.732)node[above,black,scale=0.7]{$1$}
(2,1.732)node[above,black,scale=0.7]{$1$}
(2,-1.732) node[below,black,scale=0.7]{$k$}
(-1,0)--(1,0)node[below,midway,black,scale=0.7]{$k+1$};
\end{tric}
}

\def\SkeinIHa
{\begin{tric}
\draw [scale=0.7]  (-1.74,2)--(0,1)--(1.74,2) 
(1.74,-3)--(0,0) node[right,pos=0.7,black,scale=0.7]{$1$}
(0,0)--(-1.74,-3) node[left,pos=0.3,black,scale=0.7]{$1$}
(0,0)--(0,1) node[left,midway,black,scale=0.7]{$2$}
(-1.74,-3) node[below,black,scale=0.7]{$k$}
(1.74,-3) node[below,black,scale=0.7]{$k$}
(-1.74,2) node[above,black,scale=0.7]{$1$}
(1.74,2) node[above,black,scale=0.7]{$1$}
(1.16,-2)--(-1.16,-2)node[below,midway,black,scale=0.7]{$k-1$};
\end{tric}
}

\def\GrSkeinIHa
{\begin{tric}
\draw [darkgreen, scale=0.7]  (-1.74,2)--(0,1)--(1.74,2) 
(1.74,-3)--(0,0) node[right,pos=0.7,black,scale=0.7]{$1$}
(0,0)--(-1.74,-3) node[left,pos=0.3,black,scale=0.7]{$1$}
(0,0)--(0,1) node[left,midway,black,scale=0.7]{$2$}
(-1.74,-3) node[below,black,scale=0.7]{$k$}
(1.74,-3) node[below,black,scale=0.7]{$k$}
(-1.74,2) node[above,black,scale=0.7]{$1$}
(1.74,2) node[above,black,scale=0.7]{$1$}
(1.16,-2)--(-1.16,-2)node[below,midway,black,scale=0.7]{$k-1$};
\end{tric}
}

\def\SkeinIHb
{\begin{tric}
\draw[scale=0.7] (-1.7,-1.8)--(-1,0)--(-1.7,1.8) (1.7,1.8)--(1,0)--(1.7,-1.8) 
(-1.7,-1.7)node[below,black,scale=0.7]{$k$}
(-1.7,1.732)node[above,black,scale=0.7]{$1$}
(1.7,1.732)node[above,black,scale=0.7]{$1$}
(1.7,-1.732) node[below,black,scale=0.7]{$k$}
(-1,0)--(1,0)node[below,midway,black,scale=0.7]{$k-1$};
\end{tric}
}

\def\GrSkeinIHb
{\begin{tric}
\draw[darkgreen,scale=0.7] (-2,-1.732)--(-1,0)--(-2,1.732) (2,1.732)--(1,0)--(2,-1.732) 
(-2,-1.732)node[below,black,scale=0.7]{$k$}
(-2,1.732)node[above,black,scale=0.7]{$1$}
(2,1.732)node[above,black,scale=0.7]{$1$}
(2,-1.732) node[below,black,scale=0.7]{$k$}
(-1,0)--(1,0)node[below,midway,black,scale=0.7]{$k-1$};
\end{tric}
}

\def\SkeinIHc
{\begin{tric}
\draw  [scale=0.7] 
(-1.7,2)..controls(-1,.5)and(1,.5)..(1.7,2) node[above,black,midway,scale=0.7]{$1$}
(1.7,-2)..controls(1,-.5)and(-1,-.5)..(-1.7,-2) node[below,black,midway,scale=0.7]{$k$};
\end{tric}
}

\def\GrSkeinIHc
{\begin{tric}
\draw  [darkgreen,scale=0.7] 
(-2,2)..controls(-1,1)and(1,1)..(2,2) node[above,black,midway,scale=0.7]{$1$}
(2,-2)..controls(1,-1)and(-1,-1)..(-2,-2) node[below,black,midway,scale=0.7]{$k$};
\end{tric}
}

\def\Skeinassocia
{\begin{tric}
\draw[scale=0.8] (0,0)..controls(0,0.5)and(0.2,0.7)..(0.5,1)
      (1,0)..controls(1,0.5)and(0.8,0.7)..(0.5,1)
      (0.5,1)..controls(0.5,1.5)and(0.3,1.7)..(0,2) 
              node[right,black,midway, scale=0.7]{$k+1$}
      (-1,0)..controls(-1,1)and(-0.6,1.5)..(0,2)
      (0,2)--(0,3)
      (0,0)node[below,black,scale=0.7]{$k$}
      (1,0)node[below,black,scale=0.7]{$1$}
      (-1,0)node[below,black,scale=0.7]{$1$}
      (0,3)node[above,black,scale=0.7]{$k+2$};
\end{tric}
}

\def\Skeinassoci
{\begin{tric}
\draw[scale=0.8] (0,0)..controls(0,0.5)and(-0.2,0.7)..(-0.5,1)
      (-1,0)..controls(-1,0.5)and(-0.8,0.7)..(-0.5,1)
      (-0.5,1)..controls(-0.5,1.5)and(-0.3,1.7)..(0,2) 
              node[left,black,midway, scale=0.7]{$k+1$}
      (1,0)..controls(1,1)and(0.6,1.5)..(0,2)
      (0,2)--(0,3)
      (0,0)node[below,black,scale=0.7]{$k$}
      (-1,0)node[below,black,scale=0.7]{$1$}
      (1,0)node[below,black,scale=0.7]{$1$}
      (0,3)node[above,black,scale=0.7]{$k+2$};
\end{tric}
}

\begin{notation}
Let $\K:=\mathbb{C}(q)$ and $\A:=\mathbb{C}[q]_{(q-1)\mathbb{C}[q]}\subset \K$.  Note that $\mathbb{C} \cong \A/(q-1)\A$. 
\end{notation}

\begin{defn}\label{Qinteger}
Define the quantum integer $[n]_v:= \dfrac{v^n-v^{-n}}{v-v^{-1}}$. 
Denote $ [n]_v ! :=  [1]_v [2]_v [3]_v ... [n]_v $. Denote
$ \begin{bmatrix}
n \\
k 
\end{bmatrix} _{v} := \cfrac{[n]_v!}{[k]_v![n-k]_v!} $. When $v=q$, write $[n]:= [n]_q$. 
\end{defn}

\begin{remark}
If $v$ is a power of $q$, in particular if $v=q$ or $q^2$, then quantum integers and quantum binomials lie in $\A$. Therefore we can consider their image in $\kk$ or $\K$. We do not make new notation for this, but instead leave it up to context whether a particular expression involving quantum integers is in $\K, \A$, or $\kk$. 
\end{remark}

\begin{defn}\label{DefofWeb}
Let $\R\in \lbrace \K, \A, \kk\rbrace$. Define the pivotal $\R$-linear category $\Web$ whose objects are generated by the self-dual objects $ n\in \mathbb{Z}_{\ge 0}$, and whose morphisms are generated by the following trivalent vertices:

 {\centering{\begin{align*}
     \Vertc \in {\Hom}_{\Web}(1\otimes k, k+1) \ \text{and}  \
     \Vertd \in {\Hom}_{\Web}(k \otimes 1, k+1) 
 \end{align*}}}
for $k \in \mathbb{Z}_{\ge 0}$, modulo the tensor ideal generated by $\id_k$ for $k>m$, and the following relations.

\begin{align} 
\label{defskein}
&\text{(\ref{defskein}a)} \Skeinaone = \frac{[2m-4][m]}{[m-2][2]},  \ 
\text{(\ref{defskein}b)}\Skeingg= 0  ,\  
\text{(\ref{defskein}c)}\Skeinbigon  = \frac{[2k]}{[2]}  \Skeinbigona   , \  
\text{(\ref{defskein}d)}\Skeinassoci =\Skeinassocia 
\\
& \text{(\ref{defskein}e)} \SkeinIH =  \SkeinIHa   -\frac{[2m-4k-4][m-2k]}{[m-2k-2][2m-4k]}  \SkeinIHb
  +  \frac{[2m-4k-4][m-2]}{[m-2k-2][2m-4]} \SkeinIHc  \notag 
\end{align}
\end{defn}

\def\VanTri
{\begin{tric}
\draw [scale=0.7]  (-1.74,2)--(0,1)--(1.74,2) 
(1.74,-3)--(0,0) node[right,pos=0.7,black,scale=0.7]{$1$}
(0,0)--(-1.74,-3) node[left,pos=0.3,black,scale=0.7]{$1$}
(0,0)--(0,1) node[left,midway,black,scale=0.7]{$2$}
(-1.74,-3) node[below,black,scale=0.7]{$m$}
(1.74,-3) node[below,black,scale=0.7]{$m$}
(-1.74,2) node[above,black,scale=0.7]{$1$}
(1.74,2) node[above,black,scale=0.7]{$1$}
(1.16,-2)--(-1.16,-2)node[below,midway,black,scale=0.7]{$m-1$};
\end{tric}
}

\def\VanTria
{\begin{tric}
\draw[scale=0.7] (-2,-1.732)--(-1,0)--(-2,1.732) (2,1.732)--(1,0)--(2,-1.732) 
(-2,-1.732)node[below,black,scale=0.7]{$m$}
(-2,1.732)node[above,black,scale=0.7]{$1$}
(2,1.732)node[above,black,scale=0.7]{$1$}
(2,-1.732) node[below,black,scale=0.7]{$m$}
(-1,0)--(1,0)node[below,midway,black,scale=0.7]{$m-1$};
\end{tric}
}

\def\VanTrib
{\begin{tric}
\draw  [scale=0.7] 
(-2,2)..controls(-1,1)and(1,1)..(2,2) node[above,black,midway,scale=0.7]{$1$}
(2,-2)..controls(1,-1)and(-1,-1)..(-2,-2) node[below,black,midway,scale=0.7]{$m$};
\end{tric}
}

\begin{remark}
We will use the convention that strands labelled zero can be erased and strands labelled $k<0$ are equal to zero.    
\end{remark}

\begin{remark}
The presentation of $\Web$ we give in Definition \ref{DefofWeb} is practically the same as the presentation of $\textbf{Web}(\mathfrak{sp}_{2n})$ in \cite[Definition 1.1]{bodish2021type}, but with different coefficients.
\end{remark}

\begin{remark}\label{R:coeffs-simple-q=1-1}
When $\R=\kk$, i.e. $q=1$, the coefficients in \RL ({\ref{defskein}e}) are all $\pm 1$.
\end{remark}

There are also simplifications in the representation category when $q=1$. For example, the braiding\footnote{This can be defined as the map $P\circ f\circ \Theta$ \cite[Theorem 7.8]{JantzenQgps}, where $P$ is the tensor flip map, $f$ is as in \cite[Section 7.9]{JantzenQgps}, and $\Theta$ is the quasi-$R$ matrix \cite[Section 7.2]{JantzenQgps}.} becomes symmetric, meaning it is equal to its own inverse. In the symmetric case, there is a standard definition of the exterior power of a representation. But in the braided case things become more complicated \cite{BZ-braidedalgebras}. In Section \ref{ssec:quantum-exterior-algebra} we carefully define the $q$-analogue of the exterior powers of the defining representation. Write $\Lambda_{\kk}^k$ for the usual $k$-th exterior power of the defining representation and write $\Lambda_{\K}^k$ to denote the $q$-analogue. For $\R\in \{\K, \kk\}$, the monoidal category $\Fund$ is defined to be the full monoidal subcategory of $U_\R(\om)$-mod generated by $\Lambda_\R^k$ for $k=0, \ldots, m$.

The main theorem of this article is the following.

\begin{thm}\label{T:main-thm}
Let $\R\in \{\kk, \K\}$. There is an equivalence of $\R$-linear pivotal categories
\[
\Phi_\R:\Web \rightarrow \Fund
\]
sending $k$ to $\Lambda_{\R}^k$, for $0\le k \le m$.
\end{thm}

\begin{remark}
Instead of working with $O(m)$, we use $U_{\kk}(\om)$\footnote{This is just notation for an associative algebra which acts like the enveloping algebra of a non-existent $\om$, and should not be taken literally. The Lie algebra of $O(m)$ is simply $\som$.}, which is a $\mathbb{Z}/2$ extension of the universal enveloping algebra of $\som$. Every $O(m)$ representation can be made into a module for $U_{\kk}(\om)$ and vice-versa. Moreover, a $\kk$-linear map between such representations is an $O(m)$ intertwinter if and only if it is a $U_{\kk}(\om)$ intertwiner. In particular, the category $\Fundk$ is isomorphic to the full monoidal subcategory of $\textbf{Rep}(O(m))$ generated by the exterior powers of $\kk^m$. We choose to use $U_{\kk}(\om)$, instead of $O(m)$, since it is easier to see how to relate its representations to those of $U_{\K}(\om)$. 
\end{remark}

\begin{remark}\label{R:Kar}
Let $\R\in \{\kk, \K\}$. Given an $\R$-linear monoidal category $\mathcal{C}$ we can build an additive monoidal category, denoted $\text{Add}(\mathcal{C})$, where the objects are formal direct sums of objects in the original category and morphisms are matrices of morphisms in the original category. Given an $\R$-linear additive monoidal category $\mathcal{A}$, we can build an additive monoidal category, called the Karoubi envelope of $\mathcal{A}$ and denoted $\text{Kar}(\mathcal{A})$, which is closed under taking direct summands. Objects in $\text{Kar}(\mathcal{A})$ are pairs $(X,e)$, where $X$ is an object in $\mathcal{A}$ and $e\in \End_{\mathcal{A}}(X)$ is an idempotent. Morphisms in $\text{Kar}(\mathcal{A})$ are defined by
\[
\Hom_{\text{Kar}(\mathcal{A})}((X,e), (Y,f)):=f\Hom_{\mathcal{A}}(X,Y)e.
\]

Finite dimensional representations of $U_{\R}(\om)$ are completely reducible. Moreover, every finite dimensional irreducible type-$\textbf{1}$ representation of $U_{\R}(\om)$) is a direct summand of some tensor product of $\Lambda_{\R}^k$'s. Thus, we have an equivalence of $\R$-linear additive monoidal categories
\[
\text{Kar}(\text{Add}(\Fund))\cong \textbf{Rep}(U_{\R}(\om)),
\]
where by $\textbf{Rep}(U_{\R}(\om))$ we mean the category of finite dimensional type-$\one$ representations of $U_{\R}(\om)$. Since $\text{Add}$ and $\text{Kar}$ are universal constructions, we can interpret Theorem \ref{T:main-thm} as a presentation of the monoidal category $\textbf{Rep}(U_{\R}(\om))$. 
\end{remark}

\begin{example}\label{E:m=1}
Let $m=1$. Consider the one dimensional vector space $V_{\kk}$, spanned by basis vector $v$ and equipped with symmetric form $(v,v)=1$. We have $O(1)\cong \mathbb{Z}/2$, where the generator $\sigma\in\mathbb{Z}/2$ acts as $\sigma(v) = -v$. Since $\Lambda^k(V_{\kk})=0$ for $k>1$, we see that $\textbf{Fund}(U_{\kk}(\mathfrak{o}_1))$ is the monoidal category generated by $V_{\kk}$. Write $\kk$ to denote the trivial $O(1)$-module. It is immediate that 
\[
V_{\kk}^{\otimes d}\cong \begin{cases} V_{\kk} \quad \text{if $d$ is odd} \\
\kk \quad \text{if $d$ is even}.
\end{cases},
\]
and the isomorphisms are induced by $v^{\otimes d}\mapsto v$ and $v^{\otimes d}\mapsto 1$ respectively. 

Since $m=1$, we set $\id_{k}=0$ in $\WebC$ for all $k>1$. One easily checks that the only defining relations which are not of the form $0=0$ are \RL (\ref{defskein}a), which says that the circle labelled $1$ evaluates to $1\in \kk$, and \RL (\ref{defskein}e) which says that the identity of $1\otimes 1$ is equal to the cup-cap.

The $\kk$-linear version of the $n=2$ case of \cite[Exercise 4.13(i)]{KS-sheaves-and-cats} says that if $\mathcal{C}$ is an $\kk$-linear additive monoidal category which is closed under taking direct summands, then monoidal functors $\textbf{Rep}(\kk[\mathbb{Z}/2])\rightarrow \mathcal{C}$ are in bijection with objects $X\in \mathcal{C}$, equipped with an isomorphism $\alpha: X\otimes X\rightarrow \textbf{1}_{\mathcal{C}}$, such that $\alpha \otimes \id = \id \otimes \alpha$. We leave it as an exercise to prove the $m=1$, $\R= \kk$, case of Theorem \ref{T:main-thm} by hand, and then use Remark \ref{R:Kar} to deduce the universal property described above. Hint: for the deduction step, use the pivotal structure on $\WebC$ to rewrite the identity equals cup-cap relation so it resembles $\alpha\otimes \id = \id\otimes \alpha$.
\end{example}

\begin{remark}
What about the $m=1$ case over $\K$? It is easy to verify that $\K\otimes \textbf{Web}_{\kk}(O(1))\cong \textbf{Web}_{\K}(O(1))$. Since $SO(1)\subset O(1)$ is the trivial group, $\mathfrak{so}_1$ is the trivial Lie algebra. We are led to define $U_{\K}(\mathfrak{so}_1)\footnote{Since $U_q(\mathfrak{g})$ is generated by $E_{\alpha}, F_{\alpha}$, and $K_{\alpha}^{\pm 1}$ for $\alpha\in \Delta$, and since there are no simple roots for $\mathfrak{so}_1$, we think of $U_{\K}(\mathfrak{so}_1)$ as the $\K$-algebra with no generators.} = \K$ and $U_{\K}(\mathfrak{o}_1)=\K[\mathbb{Z}/2]$. The same analysis in Example \ref{E:m=1} works here to show that there is a monoidal equivalence $\textbf{Web}_{\K}(O(1))\cong\textbf{Fund}(U_{\K}(\mathfrak{o}_1))$.
\end{remark}

\begin{example}
Let $m=2$. In this case, the Lie algebra of $SO(2)\cong S^1$ is abelian and therefore is not semisimple. In particular, its enveloping algebra does not have a Serre presentation which can be $q$-deformed as usual. We take the following approach. We \emph{define} $U_{\K}(\mathfrak{so}_2):=\K[K^{\pm 1}]$ and $V_{\K}:=\K\cdot a_1\oplus \K\cdot b_1$, with $K\cdot a_1= q^2a_1$ and $K\cdot b_1= q^{-2}b_1$. There is then an involutive algebra automorphism of $U_{\K}(\mathfrak{so}_2)$, denoted $\sigma$, which acts by $\sigma(K) = K^{-1}$. Defining $\sigma(a_1) = b_1$ and $\sigma(b_1)= a_1$ induces an action on $V_{\K}$ by the algebra $U_{\K}(\mathfrak{o}_2):=\K[K^{\pm 1}]\langle \sigma \ | \ \sigma^2=1, \quad \sigma K^{\pm 1} \sigma = \sigma(K^{\pm 1})\rangle$. The $U_{\K}(\mathfrak{o}_2)$ module $\Lambda_{\K}^2$, which is spanned by $a_1b_1$, has $K^{\pm 1}$ in its kernel and $\sigma$ acts as $-1$. The category $\textbf{Fund}(U_{\K}(\mathfrak{o}_2))$ is a $q$-analogue of the category of representations of $O(\kk^2)$, generated by $\kk^2$ and $\det$. Similar to when $m=1$, we have an equivalence of pivotal $\K$-linear categories $\textbf{Web}_{\K}(O(2))\cong \K\otimes \textbf{Web}_{\kk}(O(2))$. However, the braiding on $\textbf{Web}_{\K}(O(2))$ is non-trivial, so this is not an equivalence of braided categories.
\end{example}

\begin{remark}
The $m=1,2$ cases of our main theorem are somewhat hidden in the body of this paper. We make a few comments along the way for how things change, but for the sake of readability we mostly explain things when $m\ge 3$, so $\som$ is semisimple and therefore we have a more uniform notation. Regardless, the main results still hold for $m=1, 2$, and the careful reader will be able to see what needs to be changed.
\end{remark}

\subsection{Idea of the proof}

Citing classical results about invariant theory, Lehrer-Zhang prove \cite[Theorem 4.8]{LZbrauercat} that there is an equivalence between the Brauer category \cite[Definition 2.4]{LZbrauercat}, modulo the antisymmetrizing idempotent on $m+1$ strands, and the full monoidal subcategory of $\textbf{Rep}(O(\kk^m))$ generated by $\kk^m$. Since $\Lambda_{\kk}^k$ is a direct summand of $(\kk^m)^{\otimes k}$, for $k=0, 1, \dots, m$, one might hope to reduce the proof of our main theorem, when $\R= \kk$, to a calculation verifying that the antisymmetrizing idempotent $\frac{1}{(m+1)!}\sum_{w\in S_{m+1}}(-1)^{\ell(w)}w$ is zero in $\WebC$. We carry out this calculation in Proposition \ref{P:verification-of-classicalkernel}.


The idea of the proof of Theorem \ref{T:main-thm} is roughly as follows. First, prove the result for $\R= \kk$, using the ideas outlined above. Second, prove that $\Phi_{\K}$ is full, using that if $\R=\K$, then the braiding endomorphism for the tensor square of the vector representation generates the endomorphism rings of arbitrary tensor powers of the vector representation \cite[Theorem 8.5]{LZ-stronglymultifree}. Finally, we carefully argue that everything we defined actually makes sense over the ring $\A$. The $\A$ versions of our categories and functors can then be specialized to $\kk$ or $\K$. Since $\A$ is a local ring and a principal ideal domain, basic facts about finitely generated modules over a PID allow us to deduce our functor is faithful when $\R=\K$ from knowing it is full over $\K$ and an equivalence over $\kk$. 

Let us comment on why we do not just prove Theorem \ref{T:main-thm} directly for $\K$ the same way we do for $\kk$. There is a $q$-analogue of Lehrer-Zhang's result \cite[Theorem 8.2]{LZbrauercat}, in which the Brauer category is replaced with the BMW category. However, just as the definition of the $q$-analogue of the exterior powers is not trivial, it is not so easy to explicitly describe the $q$-analogue of the antisymmetrizer in the BMW category. Lehrer-Zhang only discuss it abstractly \cite[Theorem 8.2(iii)]{LZbrauercat}, using the theory of cellular algebras. The abstract description is sufficient to prove their result, but several years earlier Tuba-Wenzl gave a recursive formula for this idempotent by relating it to the $q$-antisymmetrizer in the Hecke algebra \cite[Equation 7.12]{TubaWenzl}. There is also work on explicitly describing the $q$-antisymmetrizer in the BMW category: when $m=3$ in \cite[Equation 7.8]{LZ-Tem-Lie-Analogue}, and for all $m\ge 1$ in \cite{HZ-idempotents, IMO-idemoptents, DHS-antisymmetrizer}. These descriptions are not very easy to compute with. An instance of this is that we have not yet found how to use the relations in $\WebK$ to show the $q$-antisymmetrizer on $m+1$ strands is zero in the web category for $O(m)$, even though this is implied by Theorem \ref{T:main-thm}.

\subsection{Future work}

\subsubsection{Type $B/D$ webs}
Given a representation of $U_{\R}(\om)$, we can restrict to $U_{\R}(\som)$, forgetting the $\mathbb{Z}/2$ action. Thus, we always have a faithful 
monoidal functor $\textbf{Res}:\Fund\rightarrow \textbf{Rep}(U_{\R}(\mathfrak{so}_m))$. The image of $\textbf{Res}$ will not contain all fundamental representations for $U_{\R}(\mathfrak{so}_m)$, since it misses\footnote{See Lemma \ref{L:samecharacter} for a precise description of the objects in the image} the spin representations. The restriction functor is not full either, since $\Lambda_{\R}^i$ and $\Lambda_{\R}^{m-i}$ become isomorphic upon restriction to $U_{\R}(\som)$.

A trivial way to remedy the failure of essential surjectivity is to define $\textbf{Fund}_{\R}(SO(m))$ to be the full monoidal subcategory of $\textbf{Rep}(U_{\R}(\som))$ generated by the exterior powers $\Lambda^i_{\R}$. Now we have a faithful and essentially surjective monoidal functor $\textbf{Res}:\textbf{Fund}(U_{\R}(\om))\rightarrow \textbf{Fund}_{\R}(SO(m))$.

When $m=2n+1$, restriction is full on the monoidal category generated by $V_{\R}$ \cite[5.1.3]{LZ-stronglymultifree}. Thus, one might try to find a presentation of $\textbf{Fund}_{\R}(SO(m))$ by modifying $\Web$ to keep track of the isomorphisms $\Lambda_{\K}^i\cong \Lambda_{\K}^{m-i}$. This is analogous to the relation between $\mathfrak{gl}_m$ webs and $\mathfrak{sl}_m$ webs, where the modification is to introduce \emph{tags} which encode the isomorphism $\Lambda_{\K}^m\rightarrow \Lambda^0_{\K}$ \cite[Section 2.1]{CKM}. 

If $m=2n$, restriction is not full on the monoidal subcategory generated by $V_{\R}$, and one needs to include more generating web morphisms, similar to how the Brauer algebra is ``enhanced" by the element $\Delta_m$ in \cite[Definition 5.1]{LZ-enhanced-brauer}. In fact, this element appears to be the composition of the multiplication map $V_{\R}^{\otimes m}\rightarrow \Lambda_{\R}^m$ with the isomorphism $\Lambda_{\R}^m\cong \Lambda_{\R}^0$. In which case, a version of $\Web$ that includes tags might also describe $\textbf{Fund}_{\R}(SO(m))$.

\subsubsection{Towards a confluent presentation}

Let $\R\in \{\kk, \K\}$. In the definition of $\Web$, we give a set of generators and relations which is informally ``small". For example, we leave out several basic relations which appear in Section \ref{SS:further-renls}, where we derive them from the defining relations. The benefit of having fewer defining relations is that our main theorem gives us a simpler ``universal mapping property" for monoidal functors out of $\Fund$. The downside is that we do not have a confluent presentation \cite[Section 3.5]{west-sikora-conflucence}. This makes graphical calculation in $\Web$ more difficult. In particular, we do not provide a basis for $\Web$ in the present work.

A good first step towards a confluent presentation is to analyze the relations, analogous to the defining relations, which are satisfied by the trivalent vertices with arbitrary labels. Note that the generating trivalent vertices of $\Web$ all have an edge labelled $1$. Then, with this new presentation in hand, one can work to describe a cellular basis for $\Web$, which is adapted to the monoidal structure, analogous to the double-ladder basis for $\mathfrak{sl}_n$ webs defined by Elias in \cite{elias2015light}. This could give a way to prove that a $\mathbb{Z}[q,q^{-1}]$ version of $\Web$ is related to tilting modules, after specializing to a field, since Elias used the double ladders for $\mathfrak{sl}_n$ webs to prove an analogous result.

\subsubsection{Howe duality}

Since many defining relations in $\mathfrak{sl}_m$ webs are reflected in the Serre presentation of the Howe dual $U_q(\mathfrak{gl}_n)$ \cite[Proposition 5.2.1]{CKM}, a parallel approach to finding more interesting relations for $\WebK$ is to try to find a version of Howe duality adapted to our situation. As we mentioned above, the orthogonal Howe duality of \cite{ST-BCDHowe} describes intertwiners for representations of the $\iota$-quantum group $U_q'(\mathfrak{o}_m)$, in terms of the Howe dual Drinfeld-jimbo quantum group $U_q(\mathfrak{so}_{2n})$. A more precise version of this duality should only describe intertwiners for a $\mathbb{Z}/2$ extension of $U_q'(\som)$, see \cite[Remark 1.2]{ST-BCDHowe}. The category $\WebK$ describes intertwiners for $U_{\K}(\om)$, so one might guess that there is a Howe dual $\iota$-quantum group $U'_{\K}(\mathfrak{so}_{2n})$, or a $\mathbb{Z}/2$ extension thereof, acting on the $U_{\K}(\om)$ module $(\Lambda_{\K}^{\bullet})^{\otimes n}$.

\subsubsection{Basis from pattern avoidance}

The full subcategory of $(U_{\K}(\mathfrak{sp}_{2n}))$-modules
monoidally generated by the the vector representation $V_{\K}$ has an explicit basis, which is described in terms of webs in \cite[Theorem 5.35]{bodish2021type}. The first ingredient in this basis is an endomorphism of $V_{\K}\otimes V_{\K}$ which is invariant under $90^{\circ}$ rotation \cite[Section 4.3]{bodish2021type}. This diagram is drawn in the web category as a quadrivalent vertex only because it is invariant under $90^{\circ}$ rotation. It does not satisfy the braid relations \cite[Equation 4.1]{bodish2021type}. The analogue of this endomorphism in $\WebK$ is the following. 
\def\QuaVertex
{\begin{tric}
\draw  (-1.4,1.4)--(1.4,-1.4) ;
\draw  (1.4,1.4)--(-1.4,-1.4) ;
\draw (-1.4,1.4)node[above,black,scale=0.7]{$1$}
      (1.4,1.4)node[above,black,scale=0.7]{$1$}
      (-1.4,-1.4)node[below,black,scale=0.7]{$1$}
      (1.4,-1.4)node[below,black,scale=0.7]{$1$};
\filldraw (0,0) circle (3pt) ; 
\end{tric}
}

\def\QuaVertexa
{\begin{tric}
\draw[scale=0.7] (-1.732,-2)--(0,-1)--(1.732,-2) (1.732,2)--(0,1)--(-1.732,2) 
(-1.732,-2)node[below,black,scale=0.7]{$1$}
(1.732,-2)node[below,black,scale=0.7]{$1$}
(1.732,2)node[above,black,scale=0.7]{$1$}
(-1.732,2) node[above,black,scale=0.7]{$1$}
(0,-1)--(0,1)node[left,midway,black,scale=0.7]{$2$};
\end{tric}
}

\def\QuaVertexb
{\begin{tric}
\draw  [scale=0.7] 
(-2,2)..controls(-1,1)and(1,1)..(2,2) node[above,black,midway,scale=0.7]{$1$}
(2,-2)..controls(1,-1)and(-1,-1)..(-2,-2) node[below,black,midway,scale=0.7]{$1$};
\end{tric}
}

\def\QuaVertexc
{\begin{tric}
\draw[scale=0.7] (-2,-1.732)--(-1,0)--(-2,1.732) (2,1.732)--(1,0)--(2,-1.732) 
(-2,-1.732)node[below,black,scale=0.7]{$1$}
(-2,1.732)node[above,black,scale=0.7]{$1$}
(2,1.732)node[above,black,scale=0.7]{$1$}
(2,-1.732) node[below,black,scale=0.7]{$1$}
(-1,0)--(1,0)node[below,midway,black,scale=0.7]{$2$};
\end{tric}
}

\def\QuaVertexd
{\begin{tric}
\draw  [scale=0.7] 
(-2,2)..controls(-1,1)and(-1,-1)..(-2,-2) node[left,black,midway,scale=0.7]{$1$}
(2,-2)..controls(1,-1)and(1,1)..(2,2) node[right,black,midway,scale=0.7]{$1$};
\end{tric}
}

\begin{equation}\label{E:quad-vertex}
\QuaVertex
  = \QuaVertexa + \frac{[2m-8][m-2]}{[m-4][2m-4]} \  \QuaVertexb
  = \QuaVertexc + \frac{[2m-8][m-2]}{[m-4][2m-4]} \  \QuaVertexd
\end{equation}

The second ingredient is a theorem of Sundaram that implies the dimension of $U_{\K}(\mathfrak{sp}_{2n})$ invariant tensors in $V_{\K}^{\otimes d}$ is equal to the number of matchings between $2d$ points satisfying a condition reminiscent of ``pattern avoidance" in the theory of Coxeter groups \cite[Proposition 5.37]{bodish2021type}. The basis is constructed by interpreting each matching as a quadrivalent graph in a disc, where the $2d$ points are where the graph meets the boundary of the disc, then viewing the graph as morphisms in the $\mathfrak{sp}_{2n}$ web category.\footnote{This recipe says that any matching can thus give rise to a morphism, but it is possible to rewrite such a diagram as a linear combination of diagrams associated to matchings satisfying the pattern avoidance condition \cite[Theorem 5.35]{bodish2021type}.} 

To prove her theorem, Sundaram notes that the dimension of the space of invariant tensors in $V_{\kk}^{\otimes d}$ is counted by length $2d$ up-down tableaux with less than $n+1$ rows. Then, Sundaram defines an analogue of the Robinson-Schensted bijection, between matchings of $2d$ points and length $2d$ up-down tableaux, and establishes that the bijection intertwines the number of rows on the tableaux side with a pattern avoidance condition on the quadrivalent graph side \cite{Sundaram}.\footnote{Actually, the bijection used for type $C$ webs, which is carefully explained in \cite{BERT-viennot}, is a slight variation of Sundaram's original bijection.}

It follows from our main theorem, and well-known results on representation theory of $O(m)$, that the dimension of $\Hom_{\WebK}(0, 1^{\otimes 2d})$ is counted by up-down tableaux, but with the restriction that the sum of the lengths of the first two columns is less than $m+1$. Applying Sundaram's bijection to such tableaux, and then interpreting these matchings as quadrivalent graphs, we can use \EQ \eqref{E:quad-vertex} to view these graphs as web diagrams. A natural conjecture is that these web diagrams are a basis.


\subsubsection{BMW antisymmetrizer}
It is well-known that the braiding of the vector representation generates all $U_{\K}(\om)$ endomorphisms of arbitrary tensor powers of $V_{\K}$. This braid group representation factors through the BMW algebra. These BMW algebras can be put together into the BMW category, denoted $\BMWK$, which is a ribbon category generated by one object such that the endomorphism algebras of tensor powers of that object are BMW algebras. This category can be defined over $\A$ and specialized to $\kk$ or $\K$. For $\R\in \{\kk, \K\}$, we can use Proposition \ref{BMWweb} to interpret the right hand side of Equation \eqref{claspVSsym} as the generator of the kernel of the functor from the category $\BMW$ to the monoidal subcategory of representations of $U_{\R}(\om)$ generated by $V_{\R}$. When $\R=\kk$, the left hand side of Equation \eqref{claspVSsym} is an explicit description of this endomorphism in $\BMWC$. As we mentioned above, there are explicit descriptions of this generator when $\R=\K$, but the answers are somewhat complicated. Our web category gives a diagrammatic method to derive explicit formulas for this element of $\BMW$. Using our web relations to rewrite the right hand side of Equation \eqref{claspVSsym} in terms of the braiding, the cup, and the cap, one might try to give a diagrammatic proof of known results about the $q$-analogue of the antisymmetrizer in $\BMW$, or maybe find new (potentially nicer) formulas.


\subsection{Structure of the paper}

In Section $2$, we use the definition of $\Web$ to derive further relations, define a braiding endomorphism for $1\otimes 1$, and prove that the web category is finitely generated. In Section $3$, we define the quantum exterior algebra as a module over the quantum analogue of the orthogonal group, then we define some morphisms between tensor products of the quantum exterior powers which we will eventually show are monoidal generators of the fundamental category. In Section $4$, we show there is a functor from the web category to the fundamental category, then argue that this functor is compatible with the functor induced by the action of the Brauer algebra on tensor powers of the defining representation of the orthogonal group. In Section $5$, we prove our main Theorem by reducing to standard subcategories, citing well-known results from classical invariant theory, and using basic facts about specialization to $\kk$ and $\K$ from $\A$.

\tableofcontents

\subsection{Acknowledgments}

E.B. was supported by the National Science Foundation’s M.S.P.R.F.-$2202897$. He thanks Ben Elias, David Rose, and Logan Tatham for teaching him about webs, Daniel Tubbenhauer for discussing quantum exterior algebra, and Victor Ostrik for sharing the exercise in Kashiwara-Shapira. H.W. was supported by the NSF grant CCF-$2009029$ and the Simons Foundation grant $994328$. He thanks Greg Kuperberg for teaching him webs and sharing insights on webs of different Lie types, Monica Vazirani for teaching him Young tableau and BMW algebra, and Eugene Gorsky for his helpful comments and feedback. 
We also thank the referee for their insightful feedback.

%% file: webs.tex
\section{Web category for quantum orthogonal group}

\def\tri
{\begin{tric}
\draw[scale=0.7] (-2,-1.732)--(-1,0) (1,0)--(2,-1.732) 
(-2,-1.732)node[below,black,scale=0.7]{$m$}
(1,0)--(0,1.5) node[right,pos=0.6,black,scale=0.7]{$1$}
(0,1.5)--(-1,0) node[left,pos=0.4,black,scale=0.7]{$1$}
(0,1.5)--(0,3) node[left,midway,black,scale=0.7]{$2$}
(2,-1.732) node[below,black,scale=0.7]{$m$}
(-1,0)--(1,0)node[below,midway,black,scale=0.7]{$m-1$};
\end{tric}
}

This section concerns $\Web$, which was introduced in Definition \ref{DefofWeb}. The category $\Web$ is defined by generators and relations. We will derive some further relations in $\Web$, and then establish a connection to the work of Birman-Murakami-Wenzl \cite{MR992598, MR927059}.

\subsection{Further relations}\label{SS:further-renls}

\def\Altbigon
{\begin{tric}
\draw[scale= 0.8] (0,0.7)--(0,1.5) (0,-0.7)--(0,-1.5)
        (0,1.5)node[above,black,scale=0.7]{$k$}
        (0,-1.5)node[below,black,scale=0.7]{$k$}
      (0,0.7)..controls(-0.5,0.7)and(-0.5,-0.7)..(0,-0.7) node[left,midway,black,scale=0.7]{$1$} (0,0.7)..controls(0.5,0.7)and(0.5,-0.7)..(0,-0.7)
      node[right,midway,black,scale=0.7]{$k+1$} ;
\end{tric}
}

\def\Altbigona
{\begin{tric}
\draw [scale=0.8] (0,1.5)--(0,-1.5)  node[below,black,scale=0.7]{$k$} ;
\end{tric}
}

\begin{lemma} 
    \begin{align}   \Altbigon =   \frac{[2m-2k][2m-4k-4][m-2k]}{[2][m-2k-2][2m-4k]} \Altbigona \label{reversebigon}
     \end{align} 
\end{lemma}

\def\BRbigonA
{\begin{tric}
\draw[scale=0.7] (-2,-1.732)--(-1,0) (1,0)--(2,-1.732)
(-1,0)..controls(-1,2)and(1,2)..(1,0) node[above,midway,black,scale=0.7]{$1$}
(-2,-1.732)node[below,black,scale=0.7]{$k$}
(2,-1.732) node[below,black,scale=0.7]{$k$}
(-1,0)--(1,0)node[below,midway,black,scale=0.7]{$k+1$};
\end{tric}
}

\def\BRbigonB
{\begin{tric}
\draw [scale=0.7]  (0,1)..controls(1,1)and(1,2)..(0,2)..controls(-1,2)and(-1,1)..(0,1)
(0.7,1.5) node[right,black,scale=0.7]{$1$}
(1.74,-3)--(0,0) node[right,pos=0.7,black,scale=0.7]{$1$}
(0,0)--(-1.74,-3) node[left,pos=0.3,black,scale=0.7]{$1$}
(0,0)--(0,1) node[left,midway,black,scale=0.7]{$2$}
(-1.74,-3) node[below,black,scale=0.7]{$k$}
(1.74,-3) node[below,black,scale=0.7]{$k$}
(1.16,-2)--(-1.16,-2)node[below,midway,black,scale=0.7]{$k-1$};
\end{tric}
}

\def\BRbigonC
{\begin{tric}
\draw[scale=0.7] (-2,-1.732)--(-1,0)  (1,0)--(2,-1.732) 
(-1,0)..controls(-1,2)and(1,2)..(1,0) node[above,midway,black,scale=0.7]{$1$}
(-2,-1.732)node[below,black,scale=0.7]{$k$}
(2,-1.732) node[below,black,scale=0.7]{$k$}
(-1,0)--(1,0)node[below,midway,black,scale=0.7]{$k-1$};
\end{tric}
}

\def\BRbigonD
{\begin{tric}
\draw  [scale=0.7] 
(0,0.5) circle (1) 
(0,1.5) node[above,black,scale=0.7]{$1$}
(2,-2)..controls(1,-1)and(-1,-1)..(-2,-2) node[below,black,midway,scale=0.7]{$k$};
\end{tric}
}

\def\BRbigonE
{\begin{tric}
\draw[scale=0.7] 
(0,-1.732) .. controls(0,1)and(3,1) .. (3,-1.732)
 node[above,black,scale=0.7,midway]{$k$} ;
\end{tric}
}

\begin{proof}
\begin{align*}
\BRbigonA 
&\stackrel{(\ref{defskein}e)}{=}  
\BRbigonB 
-\frac{[2m-4k-4][m-2k]}{[m-2k-2][2m-4k]} \BRbigonC 
+ \frac{[2m-4k-4][m-2]}{[m-2k-2][2m-4]} \BRbigonD \\
&\stackrel{\substack{
(\ref{defskein}b) \\ (\ref{defskein}c) \\ (\ref{defskein}a) }}{=}   -\frac{[2m-4k-4][m-2k]}{[m-2k-2][2m-4k]} \frac{[2k]}{[2]}
   \BRbigonE
   +\frac{[2m-4k-4][m-2]}{[m-2k-2][2m-4]} 
                \frac{[2m-4][m]}{[m-2][2]}    \BRbigonE
\end{align*}
\end{proof}

\begin{remark}\label{R:coeffs-simple-q=1-2}
When $\R = \kk$, the coefficient in \EQ \eqref{reversebigon} becomes $(m-k)$. 
\end{remark}

\def\WebQuanDim
{\begin{tric}
\draw (0.7,0) circle (0.7);
\draw (0.7,0.7)node[above,black,scale=0.7]{$k+1$}
      (0.7,-0.7)node[below,black,scale=0.7]{\phantom{x}}; 
\end{tric}
}

\def\WebQuanDima
{\begin{tric}
\draw[scale= 0.8] 
        (0,0.7)..controls(0,1.5)and(-2,1.5)..(-2,0) 
        (0,-0.7)..controls(0,-1.5)and(-2,-1.5)..(-2,0)
        (-2,0)node[left,black,scale=0.7]{$k+1$}
      (0,0.7)..controls(-0.5,0.7)and(-0.5,-0.7)..(0,-0.7) node[left,midway,black,scale=0.7]{$1$} (0,0.7)..controls(0.5,0.7)and(0.5,-0.7)..(0,-0.7)
      node[right,midway,black,scale=0.7]{$k$} ;
\end{tric}
}

\def\WebQuanDimb
{\begin{tric}
\draw (0.7,0) circle (0.7);
\draw (0.7,0.7)node[above,black,scale=0.7]{$k$}
      (0.7,-0.7)node[below,black,scale=0.7]{\phantom{x}}; 
\end{tric}
}

Using the previous Lemma, it is not hard to derive the following relation generalizing \EQ ({\ref{defskein}a})
\begin{lemma}
\begin{equation}\label{E:thick-web-circle}
\Skeinak = \frac{[2m-4k][m]}{[m-2k][2m]}
\begin{bmatrix}
m \\
k 
\end{bmatrix} _{q^2}
\end{equation}
\end{lemma}
\begin{proof}
We prove the claim by induction on $k$. The base case, $k=1$, follows from \EQ ({\ref{defskein}a}) .
Assuming \EQ \eqref{E:thick-web-circle} holds for $k$, we find

\begin{align*}
\WebQuanDim
&\stackrel{(\ref{defskein}c)}{=}\frac{[2]}{[2k+2]}\WebQuanDima
\stackrel{\eqref{reversebigon}}{=} \frac{[2]}{[2k+2]}
\frac{[2m-2k][2m-4k-4][m-2k]}{[2][m-2k-2][2m-4k]}\WebQuanDimb \\
&\stackrel{\eqref{E:thick-web-circle}}{=}
\frac{[2]}{[2k+2]}\frac{[2m-2k][2m-4k-4][m-2k]}{[2][m-2k-2][2m-4k]}
\frac{[2m-4k][m]}{[m-2k][2m]}
\begin{bmatrix}
m \\
k 
\end{bmatrix} _{q^2} \\
&=\frac{[2m-2k][2m-4k-4][m]}{[2k+2][m-2k-2][2m]}\begin{bmatrix}
m \\
k 
\end{bmatrix} _{q^2} \\
&=\frac{[2m-4(k+1)][m]}{[m-2(k+1)][2m]}\begin{bmatrix}
m \\
k+1
\end{bmatrix} _{q^2} .
\end{align*}
\end{proof}

\begin{remark}\label{R:coeffs-simple-q=1-3}
When $\R = \kk$, the coefficient in \EQ \eqref{E:thick-web-circle} becomes $\binom{m}{k}$. 
\end{remark}

\begin{remark}\label{R:L'Hospitals}
Note that $[2m-4k]/[m-2k] = [2]_{q^{m-2k}}$. Therefore, if $m=2k$, then $[2m-4k]/[m-2k] =2$.
\end{remark}

The following relations are a simplification of \EQ (\ref{defskein}e) when $k=m$.

\begin{lemma}
    $$      \tri=0  \ \ \ \ \  \quad \text{and} \quad
    \ \ \ \ \ \VanTria =
    \frac{[m-2][2m]}{[2m-4][m]} \VanTrib $$
\end{lemma}

\begin{proof}
When $k=m$, the left hand side of \EQ ({\ref{defskein}e}) is zero, since a strand carries the label $m+1$. Now, postcompose \EQ {(\ref{defskein}e)}, when $k=m$, with a trivalent vertex $1\otimes 1\rightarrow 2$ and then simplify to derive the triangle equals $0$ relation. This triangle is a subdiagram of the first term on the right hand side of \EQ {(\ref{defskein}e)}, so that term is also zero. It is then easy to derive the identity equals merge-split relation in the statement of the Lemma.
\end{proof}

\def\SmallTriangle 
{\begin{tric}
\draw [scale=0.7]
(1.74,-3)--(0,0) node[right,pos=0.7,black,scale=0.7]{$1$}
(0,0)--(-1.74,-3) node[left,pos=0.3,black,scale=0.7]{$1$}
(0,0)--(0,1.5) node[left,midway,black,scale=0.7]{$2$}
(-1.74,-3) node[below,black,scale=0.7]{$1$}
(1.74,-3) node[below,black,scale=0.7]{$1$}
(1.16,-2)--(-1.16,-2)node[below,midway,black,scale=0.7]{$2$};
\end{tric}
}

\def\SmallTrianglea
{\begin{tric}
\draw [scale=0.7] 
(1.16,-2) node[below,black,scale=0.7]{$1$}
(1.16,-2)--(0,0) 
(-1.16,-2) node[below,black,scale=0.7]{$1$}
(0,0)--(-1.16,-2)
(0,0)--(0,2) node[left,midway,black,scale=0.7]{$2$};
\end{tric}
}


\subsection{The braiding}

\def\Braid
{\begin{tric}
\draw  (-1.5,1.5)--(1.5,-1.5) ;
\draw[double=darkblue,ultra thick,white,line width=3pt] (1.5,1.5)--(-1.5,-1.5);
\draw (-1.5,-1.5) node[below,black,scale=0.7]{$1$}
      (1.5,-1.5) node[below,black,scale=0.7]{$1$}; 
\end{tric}
}

\def\NegBraid
{\begin{tric}
\draw   (1.5,1.5)--(-1.5,-1.5);
\draw[double=darkblue,ultra thick,white,line width=3pt] (-1.5,1.5)--(1.5,-1.5);
\draw (-1.5,-1.5) node[below,black,scale=0.7]{$1$}
      (1.5,-1.5) node[below,black,scale=0.7]{$1$}; 
\end{tric}
}

\def\BraidCompo
{\begin{tric}
\draw [scale=0.7] (-1.5,1.5)..controls(-1.5,0)and(1.5,0) ..(1.5,-1.5) ;
\draw[scale=0.7, double=darkblue,ultra thick,white,line width=3pt] (1.5,1.5)..controls(1.5,0)and(-1.5,0) ..(-1.5,-1.5);

\draw[scale=0.7] (-1.5,-1.5)..controls(-1.5,-3)and(1.5,-3) ..(1.5,-4.5) ;
\draw[scale=0.7, double=darkblue,ultra thick,white,line width=3pt] (1.5,-1.5)..controls(1.5,-3)and(-1.5,-3) ..(-1.5,-4.5);
\draw [scale=0.7] (1.5,-1.5) node[left,black,scale=0.7]{$1$};
\draw [scale=0.7] (-1.5,-1.5) node[right,black,scale=0.7]{$1$};
\end{tric}
}

\def\ReideA
{\begin{tric}
\draw  (-1.5,1.5)..controls(0,-1.5)and(1.5,-1.5)..(1.5,0);
\draw[double=darkblue,ultra thick,white,line width=3pt] (-1.5,-1.5)..controls(0,1.5)and(1.5,1.5)..(1.5,0);
\draw (1.5,0) node[left, black, scale=0.7] {$1$};
\end{tric}
}

\def\ReideotherA
{\begin{tric}
\draw (-1.5,-1.5)..controls(0,1.5)and(1.5,1.5)..(1.5,0);
\draw [double=darkblue,ultra thick,white,line width=3pt] (-1.5,1.5)..controls(0,-1.5)and(1.5,-1.5)..(1.5,0);
\draw (1.5,0) node[left, black, scale=0.7] {$1$};
\end{tric}
}

\def\ReideAa
{\begin{tric}
\draw  [scale=0.7] 
(-2,2)..controls(-1,1)and(-1,-1)..(-2,-2) node[left,black,midway,scale=0.7]{$1$}
(1,-2)..controls(-0.5,-2)and(-0.5,2)..(1,2) node[right,black,midway,scale=0.7]{$1$}
(1,-2)..controls(2.5,-2)and(2.5,2)..(1,2) ;
\end{tric}
}

\def\ReideAb
{\begin{tric}
\draw[scale=0.7] (-1.732,-2)--(0,-1)  (0,1)--(-1.732,2) 
(0,-1)..controls(2,-1)and(2,1)..(0,1) node[left,black,midway,scale=0.7]{$1$}
(-1.732,-2)node[below,black,scale=0.7]{$1$}
(-1.732,2) node[above,black,scale=0.7]{$1$}
(0,-1)--(0,1)node[left,midway,black,scale=0.7]{$2$};
\end{tric}
}

\def\ReideAc
{\begin{tric}
\draw  [scale=0.7] 
(-2,2)..controls(-0.5,1)..(1,1) node[above,black,midway,scale=0.7]{$1$}
(-2,-2)..controls(-0.5,-1)..(1,-1)
(1,1)..controls(2,1)and(2,-1)..(1,-1); 
\end{tric}
}

\def\ReideAhalf
{\begin{tric}
\draw[scale=0.7] (-1.5,-1.5)..controls(-1.5,-1)and(-1,-0.5)..(0,-0.5)
                 (1.5,-1.5)..controls(1.5,-1)and(1,-0.5).. (0,-0.5)
                 (0,-0.5)--(0,1)node[left,black,midway,scale=0.7]{$2$}; 

\draw[scale=0.7] (-1.5,-1.5)..controls(-1.5,-3)and(1.5,-3) ..(1.5,-4.5)
node[below,black,scale=0.7]{$1$};
\draw[scale=0.7, double=darkblue,ultra thick,white,line width=3pt] (1.5,-1.5)..controls(1.5,-3)and(-1.5,-3) ..(-1.5,-4.5)node[below,black,scale=0.7]{$1$};
\end{tric}
}

\def\IoverReideAhalf
{\begin{tric}
\draw[scale=0.7] (-1.5,-1.5)..controls(-1.5,-1)and(-1,-0.5)..(0,-0.5)
                 (1.5,-1.5)..controls(1.5,-1)and(1,-0.5).. (0,-0.5)
                 (0,-0.5)--(0,1)node[left,black,midway,scale=0.7]{$2$}; 
\draw[scale=0.7] (-1.5,1.5)--(0,1)--(1.5,1.5);
\draw[scale=0.7] (-1.5,-1.5)..controls(-1.5,-3)and(1.5,-3) ..(1.5,-4.5)
node[below,black,scale=0.7]{$1$};
\draw[scale=0.7, double=darkblue,ultra thick,white,line width=3pt] (1.5,-1.5)..controls(1.5,-3)and(-1.5,-3) ..(-1.5,-4.5)node[below,black,scale=0.7]{$1$};
\draw[scale=0.7] (-1.4,1.5) node[left,black,scale=0.7]{$1$}
(1.4,1.5) node[right,black,scale=0.7]{$1$};
\end{tric}
}

\def\bigonontri
{\begin{tric}
\draw[scale=0.7]  (-1.5,-1.5)..controls(-1.5,-1)and(-1,-0.5)..(0,-0.5)
                 (1.5,-1.5)..controls(1.5,-1)and(1,-0.5).. (0,-0.5)
                 (-1.5,-1.5)..controls(-1.5,-2)and(-1,-2.5)..(0,-2.5)
                 (1.5,-1.5)..controls(1.5,-2)and(1,-2.5).. (0,-2.5)
                 
                 (0,-0.5)--(0,1)node[left,black,midway,scale=0.7]{$2$}
                (0,-2.5)--(0,-3.5)node[left,black,midway,scale=0.7]{$2$}
                 
                 (-1.5,-4)--(0,-3.5)--(1.5,-4); 
\draw[scale=0.7] (-1.45,-1.5) node[right,black,scale=0.7]{$1$}
(1.45,-1.5) node[left,black,scale=0.7]{$1$}
(-1.5,-4)node[below,black,scale=0.7]{$1$}
(1.5,-4)node[below,black,scale=0.7]{$1$};
\end{tric}
}

\def\singononw
{\begin{tric}
\draw[scale=0.7] (-1.5,-1.5)..controls(-1.5,-1)and(-1,-0.5)..(0,-0.5)
                 (1.5,-1.5)..controls(1.5,-1)and(1,-0.5).. (0,-0.5)
                 (-1.5,-1.5)..controls(-1.5,-2)and(-1,-2.5)..(0,-2.5)
                 (1.5,-1.5)..controls(1.5,-2)and(1,-2.5).. (0,-2.5)
                 
                 (0,-0.5)--(0,1)node[left,black,midway,scale=0.7]{$2$}
                 
                 (-1.5,-4)..controls(-1,-3)and(1,-3)..(1.5,-4); 

\draw[scale=0.7] (0,-4.2) node[above, black, scale=0.7]{$1$}
(0,-2.5) node[above, black, scale=0.7]{$1$}; 
\end{tric}
}

\def\ReideAhalfother
{\begin{tric}
\draw[scale=0.7] (-1.5,-1.5)..controls(-1.5,-1)and(-1,-0.5)..(0,-0.5)
                 (1.5,-1.5)..controls(1.5,-1)and(1,-0.5).. (0,-0.5)
                 (0,-0.5)--(0,1)node[left,black,midway,scale=0.7]{$2$}; 

\draw[scale=0.7]  (1.5,-1.5)..controls(1.5,-3)and(-1.5,-3) ..(-1.5,-4.5) 
node[below,black,scale=0.7]{$1$};
\draw[scale=0.7, double=darkblue,ultra thick,white,line width=3pt]  (-1.5,-1.5)..controls(-1.5,-3)and(1.5,-3) ..(1.5,-4.5)node[below,black,scale=0.7]{$1$};
\end{tric}
}

\def\String
{\begin{tric}
\draw  (0,1.5)--(0,-1.2);
\draw (0,-1.2) node[below,black,scale=0.7]{$1$};
\end{tric}
}

\def\ReideB
{\begin{tric}
\draw [scale=0.7] (1.5,1.5)..controls(1.5,0)and(-1.5,0) ..(-1.5,-1.5) ;
\draw[scale=0.7, double=darkblue,ultra thick,white,line width=3pt] (-1.5,1.5)..controls(-1.5,0)and(1.5,0) ..(1.5,-1.5) ;

\draw[scale=0.7] (-1.5,-1.5)..controls(-1.5,-3)and(1.5,-3) ..(1.5,-4.5) ;
\draw[scale=0.7, double=darkblue,ultra thick,white,line width=3pt] (1.5,-1.5)..controls(1.5,-3)and(-1.5,-3) ..(-1.5,-4.5);
\draw [scale=0.7] (1.5,-1.5) node[left,black,scale=0.7]{$1$};
\draw [scale=0.7] (-1.5,-1.5) node[right,black,scale=0.7]{$1$};
\end{tric}
}

\def\ReideBmirror
{\begin{tric}
\draw [scale=0.7] (-1.5,1.5)..controls(-1.5,0)and(1.5,0) ..(1.5,-1.5) ;
\draw[scale=0.7, double=darkblue,ultra thick,white,line width=3pt] (1.5,1.5)..controls(1.5,0)and(-1.5,0) ..(-1.5,-1.5) ;

\draw[scale=0.7] (1.5,-1.5)..controls(1.5,-3)and(-1.5,-3) ..(-1.5,-4.5) ;
\draw[scale=0.7, double=darkblue,ultra thick,white,line width=3pt] (-1.5,-1.5)..controls(-1.5,-3)and(1.5,-3) ..(1.5,-4.5) ;
\draw [scale=0.7] (1.5,-1.5) node[left,black,scale=0.7]{$1$};
\draw [scale=0.7] (-1.5,-1.5) node[right,black,scale=0.7]{$1$};
\end{tric}
}

\def\BraidTriv
{\begin{tric}
\draw [darkgreen] (-1.5,1.5)--(1.5,-1.5) node[below,scale=0.7,black]{$1$};
\draw [darkgreen] (1.5,1.5)--(-1.5,-1.5) node[below,scale=0.7,black]{$1$};

\end{tric}
}

\def\Braida
{\begin{tric}
\draw  [scale=0.7] 
(-2,2)..controls(-1,1)and(-1,-1)..(-2,-2) node[left,black,midway,scale=0.7]{$1$}
(2,-2)..controls(1,-1)and(1,1)..(2,2) node[right,black,midway,scale=0.7]{$1$};
\end{tric}
}

\def\grBraida
{\begin{tric}
\draw  [scale=0.7,darkgreen] 
(-2,2)..controls(-1,1)and(-1,-1)..(-2,-2) node[left,black,midway,scale=0.7]{$1$}
(2,-2)..controls(1,-1)and(1,1)..(2,2) node[right,black,midway,scale=0.7]{$1$};
\end{tric}
}

\def\crossing
{\begin{tric}
\draw (1,0)--(0,1);
\draw[double=darkblue,ultra thick,white,line width=3pt] (0,0)--(1,1);
\end{tric}}

\def\Braidb
{\begin{tric}
\draw[scale=0.7] (-1.732,-2)--(0,-1)--(1.732,-2) (1.732,2)--(0,1)--(-1.732,2) 
(-1.732,-2)node[below,black,scale=0.7]{$1$}
(1.732,-2)node[below,black,scale=0.7]{$1$}
(1.732,2)node[above,black,scale=0.7]{$1$}
(-1.732,2) node[above,black,scale=0.7]{$1$}
(0,-1)--(0,1)node[left,midway,black,scale=0.7]{$2$};
\end{tric}
}

\def\grBraidb
{\begin{tric}
\draw[scale=0.7,darkgreen] (-1.732,-2)--(0,-1)--(1.732,-2) (1.732,2)--(0,1)--(-1.732,2) 
(-1.732,-2)node[below,black,scale=0.7]{$1$}
(1.732,-2)node[below,black,scale=0.7]{$1$}
(1.732,2)node[above,black,scale=0.7]{$1$}
(-1.732,2) node[above,black,scale=0.7]{$1$}
(0,-1)--(0,1)node[left,midway,black,scale=0.7]{$2$};
\end{tric}
}

\def\Braidc
{\begin{tric}
\draw  [scale=0.7] 
(-2,2)..controls(-1,1)and(1,1)..(2,2) node[above,black,midway,scale=0.7]{$1$}
(2,-2)..controls(1,-1)and(-1,-1)..(-2,-2) node[below,black,midway,scale=0.7]{$1$};
\end{tric}
}

\def\twistovercap
{\begin{tric}
\draw  [scale=0.7] 
(2,-3) ..controls(0,-2)and(-3,0)..(0,0); 
\draw[scale=0.7,double=darkblue,ultra thick,white,line width=3pt]
(-2,-3)..controls(0,-2)and(3,0)..(0,0) ;

\draw  [scale=0.7] 
(2,2)..controls(1,1)and(-1,1)..(-2,2) node[above,black,midway,scale=0.7]{$1$};
\draw (0,0) node[below,black,scale=0.7]{$1$};
\end{tric}
}

\def\ReidemCa
{\begin{tric}
\draw  (-2,2)--(2,-2);
\draw [double=darkblue,ultra thick,white,line width=3pt] (0,-2)..controls(-2,-1)and(-2,1)..(0,2);
\draw [double=darkblue,ultra thick,white,line width=3pt](-2,-2)--(2,2) ;
\draw (0,-2) node[below,black,scale=0.7]{$1$};
\draw (2,-2) node[below,black,scale=0.7]{$1$};
\draw (-2,-2) node[below,black,scale=0.7]{$1$};
\end{tric}
}

\def\ReidemCb
{\begin{tric}
\draw  (-2,2)--(2,-2);
\draw [double=darkblue,ultra thick,white,line width=3pt] (0,-2)..controls(2,-1)and(2,1)..(0,2);
\draw [double=darkblue,ultra thick,white,line width=3pt](-2,-2)--(2,2) ;
\draw (0,-2) node[below,black,scale=0.7]{$1$};
\draw (2,-2) node[below,black,scale=0.7]{$1$};
\draw (-2,-2) node[below,black,scale=0.7]{$1$};
\end{tric}
}

\def\ReidemDa
{\begin{tric}
\draw        (90:0)--(90:2) 
           (306:2)--(90:0)--(234:2); 
\draw [double=darkblue,ultra thick,white,line width=3pt]   (162:2)..controls(126:1.6)and(54:1.6)..(18:2);
\end{tric}
}

\def\ReidemDb
{\begin{tric}
\draw        (90:0)--(90:2) 
           (306:2)--(90:0)--(234:2) ;
\draw [double=darkblue,ultra thick,white,line width=3pt] (162:2)..controls(234:1.7)and(306:1.7)..(18:2);
\end{tric}
}

\def\ReidemDaUnder
{\begin{tric}
\draw (162:2)..controls(126:1.6)and(54:1.6)..(18:2);
\draw    [double=darkblue,ultra thick,white,line width=3pt]                 (90:0)--(90:2) 
           (306:2)--(90:0)--(234:2) ;
\end{tric}
}

\def\ReidemDbUnder
{\begin{tric}
\draw (162:2)..controls(234:1.7)and(306:1.7)..(18:2);
\draw     [double=darkblue,ultra thick,white,line width=3pt]
           (90:0)--(90:2) 
           (306:2)--(90:0)--(234:2) ;
\end{tric}
}

\begin{defn}\label{D:brweb}
Let $\R\in \{\kk, \A, \K\}$. We define $\brweb \in
 \rm{Hom}_{\Web}(1 \otimes1 , 1 \otimes 1) $ as 
 \begin{align}
 \Braid := q^2 \ \Braida  \ - \Braidb  
 - \frac{[m-2]}{[2m-4]}(q^2-q^{-2})\cdot q^{-m+2} \Braidc \label{defofbraid} 
 \end{align}
\end{defn}

\begin{notation}
We will write the $90$ degree rotation of $\brweb$ diagrammatically as follows.
\[
\NegBraid
\]
\end{notation}

\begin{proposition} \label{BMWrelations}
 \begin{align}
   & \NegBraid = q^{-2} \ \Braida  \ - \Braidb  
 + \frac{[m-2]}{[2m-4]}(q^2-q^{-2})\cdot q^{m-2} \Braidc  \label{negbraid}\\ 
 & \Braid -\NegBraid = (q^2-q^{-2})\cdot \left( \ \Braida \ -\Braidc \right) \label{BMWdefre}
 \end{align}

\end{proposition}

\begin{proof}
\EQ \eqref{negbraid} follows from rotating the diagrams on both sides of \EQ \eqref{defofbraid}, and then apply \EQ (\ref{defskein}e) when k=1. \EQ \eqref{BMWdefre} follows from \EQ \eqref{defofbraid} and \EQ \eqref{negbraid}.  
\end{proof}

\begin{remark}\label{R:identities}
We have the identities
\[
- \cfrac{[m-2]}{[2m-4]}(q^2-q^{-2})\cdot q^{-m+2} = q^{-2} - \cfrac{[2m-8][m-2]}{[m-4][2m-4]}
\]
and 
\[
\cfrac{[m-2]}{[2m-4]}(q^2-q^{-2})\cdot q^{m-2} = q^{2} - \cfrac{[2m-8][m-2]}{[m-4][2m-4]}.
\]
\end{remark}

\begin{proposition} \label{Reidemeister}
The following relations hold in $\Web$. 
 \begin{align}
     \ReideA = \ q^{2m-2} \ \ \String  \ \ \ \ \ \ &, \ \ \ \ \ \ \  \ReideotherA = \ q^{2-2m} \ \ \String  \label{ReideOne} \\
\ReideAhalf = \ -q^{-2} \ \ \SmallTrianglea  \ \ \ \ \ \ &, \ \ \ \ \ \ \  \ReideAhalfother = \ -q^{2} \ \ \SmallTrianglea \label{ReideOPF} \\
    \ReideB = \ \Braida  \ \ \ \ \ \ &, \ \ \ \ \ \ \ \ReideBmirror = \ \Braida  \label{ReideTwo} \\
   \ReidemCa &= \ReidemCb \label{ReideThree} 
 \end{align}

\end{proposition}

\begin{proof}
\begin{align*}
&\ReideA  \stackrel{\eqref{defofbraid}}{=}
q^2 \ \ReideAa  \ - \ReideAb  
 - \frac{[m-2]}{[2m-4]}(q^2-q^{-2})\cdot q^{-m+2} \ReideAc \\
 &\phantom{xxxxxxx} \stackrel{ \substack{(\ref{defskein}a) \\ \eqref{reversebigon}} }{=}
 \left( q^2 \frac{[2m-4][m]}{[m-2][2]}   -  
 \frac{[2m-2][2m-8][m-2]}{[2][m-4][2m-4]}
 - \frac{[m-2]}{[2m-4]}(q^2-q^{-2})\cdot q^{-m+2} \right) \ \String  \\
 &\ReideAhalf \stackrel{\eqref{defofbraid}}{=}  
 q^2   \SmallTrianglea - \bigonontri
 - \frac{[m-2](q^2-q^{-2})\cdot q^{-m+2}}{[2m-4]}  \singononw 
  \stackrel{ \substack{(\ref{defskein}c) \\ (\ref{defskein}b)} }{=} 
  \Big( q^2-\frac{[4]}{[2]} \Big) \SmallTrianglea \\
 \end{align*}

 \begin{align*}
 \ReideB &  \stackrel{\eqref{negbraid}}{=}
  q^{-2} \ \Braid  \ - \IoverReideAhalf 
 + \frac{[m-2]}{[2m-4]}(q^2-q^{-2})\cdot q^{m-2} \twistovercap \\
 &  \stackrel{\substack{\eqref{defofbraid} \\ \eqref{ReideOPF} \\ \eqref{ReideOne}} }{=}
 q^{-2} \left(  q^2 \ \Braida  \ - \Braidb  - \frac{[m-2]}{[2m-4]}(q^2-q^{-2})\cdot q^{-m+2} \Braidc \right) \\ 
 & \ \ \ \ \ \ \ \ \ \ \ \ \ \ \ \ \ \ \ \ \ \ \ \ \ \ + q^{-2} \Braidb
 + \frac{[m-2]}{[2m-4]}(q^2-q^{-2})\cdot q^{m-2} q^{2-2m} \Braidc \\
\end{align*}
The same argument to verify the Reidemeister $III$ braid relation in the proof of \cite[Proposition 5.7]{bodish2021type} also works in $\Web$, so we leave the verification of \EQ \eqref{ReideThree} as an exercise.
\end{proof}

\def\BraidTriva
{\begin{tricc}

\draw
(-1.5,-1.5)--(1.5,1.5)node[above,scale=0.7,black]{$1$} 
(1.5,-1.5)--(-1.5,1.5)node[above,scale=0.7,black]{$1$};
\draw (1.5,-1.5)--(-1.5,-1.5)  node[midway,below,scale=0.7,black]{$k-1$}
      (1.5,-1.5)--(1.7,-2.5)  node[below,scale=0.7,black]{$k$}
      (-1.5,-1.5)--(-1.7,-2.5) node[below,scale=0.7,black]{$k$};
\end{tricc}
}

\begin{cor}
       $$\BraidCompo = \ \Braida  \ + \  (q^2-q^{-2})  \Braid  
 - (q^2-q^{-2})\cdot q^{-2m+2} \Braidc  $$   
\end{cor}

\begin{proof}
Compose \EQ \eqref{BMWdefre} with the braiding $\brweb$, then apply \EQS \eqref{ReideOne} and \eqref{ReideTwo}. 
\end{proof} 

\begin{notation}
As noted in Remark \ref{R:coeffs-simple-q=1-1}, Remark \ref{R:coeffs-simple-q=1-2}, and Remark \ref{R:coeffs-simple-q=1-3}, upon specialization to $\kk$ there is a drastic simplification in the coefficients of the defining relations. In order to make clear to the reader which calculations hold for any $\R \in \{\K, \A, \kk\}$ and which are special to $\WebC$, we will color the diagrams in $\WebC$ green. Thus, a blue diagram is interpreted in $\Web$ for some $\R \in \{\K, \A, \kk\}$, depending on context, while a green diagram is always interpreted in $\WebC$.
\end{notation}

\begin{notation}
When $\R =\kk$, Equation \eqref{BMWdefre} implies ${^{\R}\beta}_{1,1} = {^{\R}\beta}^{-1}_{1,1}$, so in our green diagrammatic calculus for $\WebC$ we do not distinguish between the over-crossing and the under-crossing. Thus, the formula for the braiding becomes
\begin{equation}\label{E:q=1-braid}
\BraidTriv :=  \ \grBraida  \ - \grBraidb.
\end{equation}
\end{notation}

\begin{lemma} \label{classicalcrossing}
When $\R =\kk$, we have 

    \begin{equation}\label{EQ:classicalcrossing}
    \GrSkeinIH \ \ = \ \ \GrSkeinIHc - \BraidTriva \ .
\end{equation}

\end{lemma}
\begin{proof}
This follows from
\[
\BraidTriva \ \ \stackrel{\eqref{E:q=1-braid}}{=} \ \ \GrSkeinIHb\ \ -\ \ \GrSkeinIHa \ \ \stackrel{(\ref{defskein}e)}{=} \ \ \GrSkeinIHc \ \ -\ \ \GrSkeinIH \ .
\]
\end{proof}

\def\QuaVertex
{\begin{tric}
\draw  (-1.5,1.5)--(1.5,-1.5) ;
\draw  (1.5,1.5)--(-1.5,-1.5) ;
\filldraw (0,0) circle (3pt) ; 
\end{tric}
}

\def\QuaVertexa
{\begin{tric}
\draw[scale=0.7] (-1.732,-2)--(0,-1)--(1.732,-2) (1.732,2)--(0,1)--(-1.732,2) 
(-1.732,-2)node[below,black,scale=0.7]{$1$}
(1.732,-2)node[below,black,scale=0.7]{$1$}
(1.732,2)node[above,black,scale=0.7]{$1$}
(-1.732,2) node[above,black,scale=0.7]{$1$}
(0,-1)--(0,1)node[left,midway,black,scale=0.7]{$2$};
\end{tric}
}

\def\QuaVertexb
{\begin{tric}
\draw  [scale=0.7] 
(-2,2)..controls(-1,1)and(1,1)..(2,2) node[above,black,midway,scale=0.7]{$1$}
(2,-2)..controls(1,-1)and(-1,-1)..(-2,-2) node[below,black,midway,scale=0.7]{$1$};
\end{tric}
}

\def\QuaVertexc
{\begin{tric}
\draw[scale=0.7] (-2,-1.732)--(-1,0)--(-2,1.732) (2,1.732)--(1,0)--(2,-1.732) 
(-2,-1.732)node[below,black,scale=0.7]{$1$}
(-2,1.732)node[above,black,scale=0.7]{$1$}
(2,1.732)node[above,black,scale=0.7]{$1$}
(2,-1.732) node[below,black,scale=0.7]{$1$}
(-1,0)--(1,0)node[below,midway,black,scale=0.7]{$2$};
\end{tric}
}

\def\QuaVertexd
{\begin{tric}
\draw  [scale=0.7] 
(-2,2)..controls(-1,1)and(-1,-1)..(-2,-2) node[left,black,midway,scale=0.7]{$1$}
(2,-2)..controls(1,-1)and(1,1)..(2,2) node[right,black,midway,scale=0.7]{$1$};
\end{tric}
}

\def\BraidBMW
{\begin{tric}
\draw[black]  (-1.5,1.5)--(1.5,-1.5) ;
\draw[double=black,ultra thick,white,line width=3pt] (1.5,1.5)--(-1.5,-1.5);
\end{tric}
}

\def\NegBraidBMW
{\begin{tric}
\draw[black]   (1.5,1.5)--(-1.5,-1.5);
\draw[double=black,ultra thick,white,line width=3pt] (-1.5,1.5)--(1.5,-1.5);
\end{tric}
}
\def\BraidaBMW
{\begin{tric}
\draw  [black,scale=0.7] 
(-2,2)..controls(-1,1)and(-1,-1)..(-2,-2) node[left,black,midway,scale=0.7]{$1$}
(2,-2)..controls(1,-1)and(1,1)..(2,2) node[right,black,midway,scale=0.7]{$1$};
\end{tric}
}
\def\BraidcBMW
{\begin{tric}
\draw  [black,scale=0.7] 
(-2,2)..controls(-1,1)and(1,1)..(2,2) node[above,black,midway,scale=0.7]{$1$}
(2,-2)..controls(1,-1)and(-1,-1)..(-2,-2) node[below,black,midway,scale=0.7]{$1$};
\end{tric}
}

\def\BraidaBMWunlabel
{\begin{tric}
\draw  [black,scale=0.7] 
(-2,2)..controls(-1,1)and(-1,-1)..(-2,-2) node[left,black,midway,scale=0.7]{}
(2,-2)..controls(1,-1)and(1,1)..(2,2) node[right,black,midway,scale=0.7]{};
\end{tric}
}
\def\BraidcBMWunlabel
{\begin{tric}
\draw  [black,scale=0.7] 
(-2,2)..controls(-1,1)and(1,1)..(2,2) node[above,black,midway,scale=0.7]{}
(2,-2)..controls(1,-1)and(-1,-1)..(-2,-2) node[below,black,midway,scale=0.7]{};
\end{tric}
}

\def\ReideABMW
{\begin{tric}
\draw[black]  (-1.5,1.5)..controls(0,-1.5)and(1.5,-1.5)..(1.5,0);
\draw[double=black,ultra thick,white,line width=3pt] (-1.5,-1.5)..controls(0,1.5)and(1.5,1.5)..(1.5,0);
\end{tric}
}

\def\StringBMW
{\begin{tric}
\draw[black]  (0,1.5)--(0,-1.5);
\end{tric}
}

\subsection{Finite generation}

In order to make certain arguments relating the $O(m)$ web category over $\K$ and over $\kk$, we will need to know that the homomorphism spaces in $\WebA$ are finitely generated. We first show that the webs with all boundary labels $1$ can be rewritten in terms of the braiding along with cups and caps.

\begin{defn}
Define the standard web category $\StdWeb$ as the full monoidal subcategory of $\Web$ generated by the object 1. 
\end{defn}

\begin{defn}
Define the braiding standard web category $\WebB$ as the pivotal subcategory of $\Web$, where the objects in $\WebB$ are tensor products of the self-dual object 1, and morphisms in $\WebB$ are generated by the braiding $\brweb$.
\end{defn}

\begin{rmk}\label{R:WebB-is-braided}
Let $a,b\in \mathbb{Z}_{\ge 0}$, and write $w_0$ for the longest element in the symmetric group $S_{a+b}$. Define $\beta_{1^{\otimes a},1^{\otimes b}}$ to be the diagram in $\WebB$ which is the positive braid lift of the minimal length element in the coset $w_0\cdot (S_{a}\times S_{b}) \in S_{a+b}/(S_{a}\times S_{b})$. Using that $\WebB$ is pivotal, \EQ \eqref{ReideTwo}, and \EQ \eqref{ReideThree}, a standard argument shows that this family of maps satisfy naturality and the hexagon axioms, see e.g. \cite[Example 2.1]{JS-braided}, and thus make the category $\WebB$ a braided category.
\end{rmk}

\def\bigonwrHIH
{\begin{tric}
\draw[scale=0.7] (-2,-1.732)--(-1,0) (-1,0)--(-2,1.732) (1,0)--(2,-1.732) 
(1.8,1.3)--(1,0)node[right,pos=0.6,black,scale=0.7]{$l-1$}
(-2,-1.732)node[below,black,scale=0.7]{$l-1$}
(-2,1.732)node[above,black,scale=0.7]{$1$}
(2,-1.732) node[below,black,scale=0.7]{$1$}
(-1,0)--(1,0)node[below,midway,black,scale=0.7]{$l$};
\draw[scale=0.7]
(1.8,1.3)..controls(1.5,2)and(1.7,2.3)..(2.5,2.5)
node[left,midway,black,scale=0.7]{$1$}
(1.8,1.3)..controls(3,1.6)and(2.7,2.3)..(2.5,2.5)
node[right,midway,black,scale=0.7]{$l-2$}
(2.5,2.5)--(3.5,4)node[right,pos=0.6,black,scale=0.7]{$l-1$};
\end{tric}
}

\def\triwrHIH
{\begin{tric}
\draw [scale=0.7]  (-3,-1.74)--(-2,0)--(-3,1.74) 
(4,2.5)--(0,0) node[above,pos=0.65,black,scale=0.7]{$1$}
(0,0)--(4,-2.5) 
(2.5,-1.8)node[left,black,scale=0.7]{$l-1$}
(0,0)--(-2,0) node[below,midway,black,scale=0.7]{$l$}
(4,-2.5) node[below,black,scale=0.7]{$1$}
(4.2,2.5) node[above,black,scale=0.7]{$l-1$}
(-3,-1.74) node[below,black,scale=0.7]{$l-1$}
(-3,1.74) node[above,black,scale=0.7]{$1$}
(3,1.875)--(3,-1.875)node[right,midway,black,scale=0.7]{$l-2$};
\end{tric}
}

\def\wrHIH
{\begin{tric}
\draw[scale=0.7] (-2,-1.732)--(-1,0)--(-2,1.732) (2,1.732)--(1,0)--(2,-1.732) 
(-2,-1.732)node[below,black,scale=0.7]{$l-1$}
(-2,1.732)node[above,black,scale=0.7]{$1$}
(2,1.732)node[above,black,scale=0.7]{$l-1$}
(2,-1.732) node[below,black,scale=0.7]{$1$}
(-1,0)--(1,0)node[below,midway,black,scale=0.7]{$l$};
\end{tric}
}

\begin{proposition} \label{BMWweb}
$\StdWeb=\WebB$
\end{proposition}

\begin{proof}
We need to show that any morphism in  $\StdWeb$ is also a morphism in $\WebB$. We will prove this for web diagrams in $\StdWeb$. Then the desired result is immediate for linear combinations of web diagrams. By \EQ \eqref{defofbraid}, we have 
\[
\textbf{Span} \left \{  \QuaVertexa \ , \  \QuaVertexb \ , \  \QuaVertexd \right \}
   = \textbf{Span} \left \{  \Braid  \ , \  \QuaVertexb \ , \  \QuaVertexd \right \},
\]
so it suffices to show that an arbitrary web diagram in $\StdWeb$ can be rewritten as a linear combination of diagrams with strands only labelled $1$ and $2$.

Fix a diagram $f$ in $\StdWeb$. Suppose that the largest label on a strand in $f$ is $l$. If $l \geq m+1$, then $f=0$, and if $l \le 2$, then we are done. Assume that $3 \leq l \leq m$. Fix a point in this $l$ labelled strand, then choose a direction and traverse the strand away from this point in that direction. Since $f$ is in $\StdWeb$, the $l$ labelled strand cannot extend to the boundary. So either we return to this point, or we meet a trivalent vertex. Since $l$ is the largest label of a strand in $f$, and the trivalent vertex must be a generator from Definition \ref{DefofWeb}, this trivalent vertex has labels $1, l-1$, and $l$.

If we meet a trivalent vertex, then traversing the $l$ labelled strand in the other direction we also meet a trivalent vertex with labels $1, l-1$, and $l$. On the other hand, if the strand is closed, then we can use \EQ (\ref{defskein}c), with $k=l$, to introduce two trivalent vertices with labels $1, l-1$, and $l$, up to an invertible scalar in $\A$. In either case, the $l$ labelled strand is a segment between two trivalent vertices with labels $1, l-1$, and $l$.

We can either apply \EQ (\ref{defskein}e), with $k+1=l$, or apply the following relation
\[
\wrHIH \stackrel{(\ref{defskein}c)}{=} \frac{[2]}{[2l-2]} \bigonwrHIH \stackrel{(\ref{defskein}d)}{=} \frac{[2]}{[2l-2]} \triwrHIH
\]
and then apply \EQ (\ref{defskein}e), with $k+1=l$. Thus, we can write $f$ as a linear combination of web diagrams, each of which has one fewer strand with label $l$. By induction, we can remove all strands with label $l$, so $f$ is a linear combination of diagrams with largest strand label less than or equal to $l-1$. Using induction again, we find that $f$ is a linear combination of web diagrams with only $1$ and $2$ labelled strands.
\end{proof}

\begin{definition}\label{def:BMW}
Define the \emph{$BMW$ category}, $\BMW$, to be the free $\R $-linear braided pivotal category\footnote{This means we can draw positive crossings to represent the braiding of $\bullet$ with itself, as well as draw cups and caps coming from pivotal structure.} with generating object $\bullet$ which is self-dual of dimension $\frac{[2m-4][m]}{[m-2][2]}$ such that 
\[
\BraidBMW -\NegBraidBMW = (q^2-q^{-2})\cdot \left( \ \BraidaBMWunlabel \ -\BraidcBMWunlabel \right)
\]
and
\[
\ReideABMW = \ q^{2m-2} \ \ \StringBMW \quad .
\]
\end{definition}

\begin{proposition}\label{P:functor-bmw-to-web}
The assignments $\bullet \mapsto 1$ and 
\[
\BraidBMW \mapsto \Braid
\]
determine a full pivotal braided monoidal functor
\[
\eta_\R :\BMW\longrightarrow \StdWeb.
\]
\end{proposition}
\begin{proof}
It is clear that $1$ is self dual with dimension $\frac{[2m-4][m]}{[m-2][2]}$. Combining Remark \ref{R:WebB-is-braided} with Lemma \ref{BMWweb} we also see that $\StdWeb$ is braided. Thanks to \EQ \eqref{BMWdefre} and \EQ \eqref{ReideOne}, the claim follows from the universal mapping property of $\BMW$. The image of $\eta_\R $ is $\WebB = \StdWeb$, so $\eta_\R $ is full.
\end{proof}

\begin{lemma}\label{L:BMWfingen}
Homomorphism spaces in $\BMWA$ are finitely generated $\A$-modules.
\end{lemma}
\begin{proof}
This is standard, for example see \cite[Theorem 3]{BMWfinitegen}.
\end{proof}

\begin{proposition}\label{P:StdWeb-finitely-generated}
Homomorphism spaces in $\StdWebA$ are finitely generated $\A$-modules.
\end{proposition}
\begin{proof}
Since $\eta_{\A}$ is full, this follows from Lemma \ref{L:BMWfingen}.
\end{proof}

%% file: fund.tex
\section{Representation theory of the quantum orthogonal group}

In this section we carefully introduce the ingredients needed to define the pivotal category $\Fund$ from Theorem \ref{T:main-thm}. That is we recall the (quantum analogue's of):
\begin{itemize}
    \item The finite dimensional representations of the orthogonal group, Section \ref{ssec:qorthalg}.
    \item The defining vector representation of the orthogonal group, and the exterior algebra of this representation, Section \ref{ssec:quantum-vector-rep} and Section \ref{ssec:quantum-exterior-algebra}.
    \item Highest weight theory for the $O(m)$ representations, particularly for exterior powers of the defining vector representation, Section \ref{ssec:fund-O}.
    \item The multiplication and self-duality structures of the exterior powers of the defining vector representation, as well as the braiding for the defining vector representation, Section \ref{ssec:gen-intertwiners}.
\end{itemize}

Along the way, we are careful to work with $\R =\A$ whenever possible, as this allows us to base change to the quantum analogue, i.e. $\R = \K$, and to specialize to $q=1$, i.e. $\R =\kk$. Having this flexibility is essential later in the paper when we prove Theorem \ref{T:main-thm}.

\subsection {Quantum orthogonal algebra}\label{ssec:qorthalg}

The orthogonal group $O(m)$ has connected component $SO(m)$. In particular, the Lie algebra of $O(m)$ is the same as the Lie algebra $\mathfrak{so}_m= \mathrm{Lie}(SO(m))$. The Drinfeld-Jimbo quantum group is defined to be a $q$-analogue of the universal enveloping algebra of a simple Lie algebra, thus there is a Drinfeld-Jimbo quantum group $U_q(\mathfrak{so}_m)$. The group $O(m)$ has component group $O(m)/SO(m) \cong \mathbb{Z}/2$, so to $q$-deform the representation theory of $O(m)$, we have to encode the action of component group $O(m)/SO(m)$ on representations of quantum $U_q(\mathfrak{so}_m)$. This leads to the definition of the quantum orthogonal algebra $U_q(\mathfrak{o}_m)$, see Definition \ref{D:Uo}.

\begin{defn}\label{defn:quantumgp}
Let $\mathfrak{g}$ be a semisimple Lie algebra, over $\mathbb{C}$, with associated root system $\Phi$, viewed as a subset of the weight lattice $X$. Fix a choice of simple roots $\Delta\subset \Phi$. The Weyl group $W$ acts on $\mathbb{Z}\Phi$. Write $(-,-)$ to denote the unique $W$ invariant symmetric bilinear form on $\mathbb{Z}\Phi$, normalized such that $(\alpha, \alpha)=2$ whenever $\alpha$ is a short root. Write
\[
\alpha^{\vee}:=\frac{2\alpha}{(\alpha, \alpha)}\in X \quad \text{and} \quad q_{\alpha}:=q^{(\alpha, \alpha)/2} \in \mathbb{C}(q) \quad \text{for all $\alpha \in \Delta$}.
\]

Define $U_q(\mathfrak{g})$, following \cite[Section 4.3]{JantzenQgps}, as the associative $\K$-algebra generated by 
\[
E_{\alpha}, F_{\alpha}, K_{\alpha}^{\pm 1}, \ \alpha\in \Delta
\]
subject to the relations:
\begin{gather*}
K_{\alpha} K_{\alpha}^{-1} = 1= K_{\alpha}^{-1} K_{\alpha}, \quad K_{\alpha} K_{\beta} = K_{\beta} K_{\alpha}, 
\quad  K_{\alpha} E_{\beta}= q^{(\alpha, \beta)} E_{\beta} K_{\alpha}, \quad  K_{\alpha} F_{\beta} = q^{-(\alpha, \beta)} F_{\beta} K_{\alpha} \, , \\
E_{\alpha} F_{\beta} - F_{\beta}E_{\alpha} = \delta_{\alpha, \beta}\dfrac{K_{\alpha}- K_{\alpha}^{-1}}{q_{\alpha} - q_{\alpha}^{-1}} \, , \\
\sum_{s= 0}^{1- a_{\alpha, \beta}} (-1)^s{1- a_{\alpha, \beta}\brack s}_{q_{\alpha}} E_{\alpha}^{1- a_{\alpha, \beta}- s} E_{\beta} E_{\alpha}^s = 0, 
\quad \sum_{s= 0}^{1- a_{\alpha, \beta}} (-1)^s{1- a_{\alpha, \beta}\brack s}_{q_{\alpha}} F_{\alpha}^{1- a_{\alpha, \beta}- s} F_{\beta} F_{\alpha}^s = 0 \, ,
	\end{gather*}
where $a_{\alpha, \beta}:= (\alpha^{\vee} ,\beta)$.

The algebra $U_q(\mathfrak{g})$ is a Hopf algebra with comultiplication $\Delta$, antipode $S$, and counit $\epsilon$ defined on generators as follows:
\begin{equation}\label{E:comult}
\Delta(E_{\alpha}) = E_{\alpha}\otimes 1 + K_{\alpha}\otimes E_{\alpha}, \quad \Delta(F_{\alpha})= 1\otimes F_{\alpha} + F_{\alpha}\otimes K_{\alpha}^{-1}, \quad \Delta(K_{\alpha})= K_{\alpha}\otimes K_{\alpha},
\end{equation}
\begin{equation}\label{E:antipode}
S(E_{\alpha}) = - K_{\alpha}^{-1}E_{\alpha}, \quad S(F_{\alpha})= -F_{\alpha}K_{\alpha}, \quad S(K_{\alpha})= K_{\alpha}^{-1},
\end{equation}
\begin{equation}\label{E:counit}
\epsilon(E_{\alpha})=0,\quad \epsilon(F_{\alpha})= 0, \quad \text{and} \quad \epsilon(K_{\alpha})= 1.
\end{equation}

We also define Lusztig's divided powers algebra, denoted $U_{\A}(\mathfrak{g})$, as the $\A$-subalgebra in $U_q(\mathfrak{g})$ generated by $K_{\alpha}^{\pm 1}$, $E_{\alpha}^{(n)}:=E_{\alpha}^n/[n]_{q_{\alpha}}!$, and $F_{\alpha}^{(n)}:=F_{\alpha}^n/[n]_{q_{\alpha}}!$, for all $\alpha\in \Delta$ and $n\in \mathbb{Z}_{\ge 0}$.
\end{defn}

Write $X(\som)\subset \oplus_{i=1}^n \mathbb{Z}\frac{\epsilon_i}{2}$ for the weight lattice of $\som$, where $m=2n$ if $m$ is even, and $m=2n+1$ if $m$ is odd. We enumerate the simple roots for $\mathfrak{so}_{2n}$ (i.e. type $D_n$) as
\[
\Delta= \{\alpha_1 = \epsilon_1- \epsilon_2, \ldots, \alpha_{n-1}= \epsilon_{n-1}- \epsilon_n, \alpha_n = \epsilon_{n-1} + \epsilon_n\},
\]
and for $\mathfrak{so}_{2n+1}$ (i.e. type $B_n$) as
\[
\Delta = \{\alpha_1= \epsilon_1- \epsilon_2, \ldots, \alpha_{n-1} = \epsilon_{n-1}-\epsilon_n, \alpha_n = \epsilon_n\}.
\]
The pairing $(-,-)$ for $\mathfrak{so}_{2n}$ is defined as $(\epsilon_i, \epsilon_j)= \delta_{i,j}$ and for $\mathfrak{so}_{2n+1}$ is defined as $(\epsilon_i, \epsilon_j) = 2\delta_{i,j}$. The fundamental weights for $\som$ are:
\[
\varpi_1= \epsilon_1, \varpi_2= \epsilon_1+ \epsilon_2, \dots, \varpi_{n-2}= \epsilon_1 + \dots + \epsilon_{n-2},
\]
\[
\varpi_{n-1} = \frac{\epsilon_1 + \dots + \epsilon_{n-1}-\epsilon_n}{2}, \quad \text{and} \quad \varpi_n = \frac{\epsilon_1 + \dots + \epsilon_{n-1}+\epsilon_n}{2},
\]
if $m=2n$, and
\[
\varpi_1=\epsilon_1, \varpi_2= \epsilon_1 + \epsilon_2, \dots, \varpi_{n-1}= \epsilon_1 + \dots + \epsilon_{n-1}, 
\]
and
\[
\varpi_n = \frac{\epsilon_1 + \dots + \epsilon_{n-1} + \epsilon_n}{2},
\]
if $m=2n+1$. The dominant weights, $X_+(\som)$, are the $\mathbb{Z}_{\ge 0}$ span of the fundamental weights.

\begin{notation}
In order to make the statements of our results uniform, we need to compensate for the different conventions for $(-,-)$ if $m$ is even or odd. To this end, we will write 
\[
U_{\K}(\som):=\begin{cases}
U_q(\mathfrak{so}_{m}), \quad \text{if $m$ is odd}, \quad \text{and}\\
U_{q^2}(\mathfrak{so}_m), \quad \text{if $m$ is even}.
\end{cases}
\]
\end{notation}

\begin{definition}
For $\wta\in X(\som)$ and $V\in U_{\K}(\som)$-mod,
we set
\[
V[\wta]:=\{v\in V \ | \ K_{\alpha}v= q_{\alpha}^{(\alpha^{\vee}, \wta)}v, \  \text{for all}  \ \alpha\in \Delta\}.
\]
If $v\in V[\wta]$, then we say that $v$ is a \emph{weight vector} of \emph{weight $\wta$}.
\end{definition}

\begin{definition}\label{D:K-weight-space}
Suppose that $V$ is a finite dimensional $U_{\K}(\som)$-module such that 
\[
V= \oplus_{\wta\in X(\som)}V[\wta], 
\]
then we say that $V$ is a \emph{type-$\textbf{1}$} $U_{\K}(\som)$-module.
\end{definition}

For each $\wta\in X_+(\som)$, there is an irreducible type-$\textbf{1}$ $U_{\K}(\som)$-module with highest weight $\wta$, and highest weight vector $v_{\wta}^+$, which we will denote by $L_{\K}(\wta)$. Moreover, each finite dimensional irreducible type-$\textbf{1}$ $U_{\K}(\som)$-module is isomorphic to $L_{\K}(\wta)$ for some $\wta\in X_+(\som)$ \cite[Theorem 5.10]{JantzenQgps}.

\begin{definition}
Suppose that $V$ is a free finitely generated $\A$-module with an action of $U_{\A}(\som)$ such that the $K_{\alpha}$ action on $V$ is diagonalizable over $\A$ with all eigenvalues positive powers of $q$, for all $\alpha\in \Delta$. Then we say that $V$ is a type-$\textbf{1}$ $U_{\A}(\som)$-module.
\end{definition}

\begin{notation}\label{N:C-weight-space}
Let $U_{\kk}(\som)$ denote the usual enveloping algebra of $\mathfrak{so}_m$. Upon specialization to $\kk$, we have instead to consider the elements $h_{\alpha}$ in the Cartan subalgebra. For $V \in U_{\kk}(\som)$-mod and $\wta\in X(\som)$ we have
\[
V[\wta]:= \{v\in V \ | \ h_{\alpha}v= \wta(h_{\alpha})v, \ \text{for all} \ \alpha\in \Delta\}.
\]
This is the classical notion of weight vector. For convenience, we will refer to finite dimensional $U_{\kk}(\som)$-modules as type-$\textbf{1}$ modules. 
\end{notation}

\begin{notation}
Let $\R\in \{\kk, \A, \K\}$. Suppose that $V$ is a type-$\textbf{1}$ $U_{\R}(\som)$-module. If $V[\wta] \ne 0$, then we say that $\wta$ is a \emph{weight} of $V$.
\end{notation}

\begin{definition}
Let $\R\in \{\kk, \A, \K\}$. Write $\chi(\som)$ for the free $\mathbb{Z}$-module with basis $\{e^{\wta}\}_{\wta\in X(\som)}$. The \emph{formal character} of a type-$\one$ representation $V$ is the expression 
\[
\mathrm{ch}(V):=\sum_{\wta\in X(\som)} \dim_{\R}V[\wta]\cdot e^{\wta}\in \chi(\som).
\]
\end{definition}

For each $\wta\in X^+(\som)$ there is an irreducible $U_{\kk}({\som})$-module with highest weight $\wta$, and highest weight vector $v_{\wta}^+$, which we denote by $L_{\kk}(\wta)$. Each finite dimensional irreducible representation of $U_{\kk}(\som)$ is isomorphic to $L_{\kk}(\wta)$ for some $\wta\in X_+(\som)$. 

\begin{lemma}\label{L:so-linearindependence}
Let $\R \in \{\kk, \K\}$. The characters $\{\mathrm{ch}(L_{\R}(\wta))\}_{\wta \in X_+(\som)}$ are a basis for the span of formal characters of all type-$\one$ representations.
\end{lemma}
\begin{proof}
Use that if $L_{\R}(\wta)[\wtb]\ne 0$, then $(\wta - \wtb) \in \mathbb{Z}_{\ge 0}\Phi_+$.
\end{proof}

\begin{lemma}\label{L:so-completely-reducible}
Let $\R\in \{\kk, \K\}$. Every type-$\one$ $U_{\R}(\som)$-module is completely reducible. 
\end{lemma}
\begin{proof}
This is Weyl's theorem on complete reducibility, when $\R= \kk$, and \cite[Theorem 5.17]{JantzenQgps}, when $\R= \K$
\end{proof}

\begin{lemma}\label{L:formalcharimpliesdecomp}
If $V$ is a type-$\textbf{1}$ $U_{\A}(\som)$-module, then
\[
\kk\otimes V\cong \bigoplus_{\wta\in X_+(\som)}L_{\kk}(\wta)^{\oplus m_{\wta}},\quad \K\otimes V\cong \bigoplus_{\wta\in X_+(\som)}L_{\K}(\wta)^{\oplus n_{\wta}}, \quad \text{and $m_{\wta} = n_{\wta}$}.
\]
\end{lemma}
\begin{proof}
Follows from Lemma \ref{L:so-completely-reducible}, Lemma \ref{L:so-linearindependence}, and that $\dim_{\kk} L_{\kk}(\wta)[\wtb]= \dim_{\K} L_{\K}(\wta)[\wtb]$ for all $\wta\in X_+(\som)$ and $\wtb\in X(\som)$ \cite[Theorem 5.15]{JantzenQgps}.
\end{proof}

\begin{lemma}\label{L:som-Schurslemma}
Let $\R\in \{\kk, \K\}$. If $\wta, \wtb\in X_+(\som)$, then $\Hom_{U_\R(\som)}(L_\R(\wta), L_\R(\wtb)) = 0$ if $\wta\ne \wtb$, and $\End_{U_{\R}(\som)}(L_\R(\wta))= \R\cdot \id_{L_\R(\wta)}$.
\end{lemma}
\begin{proof}
Follows from Schur's lemma and standard theory about highest weight vectors.
\end{proof}

Now, we explain how to encode the action of the component group $O(m)/SO(m)$ on type-$\one$ representations of $U_{\R}(\mathfrak{so}_m)$. This combination of the group $O(m)/SO(m)$ and the algebra $U_{\R}(\mathfrak{so}_m)$ is neatly packaged together as the quantum orthogonal algebra, which we define in Definition \ref{D:Uo} and Definition \ref{D:Uo-C}.

When $m=2n$ there is an order $2$ automorphism $\sigma$ of the Dynkin diagram, swapping the simple roots $\alpha_{n-1} = \epsilon_{n-1}- \epsilon_n$ and $\alpha_n= \epsilon_{n-1}+\epsilon_n$. This induces an automorphism of $U_{\K}(\som)$ such that
\[
\sigma(E_{\alpha_{n-1}})= E_{\alpha_n}, \ \sigma(F_{\alpha_{n-1}})= F_{\alpha_n}, \ \sigma(K_{\alpha_{n-1}})= K_{\alpha_n},
\]
\[
\sigma(E_{\alpha_n})= E_{\alpha_{n-1}}, \ 
\sigma(F_{\alpha_n})= F_{\alpha_{n-1}}, \ 
\sigma(K_{\alpha_n})= K_{\alpha_{n-1}},
\]
and $\sigma$ fixes all the other generators for $U_{\K}(\som)$. 

If $m=2n+1$, there are no Dynkin diagram automorphisms. In this case we write $\sigma$ to denote the identity automorphism of $U_{\K}(\som)$.

\begin{defn}{\cite[Section 8.1.2]{LZ-stronglymultifree}}\label{D:Uo}
Let $U_{\K}(\om)$ be the associative algebra generated by $U_{\K}(\som)$ and $\sigma$, such that $\sigma^2=1$ and $\sigma X \sigma^{-1} = \sigma(X)$, for $X\in U_{\K}(\som)$.


The algebra $U_{\K}(\om)$ is a Hopf algebra with $\Delta, S, \epsilon$ defined on elements of $U_{\K}(\som)$ as in Definition \ref{defn:quantumgp}, along with 
\[
\Delta(\sigma) = \sigma\otimes \sigma, \quad S(\sigma) = \sigma^{-1}, \ \text{and} \quad \epsilon(\sigma) = 1. 
\]

The automorphism $\sigma$ preserves $U_{\A}(\som)\subset U_{\K}(\som)$, so we define $U_{\A}(\om)$ as the algebra generated by $U_{\A}(\som)$ and $\sigma$, such that $\sigma X\sigma^{-1} = \sigma(X)$,
for $X\in U_{\A}(\som)$. Note that $U_{\A}(\om)$ is the unital $\A$ subalgebra of $U_K(\om)$ generated by $\sigma$, $K_{\alpha}^{\pm 1}$, $E_{\alpha}^{(n)}$, and $F_{\alpha}^{(n)}$, for all $n\in \mathbb{Z}_{\ge 0}$, $\alpha\in \Delta$.
\end{defn}

\begin{defn}\label{D:Uo-C}
Define $U_{\kk}(\om)$ as the universal enveloping algebra of $\mathfrak{so}_m(\kk)$, denoted $U(\som(\kk))$, augmented by the algebra automorphism, which we will denote by $\sigma$, determined by the non-trivial Dynkin diagram automorphism when $m$ is even, and the identity automorphism when $m$ is odd.
\end{defn}

For any $U_{\R}(\om)$-module which restricts to a type-$\one$ $U_{\R}(\som)$-module, we can use the same notion of weight spaces as in Definition \ref{D:K-weight-space} and Notation \ref{N:C-weight-space}, and such a module will be a direct sum of its weight spaces. Note that the equation $\sigma K_{\alpha} \sigma^{-1} = \sigma(K_{\alpha})$ implies that $\sigma$ acts on weight spaces. The induced action on weights is such that $\sigma$ acts on $X(\som)$ trivially if $m=2n+1$, and $\sigma$ swaps $\varpi_{n-1}$ and $\varpi_{n}$ if $m=2n$.

\begin{remark}
Any finite dimensional representation of $O(\kk^{m})$ is a finite dimensional representation of $U_{\kk}(\om)$ such that the weights are contained in $\oplus_{i=1}^n\mathbb{Z}\epsilon_i$, and vice-versa. Moreover, a linear map between such representations commutes with the actions of $O(\kk^{m})$ if and only if the map commutes with $U_{\kk}(\om)$. Such representations are exactly the $O(\kk^{m})$ modules which occur as submodules of $(\kk^{m})^{\otimes d}$ for some $d\ge 0$.
\end{remark}

\begin{definition}\label{D:orthogonal-typeone}
Let $\R\in \{\kk, \A, \K\}$. A $U_\R(\om)$-module such that its restriction to $U_\R(\som)$ is type-$\textbf{1}$ with weights contained in $\oplus_{i=1}^n\mathbb{Z}\epsilon_i$, will be referred to as a \emph{type-$\textbf{1}$} $U_\R(\om)$-module\footnote{An example of a $U_{\kk}(\om)$-module which is not of this form would be the induction, from $U_{\kk}(\mathfrak{so}_{2n+1})$, of the irreducible spinor module $V(\varpi_k)$, i.e. $U_{\kk}(\mathfrak{o}_{2n+1})\otimes_{U_{\kk}(\mathfrak{so}_{2n+1})}V(\varpi_n)$.}.
\end{definition}

\begin{definition}
Let $\R[\mathbb{Z}/2]$ denote the group algebra of $\mathbb{Z}/2$ over $\R$. There is an algebra homomorphism
\[
U_\R(\om)\rightarrow \R[\mathbb{Z}/2]
\]
with $U_\R(\som)$ in the kernel and such that $\sigma\mapsto -1\in \mathbb{Z}/2$. Composing this homomorphism with the sign representation of $\mathbb{Z}/2$, we obtain a one dimensional $U_\R(\om)$-module, denoted $\det_\R$.
\end{definition}

\begin{remark}
The module $\detR$ restricts to the trivial $U_{\R}(\som)$ module, and therefore is a type-$\one$ $U_{\R}(\om)$-module.
\end{remark}

There is a classification of finite dimensional irreducible type-$\textbf{1}$ $U_{\R}(\om)$-modules. If $\wta\in \oplus_{i=1}^n\mathbb{Z}\epsilon_i\cap X_+(\som)$, then $\wta = \sum_{i=1}^n \wta_i\epsilon_i$ such that $\wta_1\ge \dots \ge \wta_{n-1}\ge |\wta_n|$, and $\wta_n= |\wta_n|$ if $m=2n+1$, while $\wta_n$ is any integer if $m=2n$. For such an $\wta$, we obtain a representation $L_{\R}(\wta)$ of $U_{\R}(\som)$, which we can then induce to $U_{\R}(\om)$. The induced module $U_{\R}(\om)\otimes_{U_{\R}(\som)} L_{\R}(\wta)$ is isomorphic to $L_{\R}(\wta)\oplus L_{\R}(\sigma(\wta))$ as $U_{\R}(\som)$-modules, by the map $1\otimes \ell \mapsto (\ell, 0)$ and $\sigma\otimes \ell \mapsto (0,\ell)$. The action of $U_{\R}(\om)$ is determined by
\[
\sigma\cdot (\ell, \ell') = (\ell', \ell) \quad \text{and} \quad X\cdot (\ell, \ell') = (X\cdot \ell, \sigma(X)\cdot \ell'). 
\]
If $m=2n+1$, or $m=2n$ and $\wta_n=0$, then the induced module decomposes into a direct sum of two irreducible $U_{\R}(\om)$-modules corresponding to the $+1$ and $-1$ eigenspaces of $\sigma$. We write $L_{\R}(\wta, +1)$ and $L_{\R}(\wta, -1)$ for these representations. If $m=2n$ and $\wta_n\ne 0$, then the induced module is irreducible and is isomorphic to $U_{\R}(\om)\otimes_{U_{\R}(\som)}L(\wta_1, \dots,\wta_{n-1}, -\wta_n)$.

\begin{proposition}\label{P:o-list-of-simples}
Let $\R\in \{\kk, \K\}$. The following is a complete and irredundant list of irreducible type-$\one$ representations of $U_{\R}(\om)$. For $m=2n+1$:
\[
L_{\R}(\wta, +1) \quad \text{and} \quad L_{\R}(\wta, -1), \quad \text{such that} \quad \wta_1\ge \dots \ge \wta_n\ge 0,
\]
and for $m=2n$:
\begin{align*}
U_{\R}(\om)\otimes_{U_{\R}(\som)} L(\wta)\quad \text{such that} \quad \wta_1\ge \dots \ge \wta_n >0, \\
L_{\R}(\wta, +1) \quad \text{and} \quad L_{\R}(\wta, -1), \quad \text{such that} \quad \wta_1\ge \dots \ge \wta_n= 0.
\end{align*}
\end{proposition}
\begin{proof}
Use that $\sigma$ is central when $m=2n+1$. For $m=2n$, observe that $\sigma$ preserves the space of vectors annihilated by $E_{\alpha}$'s and acts on weights by $(\wta_1, \dots, \wta_{n-1}, \wta_n)\mapsto (\wta_1, \dots, \wta_{n-1}, -\wta_n)$. For more details, see \cite[Section 5.5.5]{GoodmanWallach-invariants}.
\end{proof}

\begin{lemma}\label{L:o-Schurslemma}
Let $\R\in \{\kk, \K\}$. If $S$ and $T$ are two irreducible type-$\one$ $U_{\R}(\om)$-modules from the list of irreducibles in Proposition \ref{P:o-list-of-simples}, then
\[
\Hom_{U_{\R}(\om)}(S,T) = \begin{cases}
    0 \quad \text{if $S\ne T$}, \quad \text{and} \\
    \R\cdot \id_{S} \quad \text{if $S=T$}.
\end{cases}
\]
\end{lemma}
\begin{proof}
This follows by looking first at $\Hom_{U_{\R}(\som)}(\textbf{Res}(S), \textbf{Res}(T))$, then analyzing which of these maps commute with $\sigma$. We leave it to the reader to complete the case-by-case analysis.
\end{proof}

If $W$ is a type-$\one$ $U_{\R}(\om)$-module, then $W^*:=\Hom_{\R}(W, \R)$ is an $U_{\R}(\om)$-module via the antipode, denoted $S$ in Definition \ref{D:Uo}. We say $W$ is \emph{self-dual} if $W\cong W^*$ as $U_{\R}(\om)$-modules.

\begin{lemma}\label{L:selfduality-of-O-reps}
Let $\R\in \{\kk, \K\}$. The type-$\one$ irreducible representations of $U_{\R}(\om)$ are self-dual.
\end{lemma}
\begin{proof}
The irreducible representations of $U_{\R}(\som)$ are self-dual. We leave it as an exercise to the reader to verify that inducing a self-dual module from $U_{\R}(\som)$ to $U_{\R}(\om)$ results in a self-dual module. Since a direct summand of a self-dual module is self-dual, the claim follows from Proposition \ref{P:o-list-of-simples}.
\end{proof}

\begin{lemma}\label{L:o-completely-reducible}
Let $\R\in \{\kk, \K\}$. Every type-$\textbf{1}$ $U_{\R}(\om)$-module is completely reducible. 
\end{lemma}
\begin{proof}
Let $W$ be a type-$\one$ $U_{\R}(\om)$-module and let $S\subset W$ be a $U_{\R}(\om)$-submodule, it suffices to show that $S$ is a direct summand of $W$. We adapt the argument in \cite[Theorems 3.1, 9.2]{Alperin-book}. By Lemma \ref{L:so-completely-reducible} there is an idempotent $e_S\in \End_{U_{\R}(\som)}(W)$ with image $S$. Note that $\sigma$ induces a linear endomorphism of $W$ which preserves $S$. One can check that $e_S':=\frac{1}{2}(e_S + \sigma \circ e_S\circ \sigma)$ is an endomorphism of $W$ which commutes with $U_{\R}(\om)$, has image contained in $S$, and acts as the identity on $S$. Thus, $e_S'\in \End_{U_{\R}(\om)}(W)$ is an idempotent with image $S$.
\end{proof}

\begin{remark}
It now follows that every type-$\one$ $U_{\R}(\om)$-module is self-dual, assuming that $\R\in \{\kk, \K\}$.
\end{remark}

Let $\R\in \{\kk, \A, \K\}$. For a type-$\one$ $U_{\R}(\om)$-module $W$, we can restrict to obtain a type-$\one$ $U_{\R}(\som)$-module. In particular, for $\wta\in X(\mathfrak{so}_m)$, we have $W[\wta]$. For $\epsilon\in \{\pm 1\}$ and $\wta\in X(\som)$ such that $\sigma(\wta) = \wta$, define
\[
W[\wta, \epsilon]:=\{w\in W[\wta] \ | \ \sigma(w) = \epsilon\cdot w\}.
\]
Note that if $\sigma(\wta) \ne \wta$, then $\sigma$ acts on $W[\wta]\oplus W[\sigma(\wta)]$ as $\dim W[\wta]$ copies of the regular representation of $\langle \sigma\rangle \cong \mathbb{Z}/2$.

\begin{definition}
Let $\chi(O(m))$ denote the free $\mathbb{Z}$-module with basis $\{e^{(\wta, \epsilon)}\}_{\substack{\wta\in X(\om), \epsilon = \pm 1 \\ \sigma(\wta) = \wta}}\cup \{e^{\wta}\}_{\substack{\wta\in X(\om) \\ \sigma(\wta) \ne \wta}}$. We define the \emph{formal character} of $W$ to be the expressions
\[
\mathrm{ch}(W):=\sum_{\substack{\wta\in X(\som) \\ \sigma(\wta) = \wta}}\dim W[\wta, \epsilon] e^{(\wta, \epsilon)} + \sum_{\substack{\wta\in X(\som)\\ \sigma(\wta) \ne \wta}} \dim W[\wta] e^{\wta}\in \chi(O(m)).
\]
\end{definition}

\begin{lemma}\label{L:o-linearindependence}
Let $\R\in \{\kk, \K\}$. The formal characters of the irreducible representations in Proposition \ref{P:o-list-of-simples} form a basis for the span of formal characters of all type-$\one$ representations.
\end{lemma}
\begin{proof}
Use Lemma \ref{L:so-linearindependence} and keep track of $\pm 1$ eigenspaces of $\sigma$.
\end{proof}

\begin{proposition}\label{P:o-character-determines-module}
Let $\R\in \{\kk, \K\}$. The character of a type-$\one$ $U_{\R}(\om)$-module determines the isomorphism class of the representation.
\end{proposition}
\begin{proof}
Use Proposition \ref{P:o-list-of-simples}, Lemma \ref{L:o-linearindependence}, and Lemma \ref{L:o-completely-reducible}.
\end{proof}

\begin{lemma}\label{L:dim1}
Suppose that $V,W$ are type-$\textbf{1}$ $U_{\A}(\om)$-modules. Then
\[
\dim_{\kk} \Hom_{U_{\kk}(\om)}(\kk\otimes V, \kk\otimes W) = \dim_{\K} \Hom_{U_{\K}(\om)}(\K\otimes V, \K\otimes W).
\]
\end{lemma}
\begin{proof}
If $U$ is a type-$\one$ $U_{\A}(\om)$-module and $\wta\in X(\som)$ such that $\sigma(\wta) = \wta$, then 
\[
\dim_{\kk} \kk\otimes U[\wta, \epsilon] = \rk_{\A}U[\wta, \epsilon] = \dim_{\K}\K\otimes U[\wta, \epsilon].
\]
The result then follows from Proposition \ref{P:o-character-determines-module} and Lemma \ref{L:o-Schurslemma}
\end{proof}

\begin{lemma}\label{L:k-tensor-U-over-A}
\[
\kk\otimes U_{\A}(\om)/(K_{\alpha}-1, \alpha\in \Delta)\cong U_{\kk}(\om) \ \ \ \ \ \text{and} \ \ \ \ \ \K\otimes U_{\A}(\om)\cong U_{\K}(\om). 
\]
\end{lemma}
\begin{proof}
The first isomorphism follows from \cite[Proposition 9.2.3]{CP-guide-to-quantum-groups}. The second isomorphism is clear.
\end{proof}

\begin{lemma}\label{L:hom-sub-UA-vsHom-sub-AtimesU}
Let $\R\in \{\kk, \K\}$. If $V$ and $W$ are type-$\textbf{1}$ representations, then 
\[
\Hom_{\R\otimes U_{\A}(\om)}(\R\otimes V, \R\otimes W) = \Hom_{U_{\R}(\om)}(\R\otimes V, \R\otimes W). 
\]
\end{lemma}
\begin{proof}
The action of $\kk\otimes U_{\A}(\om)$ on $V$ and $W$ factors through $U_{\kk}(\om)\cong \kk\otimes U_{\A}(\om)/(K_{\alpha}-1, \alpha\in \Delta)$. The claim then follows from Lemma \ref{L:k-tensor-U-over-A}
\end{proof}

\begin{lemma}\label{L:rephombasechange}
Suppose that $V, W$ are type-$\textbf{1}$ $U_{\A}(\om)$-modules. Let $\R\in \lbrace \kk, \K\rbrace$. Then there is an $\R$-linear map

\begin{align*}
\smbase_{\R}:\R\otimes \Hom_{U_{\A}(\om)}(V, W) &\rightarrow \Hom_{U_{\R}(\om)}(\R\otimes_{\A} V, \R\otimes_{\A} W), \\
 1\otimes f &\mapsto \big(1\otimes v \mapsto 1\otimes f(v) \big).
\end{align*}

\end{lemma}
\begin{proof}
Since $V$ and $W$ are finitely generated free $\A$-modules, so is $\Hom_{\A}(V, W)$. Therefore, 
\[
\R\otimes \Hom_{\A}(V, W)\cong \Hom_{\R\otimes \A}(\R\otimes V, \R\otimes W).
\]
We obtain a map 
\[
\base_{\R}: \R\otimes \Hom_{U_{\A}(\om)}(V, W)\rightarrow \R\otimes\Hom_{\A}(V, W)\xrightarrow{\cong} \Hom_{\R\otimes \A}(\R\otimes V, \R\otimes W),
\]
and it is routine to verify that the image is contained in $\Hom_{\R\otimes U_{\A}(\om)}(\R\otimes V, \R\otimes W)$. The claim then follows from Lemma \ref{L:hom-sub-UA-vsHom-sub-AtimesU}.
\end{proof}

\begin{remark}\label{R:K-is-flat-A-mod}
In general, if $f:\A^m\rightarrow \A^n$ is injective, then $\K\otimes f: \K\otimes \A^m\rightarrow \K\otimes \A^n$ is also injective. However, this may fail for $\kk\otimes_{\A} (-)$. For example the endomorphism of $\A$ given by multiplication by $q-1$ is injective, but becomes zero after applying the functor $\kk\otimes_{\A}(-)$.
\end{remark}

\subsection{Quantum vector representation}\label{ssec:quantum-vector-rep}

The defining vector representation of the orthogonal group has a type-$\textbf{1}$ $q$-analogue, which we now recall. 

\begin{notation}
Fix $m\in \mathbb{Z}_{\ge 1}$. Let $n$ be such that $m=2n$, if $m$ is even, and $m=2n+1$, if $m$ is odd.
\end{notation}

\begin{definition}\label{D:action-on-V}
Let $V_{\K}$ be the $U_{\K}(\mathfrak{so}_{m})$-module with basis:
\[
\begin{cases}
a_1, a_2, \ldots, a_n, u, b_n, \ldots, b_2, b_1 \quad \text{if $m=2n+1$} \\
a_1, a_2, \ldots, a_n, b_n, \ldots, b_2, b_1 \quad \text{if $m=2n$},
\end{cases}
\]
such that for $i=1, \dots, n-1$
\[
F_i\cdot a_i = a_{i+1}, \quad F_i\cdot b_{i+1}= b_i, 
\]
\[
E_i\cdot a_{i+1} = a_i, \quad E_i\cdot b_i= b_{i+1},
\] 
\[
\begin{cases}
F_n\cdot a_n = u, \quad  F_n\cdot u = (q+q^{-1})b_n, \quad E_n\cdot u= (q+q^{-1})a_n, \quad  E_n\cdot b_n = u, \qquad \text{if $m=2n+1$},\\
F_n\cdot a_n = b_n, \quad E_n\cdot b_n = a_n, \qquad \text{if $m=2n$},
\end{cases}
\]
and
\[
K_{\alpha}v= (q^2)^{(\alpha, \wt v)}, \quad \text{where} \ \wt(a_i) = \epsilon_i,\  \wt(u) = 0, \ \text{and} \ \wt(b_i) = -\epsilon_i. 
\]
\end{definition}

\begin{notation}
Write $V_{\A}$ to denote the $\A$ span of the given basis for $V_{\K}$. This is a free $\A$-module of rank $m$.
\end{notation}

\begin{lemma}
The algebra $U_{\A}(\som)$ preserves the $\A$-module $V_{\A}$.
\end{lemma}
\begin{proof}
All the higher divided power operators: $E_{k}^{(d)}$ and $F_{k}^{(d)}$, for $k=1, \dots, n$, and $d\ge 2$, act as zero on $V_{\K}$, except if $m=2n+1$, when $F_n^{(2)}a_n= b_n$ and $E_n^{(2)}b_n =a_n$.
\end{proof}

\begin{remark}
The algebra $U_{\kk}(\som) \cong  \kk\otimes U_{\A}(\som)/(K_{\alpha}-1)$ acts on  $V_{\kk}:=\kk\otimes V_{\A}$. 
\end{remark}

\begin{remark}
In the notation of Section \ref{defn:quantumgp}, we have $V_{\R}\cong L_{\R}(\varpi_1)$, for $\R\in \{\kk, \K\}$.
\end{remark}

\begin{lemma}\label{lem:sigmaaction}
Setting $\sigma\cdot a_1= (-1)^ma_1$, induces an action of $U_{\R}(\om)$ on $V_{\R}$, for $\R\in \{\kk, \A, \K\}$.
\end{lemma}
\begin{proof}
Let $v\in V_{\R}$. Then there is $X_v\in U_{\R}(\som)$ such that $v = X_v\cdot a_1$. Define $\sigma\cdot v = \sigma(X_v)\cdot (\sigma\cdot a_1)$. This determines a well-defined $U_{\R}(\om)$ action if $\sigma\cdot(\sigma \cdot v) = v$ and $\sigma\cdot(X \cdot (\sigma \cdot v)) = \sigma(X)\cdot v$, for all $X\in U_{\R}(\som)$ and $v\in V_{\R}$. We check the first equality:
\[
\sigma\cdot(\sigma \cdot v) = \sigma \cdot (\sigma(X_v)\cdot (\sigma\cdot a_1)) = (-1)^m \sigma\cdot (\sigma(X_v)\cdot a_1) = (-1)^m \sigma(\sigma(X_v))\cdot (\sigma \cdot a_1) = X_v\cdot a_1 = v,
\]
and the second equality:
\begin{align*}
\sigma\cdot(X\cdot (\sigma\cdot v)) &= \sigma \cdot (X\cdot (\sigma(X_v)\cdot (\sigma\cdot a_1))) = (-1)^m\sigma \cdot (X\cdot(\sigma(X_v)\cdot a_1)) \\
&=(-1)^m\sigma \cdot (\sigma^2(X)\cdot(\sigma(X_v)\cdot a_1)) =(-1)^m\sigma \cdot ((\sigma^2(X)\sigma(X_v))\cdot a_1) \\
&= (-1)^m\sigma \cdot (\sigma\big(\sigma(X)X_v\big)\cdot a_1) = (-1)^m((\sigma(X)X_v)\cdot (\sigma \cdot a_1)) \\
&=(\sigma(X)X_v)\cdot a_1 = \sigma(X)\cdot (X_v\cdot a_1) = \sigma(X)\cdot v.
\end{align*}
\end{proof}

\begin{remark}\label{rem:sigmaact}
We have the following explicit description of $\sigma$'s action on $V_\R$. 
\[
\begin{cases}
\sigma\cdot a_i =-a_i, \ \text{for $i\le n$}, \quad \sigma\cdot u =-u, \quad \text{and}\quad \sigma \cdot b_i = -b_i, \ \text{for $i\le n$}, \qquad \text{if $m=2n+1$}, \\
\sigma \cdot a_i = a_i, \ \text{for $i< n$}, \quad \sigma\cdot a_n = b_n, \ \sigma\cdot b_n=a_n, \quad \text{and} \quad \sigma \cdot b_i = b_i, \ \text{for $i< n$}, \qquad \text{if $m=2n$}.
\end{cases}
\]
\end{remark}

\begin{remark}
In the notation of Proposition \ref{P:o-list-of-simples}, the representation $V_{\R}$ is isomorphic to $L_{\R}(\varpi_1, -1)$, if $m$ is odd, and $L_{\R}(\varpi_1, +1)$, if $m$ is even. The reason for the choice of sign becomes apparent in the next section. It is to ensure that $U_{\R}(\om)$ acts on the exterior algebra by algebra automorphisms, making the algebra structure maps $U_{\R}(\om)$-module homomorphisms, and that $U_{\R}(\om)$ acts on the top exterior power as $\det_{\R}$.
\end{remark}

\subsection{Quantum exterior algebra}\label{ssec:quantum-exterior-algebra}

The usual exterior algebra $\Lambda^{\bullet}(V_{\kk})$ is defined as the quotient of the tensor algebra of $V_{\kk}$ by the two sided ideal generated by the symmetric tensors $S^2(V_{\kk})$. As a module over the special orthogonal group $SO(V_{\kk})$, we find that $S^2(V_{\kk})$ contains a copy of the trivial module, corresponding to the symmetric form preserved by $SO(V_{\kk})$. The complement of the trivial module in $S^2(V_{\kk})$ is the irreducible module $L_{\kk}(2\varpi_1)$.

From this perspective, we see that to define a $q$-analogue of the exterior algebra, we first need to find the $q$-analogue of the symmetric square. Moreover, this can be done by decomposing $V_{\K}^{\otimes 2}$ into irreducible submodules and defining the symmetric square to be the submodule of $V_{\K}^{\otimes 2}$ which is isomorphic to $L_{\K}(2\varpi_1)\oplus L_{\K}(0)$.

\begin{lemma}
The $\K$-span of the vectors
\[
a_i\otimes a_i, \ \ \ \ \ b_i\otimes b_i,
\]
\[
a_i\otimes a_j + q^{-2}a_j\otimes a_i \ \ \ \ \ i<j,
\]
\[
b_j\otimes b_i +q^{-2}b_i\otimes b_j \ \ \ \ \ i< j,
\]
\[
a_i\otimes b_j +q^{-2}b_j\otimes a_i \ \ \ \ \ i\ne j,
\]
\[
a_i\otimes u + q^{-2} u\otimes a_i, \ \ \ \ \ u\otimes b_i + q^{-2}b_i\otimes u,
\]
\[
a_i\otimes b_i + q^{-2}b_{i+1}\otimes a_{i+1} + q^{-2}a_{i+1}\otimes b_{i+1} + q^{-4}b_i\otimes a_i  \ \ \ \ \  i< n,
\]
\[
a_n\otimes b_n + q^{-4}b_n\otimes a_n + q^{-1}u\otimes u \ \ \ \ \ \text{if $m$ is odd},
\]
is the $U_{\K}(\om)$ submodule $L_{\K}(2\varpi_1)\subset V_{\K}^{\otimes 2}$ generated by $a_1\otimes a_1$.

The $\K$-span of the vector
\[
\frac{(-q^2)^{n}}{q+q^{-1}}u\otimes u + \sum_{i=1}^n\left((-q^2)^{i-1}a_i\otimes b_i-(-q^2)^{2n-i}b_i\otimes a_i\right) \ \ \ \ \ \text{if $m$ is odd}, 
\]
or the vector
\[
\sum_{i=1}^n\left((-q^2)^{i-1}a_i\otimes b_i+(-q^2)^{2n-i-1}b_i\otimes a_i\right)  \ \ \ \ \ \text{if $m$ is even}.
\]
is the unique copy of the trivial submodule of $V_{\K}^{\otimes 2}$. 

Thus, the $\K$-span of all the vectors listed in this Lemma (keeping track of whether $m$ is odd or even) is the submodule $L_{\K}(2\varpi_1)\oplus L_{\K}(0)\subset V_{\K}^{\otimes 2}$.
\end{lemma}
\begin{proof}
We leave it to the reader to use Definition \ref{D:action-on-V} and \EQ \eqref{E:comult} to check this claim. 
\end{proof}

\begin{remark}
The braiding endomorphism of $V_{\K}^{\otimes 2}$ acts on $L_{\K}(2\varpi_1)$ as $q^2$, on $L_{\K}(0)$ as $q^{2-2m}$, and on $L_{\K}(\varpi_2)$ as $-q^{-2}$ \cite[Equation 6.12]{LZ-stronglymultifree}. It follows that $L_{\K}(2)\oplus L_{\K}(0)$ is equal to the subspace of ``positive" eigenvectors for the braiding. This subspace is also referred to as $S_q^{2}$, in \cite{BZ-braidedalgebras}.
\end{remark}

\begin{definition}\label{D:exterior-algebra}
Define $\Lambda_{\A}^{\bullet}$ to be the associative $\A$-algebra generated by the elements
\[
a_1, \ldots, a_n, u, b_n, \ldots, b_1
\]
subject to the following relations:
\[
a_i^2=0 \ \ \ \ \ b_i^2=0,
\]
\[
a_ja_i= -q^2a_ia_j \ \ \ \ \ i<j,
\]
\[
b_ib_j= -q^2b_jb_i \ \ \ \ \ i< j,
\]
\[
b_ja_i= -q^2a_ib_j \ \ \ \ \ i\ne j,
\]
\[
ua_i= -q^2a_iu, \ \ \ \ \ b_iu= -q^2ub_i,
\]
\[
b_{i+1}a_{i+1} = -a_{i+1}b_{i+1} - \left(q^{-2}b_ia_i + q^2a_ib_i\right) \ \ \ \ \  i< n,
\]
\[
b_na_n=-q^4a_nb_n - q^3u^2 \ \ \ \ \ \text{if $m$ is odd},
\]
\[
u=0 \ \ \ \ \ \text{if $m$ is even},
\]
\[
\frac{(-q^2)^{n}}{q+q^{-1}}u^2 + \sum_{i=1}^n\left((-q^2)^{i-1}a_ib_i-(-q^2)^{2n-i}b_ia_i\right) = 0 \ \ \ \ \ \text{if $m$ is odd},
\]
and
\[
\sum_{i=1}^n\left((-q^2)^{i-1}a_ib_i+(-q^2)^{2n-i-1}b_ia_i\right) = 0 \ \ \ \ \ \text{if $m$ is even}.
\]
Let $\Lambda_{\A}^k$ be the $\A$ submodule spanned by monomials of degree $k$. 
\end{definition}

\begin{lemma}\label{L:bavsabreln}
If $m=2n$ is even, then the relations
\[
b_ia_i= -a_ib_i - \sum_{k=1}^{i-1}(-q^2)^{-k+1}(q^2-q^{-2})a_{i-k}b_{i-k} \quad \text{i=1, \dots, n},
\]
are equivalent to the relations
\[
b_{i+1}a_{i+1} = -a_{i+1}b_{i+1} - \left(q^{-2}b_ia_i + q^2a_ib_i\right) \quad  i=1, \dots, n-1
\]
and
\[
\sum_{i=1}^n\left((-q^2)^{i-1}a_ib_i+(-q^2)^{2n-i-1}b_ia_i\right) = 0.
\]
If $m=2n+1$ is odd, then the relations
\[
b_ia_i= -a_ib_i - \sum_{k=1}^{i-1}(-q^2)^{-k+1}(q^2-q^{-2})a_{i-k}b_{i-k} \quad \text{i=1, \dots, n} 
\]
and
\[
u^2 = q\sum_{k=1}^{n}(-q^2)^{-k}(q^2-q^{-2})a_{n+1-k}b_{n+1-k},
\]
are equivalent to the relations
\[
\frac{(-q^2)^{n}}{q+q^{-1}}u^2 + \sum_{i=1}^n\left((-q^2)^{i-1}a_ib_i-(-q^2)^{2n-i}b_ia_i\right) = 0,
\]
\[
b_{i+1}a_{i+1} = -a_{i+1}b_{i+1} - \left(q^{-2}b_ia_i + q^2a_ib_i\right) \quad  i=1, \dots, n-1,
\]
and
\[
b_na_n=-q^4a_nb_n - q^3u^2.
\]
\end{lemma}

\begin{proof}
We provide an example calculation in each case. Generalizing these to the general case is left to the reader. 

Suppose $m=2\cdot 2$. We are tasked with showing that 
\begin{equation}\label{E:even-reln-one}
b_1a_1=-a_1b_1 \quad \text{and} \quad b_2a_2=-a_2b_2-(q^2-q^{-2})a_1b_1
\end{equation}
is equivalent to 
\begin{equation}\label{E:even-reln-two}
b_2a_2=-a_2b_2-(q^{-2}b_1a_1+q^2a_1b_1) \quad \text{and} \quad a_1b_1 +q^4b_1a_1 -q^2a_2b_2 -q^2b_2a_2 = 0.
\end{equation}
Assume \EQ \eqref{E:even-reln-one}, then
\[
b_2a_2+a_2b_2+q^{-2}b_1a_1+q^2a_1b_1 = -a_2b_2-(q^2-q^{-2})a_1b_1 + a_2b_2 -q^{-2}a_1b_1 +q^2a_1b_1 = 0
\]
and
\[
a_1b_1 +q^4b_1a_1 -q^2a_2b_2 -q^2b_2a_2 = a_1b_1 - q^4a_1b_1 -q^2a_2b_2 -q^2\left(-a_2b_2-(q^2-q^{-2})a_1b_1\right).
\]
Assume \EQ \eqref{E:even-reln-two}. First, we rewrite the second relation as
\[
b_1a_1= -q^{-4}a_1b_1 + q^{-2}a_2b_2 + q^{-2}b_2a_2,
\]
so the first relation becomes
\[
b_2a_2=-a_2b_2-(q^{-2}b_1a_1+q^2a_1b_1) = -a_2b_2-q^{-2}\left(-q^{-4}a_1b_1 + q^{-2}a_2b_2 + q^{-2}b_2a_2\right) - q^2a_1b_1
\]
which we rewrite as
\[
q^{-2}(q^2+q^{-2})b_2a_2 = -q^{-2}(q^2+q^{-2})a_2b_2 -q^{-2}(q^{4}-q^{-4})a_1b_1. 
\]
Since $q^2+q^{-2}\in \A^{\times}$ and $q^2+q^{-2} = (q^{4}-q^{-4})/(q^2-q^{-2})$, this implies
\[
b_2a_2 = -a_2b_2 - (q^2-q^{-2})a_1b_1.
\]
Now, rewriting the second relation in \EQ \eqref{E:even-reln-two} again we find
\[
b_1a_1 =-q^{-4}a_1b_1 + q^{-2}a_2b_2 + q^{-2}\left(-a_2b_2 - (q^2-q^{-2})a_1b_1\right) = -a_1b_1.
\]

Suppose $m=2\cdot 1 + 1$. We need to show that
\begin{equation}\label{E:odd-relation-one}
b_1a_1 = -a_1b_1 \quad \text{and} \quad u^2 = -q^{-1}(q^2-q^{-2})a_1b_1
\end{equation}
is equivalent to 
\begin{equation}\label{E:odd-relation-two}
\frac{-q^2}{q+q^{-1}}u^2 + a_1b_1 + q^{2}b_1a_1 = 0 \quad \text{and} \quad b_1a_1 = -q^4a_1b_1 - q^3u^2.
\end{equation}
Assume \EQ \eqref{E:odd-relation-one}. Then
\begin{align*}
\frac{-q^2}{q+q^{-1}}u^2 + a_1b_1 + q^{2}b_1a_1 &= \frac{-q^2}{q+q^{-1}}\left(-q^{-1}(q^2-q^{-2})a_1b_1\right) + a_1b_1 - q^{2}a_1b_1 \\
&= q(q-q^{-1})a_1b_1 + a_1b_1 -q^2a_1b_1 = 0
\end{align*}
and
\[
b_1a_1 + q^4a_1b_1 + q^3u^2 = -a_1b_1 +q^4a_1b_1 +q^3\left(-q^{-1}(q^2-q^{-2})a_1b_1\right) = 0.
\]
Assume \EQ \eqref{E:odd-relation-two}. Then we can rewrite the first relation as
\[
u^2 = q^{-2}(q+q^{-1})a_1b_1 + (q+q^{-1})b_1a_1
\]
so the second relation becomes
\[
b_1a_1 = -q^4a_1b_1 - q^3u^2 = -q^4a_1b_1 - q^3\left(q^{-2}(q+q^{-1})a_1b_1 + (q+q^{-1})b_1a_1\right)
\]
which we can rewrite as
\[
(q^4+q^2+1)b_1a_1 = (-q^4-q^2-1)a_1b_1. 
\]
Since $q^4+q^2+1\in \A^{\times}$, it follows that $b_1a_1 = -a_1b_1$. Thus, we can rewrite the first relation in \EQ \eqref{E:odd-relation-two} again as
\[
u^2 = q^{-2}(q+q^{-1})a_1b_1 - (q+q^{-1})a_1b_1 =-q^{-1}(q^2-q^{-2})a_1b_1.
\]
\end{proof}

\begin{corollary}\label{C:exterior-algebra-confluentpresentation}
The algebra $\Lambda_{\A}^{\bullet}$ is the associative $\A$-algebra generated by the elements
\[
a_1, \ldots, a_n, u, b_n, \ldots, b_1
\]
subject to the following relations:
\[
a_i^2=0 \ \ \ \ \ b_i^2=0,
\]
\[
a_ja_i= -q^2a_ia_j \ \ \ \ \ i<j,
\]
\[
b_ib_j= -q^2b_jb_i \ \ \ \ \ i< j,
\]
\[
b_ja_i= -q^2a_ib_j \ \ \ \ \ i\ne j,
\]
\[
ua_i= -q^2a_iu, \ \ \ \ \ b_iu= -q^2ub_i,
\]
\[
b_ia_i= -a_ib_i - \sum_{k=1}^{i-1}(-q^2)^{-k+1}(q^2-q^{-2})a_{i-k}b_{i-k},
\]
\[
u=0 \ \ \ \ \ \text{if $m$ is even},
\]
\[
u^2 = q\sum_{k=1}^{n}(-q^2)^{-k}(q^2-q^{-2})a_{n+1-k}b_{n+1-k} \ \ \ \ \ \text{if $m$ is odd}.
\]
\end{corollary}
\begin{proof}
This follows immediately from Definition \ref{D:exterior-algebra} and Lemma \ref{L:bavsabreln}.
\end{proof}

\begin{definition}
It will be convenient to write the vectors in $V_{\A}$ as follows:
\[
v_i:=a_i, \quad \text{for} \quad  i=1, \ldots, n,
\]
\[
v_{m-i+1}:=b_i, \quad \text{for} \quad i=1, \ldots, n,
\]
and if $m=2n+1$, then 
\[
v_{n+1}:=u.
\]
There is an involution of $\lbrace1, \ldots, m\rbrace$ defined by $i\mapsto i':= m-i+1$. We will write $v_{i'}:= v_{m-i+1}$. Let $S\subset \lbrace 1, \ldots, m\rbrace$. If $S= \lbrace s_1, \ldots, s_k\rbrace$ such that $s_1< \ldots< s_k$, set $v_S:=v_{s_1}\cdots v_{s_{k}}$. We extend the involution $i\mapsto i'$ to the set of subsets of $\lbrace 1, \ldots, m\rbrace$ by $S\mapsto S':=\lbrace s_1', \ldots, s_k'\rbrace$. Note that $v_{S'}:= v_{s_{k}'}\cdots v_{s_{1}'}$. 
\end{definition}

\begin{thm}\label{T:basis-for-LambdaA}
The set $\lbrace v_S \rbrace_{S\subset \lbrace 1, \ldots, m\rbrace}$ forms a basis of $\Lambda_{\A}^{\bullet}$. In particular, $\lbrace v_S\rbrace_{\substack{S\subset \lbrace 1, \ldots, m\rbrace \\ |S|=k}}$ forms a basis of $\Lambda_{\A}^k$. 
\end{thm}
\begin{proof}
This is a standard application of Bergman's diamond lemma \cite[Theorem 1.2]{Ber-diamond}. Note that if $1< \dots < m$, then the lexicographic order on monomials in $v_i$ satisfies the hypotheses of the diamond lemma, and the irreducible monomials are the elements of $\{v_S\}_{S\subset \{1, \dots, m\}}$. Therefore, it suffices to show that all the ambiguities in the defining relations are resolvable.

The following are all the overlap ambiguities in the defining relations of $\Lambda_{\A}^{\bullet}$:
\[
a_xa_xa_x, \quad a_ja_ja_i, \quad b_xb_xb_x, \quad b_ib_ib_j, \quad b_xb_xa_y, \quad b_xb_xu, \quad a_ja_ia_i, \quad a_ja_ia_k,
\]
\[
b_ib_jb_j, \quad b_ib_ja_i, \quad b_ib_ju, \quad b_ib_ja_j, \quad b_xa_ya_y, \quad b_xa_ya_z, 
\]
\[
ua_xa_x, \quad ua_ja_i, \quad b_xua_y, \quad b_xua_x, \quad b_xuu, \quad b_xa_xa_x, \quad b_ja_ja_i, \quad uua_x, \quad \text{and} \quad uuu,
\]
where $1\le x,y,z,i,j,k\le n$, $x\ne y$, $z<y$, and $k< i< j$.

We provide an example calculation to verify the resolution of an ambiguity, and leave the rest as an exercise. To simplify notation, write $\xi:=(q^2-q^{-2})$. On the one hand, we have
\begin{align*}
(uu)a_i &= q\sum_{k=1}^{n}(-q^2)^{-(n-k+1)}\xi a_{k}b_{k}a_i \\
&=q\sum_{1\le k< i}(-q^2)^{-n+k}\xi a_{k}a_ib_{k} - q\sum_{1\le k < i}(-q^2)^{-n+k+1}\xi^2a_{k}a_{i}b_{k} +q\sum_{i< k\le n}(-q^2)^{-n+k+1}\xi a_ia_{k}b_{k} \\
&=q\sum_{1\le k< i} (-q^2)^{-n+k}\xi\left(1 - (-q^2)\xi\right) a_{k}a_{i}b_{k} +q\sum_{i< k\le n}(-q^2)^{-n+k+1}\xi a_ia_{k}b_{k} \\
&=q\sum_{1\le k< i} (-q^2)^{-n+k+2}\xi a_{k}a_{i}b_{k} +q\sum_{i< k\le n}(-q^2)^{-n+k+1}\xi a_ia_{k}b_{k},
\end{align*}
and on the other 
\begin{align*}
u(ua_i) &= (-q^2)ua_iu (-q^2)^2a_iu^2= q\sum_{k=1}^{n}(-q^2)^{-n+k+1}\xi a_ia_{k}b_{k} \\
&= q\sum_{1\le k < i}(-q^2)^{-n+k+2}\xi a_{k}a_ib_{k} + q\sum_{i< k\le n}(-q^2)^{-n+k+1}\xi a_ia_{k}b_{k}.
\end{align*}
\end{proof}

\begin{remark}\label{R:Lambda-over-R}
Write $\Lambda_{\R}^{\bullet}$ to denote the associative $\R$-algebra with generators and relations as in Definition \ref{D:exterior-algebra}. The basis $\{v_S\}_{S\subset \{1, \dots, m\}}$ of $\Lambda^{\bullet}_{\A}$ is also a basis for $\Lambda_{\R}^{\bullet}$, when $\R\in \{\kk, \K\}$. Define isomorphism $\iota_{k}:\Lambda_{\R}^k\rightarrow \R\otimes \Lambda_{\A}^k$ by $v_S\mapsto 1\otimes v_S$. 
\end{remark}

\subsection{Fundamental category for orthogonal groups}\label{ssec:fund-O}

For simple Lie algebras, the fundamental category is the monoidal category generated by irreducible representations with highest weight a fundamental weight. We define an analogue of this category for $U_{\R}(\om)$, when $\R\in \{\K, \A, \kk\}$. The first step is to show that for $k=0, 1, \dots, m$, $\Lambda_{\R}^k$ are non-isomorphic, irreducible, self-dual $U_{\R}(\om)$-modules.  

\begin{lemma}
The action of $U_{\K}(\som)$ on the tensor algebra of $V_{\K}$ descends to an action of $U_{\K}(\som)$ on $\Lambda_{\K}^{\bullet}$. Moreover, the multiplication for $\Lambda_{\K}^{\bullet}$ is $U_{\K}(\som)$ equivariant.
\end{lemma}
\begin{proof}
This follows from observing that $U_{\K}(\som)$ preserves the defining relations. See the discussion in \cite[Sections 3.2, 4.1]{LZZ-invarianthteory} for the symmetric analogue. 
\end{proof}

\begin{definition}
Let $\wt v_S:= \sum_{i\in S}\wt v_i$, where $\wt v_i$ is as defined in Definition \ref{D:action-on-V}. Then $K_{\alpha}\cdot v_S = (q^2)^{(\alpha, \wt v_S)}$.
\end{definition}

\begin{remark}\label{R:Lambda-is-typeone}
The modules $\Lambda_{\K}^k$ are finite dimensional type-$\textbf{1}$ representations of $U_{\K}(\som)$, in particular we have
\[
\Lambda_{\K}^k= \bigoplus_{\wta\in X(\som)}\Lambda_{\K}^k[\wta].
\]
In fact, this remains true over $\A$, since $\Lambda_{\A}^{\bullet}$ is spanned by $\{v_S\}_{S\subset \lbrace 1, \ldots, m\rbrace}$ and each $v_S$ is a weight vector. 
\end{remark}

\begin{lemma}\label{L:equality-of-type1-chars}
We have the following equality of formal characters:
\[
\sum_{\wta\in X(\som)} \dim\Lambda_{\K}^k[\wta]e^{\wta} = \sum_{\wta\in X(\som)} \dim\Lambda_{\kk}^k[\wta]e^{\wta}.
\]
\end{lemma}
\begin{proof}
Follows from observing that $v_S\in \Lambda^{|S|}_{\R}[\wt v_S]$ for $\R\in \{\kk, \K\}$. 
\end{proof}

\begin{lemma}\label{L:samecharacter}
Let $\R\in \{\kk, \K\}$. We have the following isomorphisms of $U_{\R}(\som)$-modules. 

If $m=2n+1$, then
\[
\Lambda_\R^1\cong L_\R(\varpi_1), \quad  \Lambda_\R^2\cong L_\R(\varpi_2), \quad \dots, \quad \Lambda_\R^{n-1}\cong L_\R(\varpi_{n-1}),
\]
\[
\text{and} \quad 
\Lambda_\R^{n}\cong L_\R(2\varpi_n).
\]

If $m=2n$, then
\[
\Lambda_\R^1\cong L_\R(\varpi_1), \quad \Lambda_\R^2\cong L_\R(\varpi_2), \quad \dots, \quad \Lambda_\R^{n-2}\cong L_\R(\varpi_{n-2}),
\]
\[
\Lambda_\R^{n-1}\cong L_\R(\varpi_{n-1}+ \varpi_n), \quad \text{and} \quad \Lambda_\R^n\cong L_\R(2\varpi_{n-1})\oplus L_\R(2\varpi_n).
\]

In either case $\Lambda_\R^m\cong L_\R(0)$
\end{lemma}

\begin{proof}
Thanks to Lemma \ref{L:so-completely-reducible}, it suffices to compute characters, and the result follows for $\R= \kk$ from \cite[Sections 19.2, 19.4]{Fultonharris}, and then for $\R= \K$, from Lemma \ref{L:equality-of-type1-chars}. 
\end{proof}

\begin{lemma}
The algebra $U_{\A}(\som)$ preserves the lattice $\Lambda_{\A}^{\bullet}\subset \Lambda_{\K}^{\bullet}$. 
\end{lemma}
\begin{proof}
Since $U_{\A}(\som)$ preserves $V_{\A}^{\otimes d}\subset V_{\K}^{\otimes d}$ for all $d\ge 0$, it suffices to show that $U_{\A}(\som)$ preserves the $\A$-span of the defining relations for $\Lambda_{\A}^{\bullet}$. This then immediately reduces to verifying that $E_{\alpha}^{(2)}$ and $F_{\alpha}^{(2)}$ preserve the $\A$-span of the defining relations. For example, using that $\Delta(F)= 1\otimes F + F\otimes K^{-1}$, we find:
\[
F_1^{(2)}\cdot a_1\otimes a_1 = \frac{1}{q^2+q^{-2}}F_1\cdot (a_1\otimes a_2 + q^{-2}a_2\otimes a_1) = \frac{1}{q^2+q^{-2}}(q^2a_2\otimes a_2  + q^{-2}a_2\otimes a_2) = a_2\otimes a_2.
\]
The remaining calculations we leave to the reader.
\end{proof}

\begin{lemma}
Let $\R\in \{\kk, \A, \K\}$. The operator $\sigma\in U_{\R}(\om)$ acts via the coproduct on $V_{\R}^{\otimes d}$, for all $d\ge 0$, and preserves the defining relations of $\Lambda_{\R}^{\bullet}$. Thus, there is an action of $U_{\R}(\om)$ on $\Lambda_{\R}^{\bullet}$, and the algebra structure maps are $U_{\R}(\om)$ equivariant.
\end{lemma}
\begin{proof}
Use the description of $\sigma$'s action on $V_\R$ in Remark \ref{rem:sigmaact} to verify that $\sigma$ preserves the relations in Definition \ref{D:exterior-algebra} of $\Lambda_\R^{\bullet}$. 
\end{proof}

\begin{proposition}\label{L:ext-powers-irredicible}
Let $\R\in \{\kk, \K\}$. The $U_\R(\om)$-module $\Lambda_\R^k$ is self-dual and irreducible for $k=0, \dots, m$, if $\Lambda_\R^i\cong \Lambda_\R^j$, then $i=j$, we have the following isomorphisms of $U_{\R}(\om)$-modules\footnote{In the notation of Proposition \ref{P:o-list-of-simples}. Writing $\varpi_0=0$.}. 

If $m=2n+1$, then
\[
\Lambda_{\R}^i\cong L_{\R}(\varpi_i, (-1)^{i}), \quad \Lambda_{\R}^{m-i}\cong L_{\R}(\varpi_i, -(-1)^i), \quad \text{for $i=0, 1, \dots, n-1$},
\]
\[ 
\Lambda_\R^{n}\cong L_\R(2\varpi_n, (-1)^n), \quad \text{and} \quad  \Lambda_{\R}^{n+1}\cong L_{\R}(2\varpi_{n}, -(-1)^{n}).
\]

If $m=2n$, then
\[
\Lambda_{\R}^i\cong L_{\R}(\varpi_i, +1), \quad \Lambda_{\R}^{m-i} \cong L_{\R}(\varpi_i, -1), \quad \text{for $i=0,1, \dots, n-2$}, 
\]
\[
\Lambda_{\R}^{n-1}\cong L_{\R}(\varpi_{n-1}+ \varpi_n, +1), \quad \Lambda_{\R}^{n+1}\cong L_{\R}(\varpi_{n-1}+ \varpi_n, -1), \quad \text{and}
\]
\[
\Lambda_\R^n\cong U_{\R}(\om)\otimes_{U_{\R}(\som)}L_\R(2\varpi_{n-1}).
\]
\end{proposition}

\begin{proof}
Self duality is from Lemma \ref{L:selfduality-of-O-reps}. Thanks to Proposition \ref{P:o-character-determines-module}, and Lemma \ref{L:samecharacter}, the remaining claims follow once we show that $\sigma$ acts on $v_1v_1\dots v_i$ by the prescribed eigenvalue in the statement of the Proposition. We will argue this for $v_1v_2\dots v_m$, where $\sigma$ always acts by $-1$, leaving the other cases to the reader. Using the coproduct from Definition \ref{D:Uo}, we find $\sigma(v_1\dots v_m) = \sigma(v_1)\dots \sigma(v_m)$. From Remark \ref{rem:sigmaact} we see that for $m=2n+1$, 
\[
\sigma(v_1v_2\dots v_m) = (-v_1)(-v_2)\dots (-v_m) = -v_1v_2\dots v_m.
\]
For $m=2n$,
\[
\sigma(v_1v_2\dots v_m) = v_1\dots v_{n-1}v_{n+1}v_{n}v_{n+2}\dots v_{2n} =-v_1v_2\dots v_m,
\]
where the last equality follows from Lemma \ref{L:bavsabreln} and the defining relations of $\Lambda_\R^{\bullet}$.
\end{proof}

\begin{remark}
Let $\R\in \{\kk, \K\}$. There is an isomorphism of $U_\R(\om)$-modules $\Lambda_\R^m\cong \det_\R$, and isomorphisms of $U_\R(\om)$-modules $\Lambda_\R^i\cong \det_\R\otimes\Lambda_\R^{m-i}$. 
\end{remark}

\begin{remark}
Since $\A$ is not a field, it does not make sense to ask for $\Lambda_{\A}^i$ to be irreducible. However, we will show, in Lemma \ref{L:psi-an-iso}, that $\Lambda_{\A}^i$ is a self-dual $U_{\A}(\om)$-module, essentially by proving that there is an isomorphism $\Lambda_{\K}^i\rightarrow (\Lambda_{\K}^i)^*$ which preserves $\Lambda_{\A}^i$ and $(\Lambda_{\A}^i)^*$.
\end{remark}

\begin{lemma}\label{L:pieri}
Let $\R\in \{\kk, \K\}$, then we have the following tensor product decompositions.
\[
\Lambda_{\R}^k\otimes V_{\R} \cong \begin{cases}
L_\R(\varpi_1 + \varpi_k, +1)\oplus \Lambda_\R^{k+1}\oplus \Lambda_\R^{k-1}, \quad \text{if} \quad k\le m-1, \quad \text{and}\\
\Lambda_\R^{m-1}, \quad \text{if} \quad k=m.
\end{cases}
\]
\end{lemma}
\begin{proof}
Standard result (up to tensoring with $\det_{\R}$). See \cite{KoikeTerada} and \cite[Equation 6.1]{TubaWenzl}.
\end{proof}

\begin{definition}\label{D:def-of-fund}
Let $\R\in \{\K, \A, \kk\}$. Define the monoidal category $\Fund$ to be the full monoidal subcategory of $U_\R(\om)$-modules generated by $\Lambda_\R^k$ for $k=0, \ldots, m$. Define $\StdRep$ as the full monoidal subcategory of $U_\R(\om)$-modules generated by $\Lambda_\R^1=V_\R$. 
\end{definition}

Let $\wtk = (\wtk_1, \ldots, \wtk_s)$, such that $0\le \wtk_i\le m$, for $i= 1, \ldots, s$. We write $\Lambda^{\wtk}_\R:=\Lambda^{\wtk_1}_\R\otimes \cdots \otimes\Lambda_\R^{\wtk_s}$. Objects in $\Fund$ are all of the form $\Lambda_\R^{\wtk}$ for some $\wtk$.

\begin{proposition}\label{P:dimhomrep}
Let $\wtk=(\wtk_1, \dots, \wtk_s)$ and $\wtl = (\wtl_1, \dots, \wtl_t)$, such that $0\le \wtk_i \le m$, for $i=1, \dots, s$, and $0\le \wtl_j\le m$, for $j=1, \dots, t$. Then 
\[
\dim_{\kk} \Hom_{U_{\kk}(\om)}(\Lambda_{\kk}^{\wtk}, \Lambda_{\kk}^{\wtl}) = \dim_{\K} \Hom_{U_{\K}(\om)}(\Lambda_{\K}^{\wtk}, \Lambda_{\K}^{\wtl}).
\]
\end{proposition}
\begin{proof}
The $U_{\R}(\om)$-module $\Lambda_{\A}^k$ is type-$\one$, for $k=0, \dots, m$, see Definition \ref{D:orthogonal-typeone} and Remark \ref{R:Lambda-is-typeone}. Tensor products of type-$\one$ modules are type $\one$, so $\Lambda_{\A}^{\wtk}$ is a type-$\textbf{1}$ $U_{\A}(\om)$-module for all $\wtk$. The result then follows from Lemma \ref{L:dim1}. 
\end{proof}

\begin{lemma}\label{L:FundA-homs-free-fin-gen}
Homomorphism spaces in $\FundA$ are free and finitely generated $\A$-modules. 
\end{lemma}
\begin{proof}
Since $\Lambda_{\A}^{\wtk}$ and $\Lambda_{\A}^{\wtl}$ are both free and finitely generated $\A$-modules, see Theorem \ref{T:basis-for-LambdaA}, the $\A$-module $\Hom_{\A}(\Lambda_{\A}^{\wtk}, \Lambda_{\A}^{\wtl})$ is free and finitely generated over $\A$. Since $\A$ is a PID, the claim follows from observing that $\Hom_{U_{\A}(\om)}(\Lambda_{\A}^{\wtk}, \Lambda_{\A}^{\wtl})\subset \Hom_{\A}(\Lambda_{\A}^{\wtk}, \Lambda_{\A}^{\wtl})$.
\end{proof}

We also define an auxiliary $\R$-linear monoidal category $\R\otimes \FundA$, which has the same objects as $\FundA$, but with morphisms
\[
\Hom_{\R\otimes \FundA}(\Lambda_{\A}^{\wtk}, \Lambda_{\A}^{\wtl}):=\R\otimes \Hom_{\FundA}(\Lambda_{\A}^{\wtk}, \Lambda_{\A}^{\wtl}).
\]

We have identifications $\iota_{\wtk}:=\iota_{\wtk_1}\otimes \dots \otimes \iota_{\wtk_s}:\Lambda_{\R}^{\wtk} \rightarrow \R\otimes \Lambda_{\A}^{\wtk}$, see Remark \ref{R:Lambda-over-R}. Using $\smbase_{\R}$ from Lemma \ref{L:rephombasechange}, we define a monoidal functor
\[
\base_{\R}:\R\otimes \FundA\rightarrow \Fund,
\]
on objects as $\Lambda_{\A}^{\wtk}\mapsto \Lambda_{\R}^{\wtk}$, and on morphisms by sending $r\otimes f\in \Hom_{\R\otimes \FundA}(\Lambda_{\A}^{\wtk}, \Lambda_{\A}^{\wtl})$ to 
\[
\iota_{\wtl}^{-1}\circ \smbase_{\R}(r\otimes f)\circ \iota_{\wtk} \in \Hom_{\Fund}(\Lambda_{\R}^{\wtk}, \Lambda_{\R}^{\wtl}).
\]

\begin{remark}
The notation makes $\base_{\R}$ appear more complicated than it is. Let $\R\in \{\kk, \A, \K\}$. The $\{v_S\}$ basis for $\Lambda_{\R}^{\bullet}$ gives rise to a basis for $\Lambda_{\R}^{\wtk}$, for all $\wtk$, and therefore a basis for $\Hom_{\R}(\Lambda_{\R}^{\wtk}, \Lambda_{\R}^{\wtl})$ for all $\wtk, \wtl$. For $f\in \Hom_{\FundA}(\Lambda_{\A}^{\wtk}, \Lambda_{\A}^{\wtl})$, we can use this basis to view $f$ as a matrix with entries in $\A$. Then for $\R\in \{\kk, \K\}$, $\base_{\R}(1\otimes f)$ is the same matrix, but with the entries interpreted as elements of $\R$. 
\end{remark}

One of our main goals is to derive various relations among morphisms in $\FundA$. However, it will be easier to work in $\FundK$, so the following lemma is useful. 

\begin{lemma}\label{L:sufficestocheckoverK}
The functor $\base_{\K}$ is faithful.
\end{lemma}
\begin{proof}
Follows from injectivity of $\smbase_{\K}$, see Remark \ref{R:K-is-flat-A-mod}.
\end{proof} 

\subsection{Generating intertwiners for tensor products of exterior powers}\label{ssec:gen-intertwiners}

\def\freewebgenone
{\begin{FGtric}
\draw [scale=0.8] (0,0)--(90:1) (0,0)--(210:1) (0,0)--(330:1);
\draw (90:0.8)node[black,above,scale=0.7]{$k+1$}
      (210:0.8)node[black,below,scale=0.7]{$k$}
      (330:0.8)node[black,below,scale=0.7]{$1$}; 
\end{FGtric}
}

\def\freewebgentwo
{\begin{FGtric}
\draw [scale=0.8] (0,0)--(90:1) (0,0)--(210:1) (0,0)--(330:1);
\draw (90:0.8)node[black,above,scale=0.7]{$k+1$}
      (210:0.8)node[black,below,scale=0.7]{$1$}
      (330:0.8)node[black,below,scale=0.7]{$k$}; 
\end{FGtric}
}

\def\VertoMultiRef
{\begin{Gtric}
\draw [scale=0.8] (0,0)--(90:1) (0,0)--(210:1) (0,0)--(330:1);
\draw (90:0.8)node[black,above,scale=0.7]{$i+j$}
      (210:0.8)node[black,below,scale=0.7]{$i$}
      (330:0.8)node[black,below,scale=0.7]{$j$}; 
\end{Gtric}
}

\def\CapKRef
{\begin{FGtric}
\draw (0,0)..controls(0.5,1)and(1,1)..(1.5,0); 
\draw (0.75,0.7)node[black,above,scale=0.7]{$k$};
\end{FGtric}
}

\def\CupKRef
{\begin{FGtric}
\draw (0,0)..controls(0.5,-1)and(1,-1)..(1.5,0); 
\draw (0.75,-0.7)node[black,below,scale=0.7]{$k$};
\end{FGtric}
}

\def\GCapKRef
{\begin{Gtric}
\draw (0,0)..controls(0.5,1)and(1,1)..(1.5,0); 
\draw (0.75,0.7)node[black,above,scale=0.7]{$k$};
\end{Gtric}
}

\def\GCupKRef
{\begin{Gtric}
\draw (0,0)..controls(0.5,-1)and(1,-1)..(1.5,0); 
\draw (0.75,-0.7)node[black,below,scale=0.7]{$k$};
\end{Gtric}
}

So far, we have studied the objects in $\Fund$, and the dimensions of homomorphism spaces between these objects, see Proposition \ref{P:dimhomrep}. Next, we will study specific morphisms in $\Fund$, namely
\begin{equation}\label{eqn:mec-maps}
{^\R m}_{i,j}^{i+j}:\Lambda_\R^i\otimes \Lambda_\R^j\longrightarrow \Lambda_\R^{i+j}, \quad {^\R e}_k:\Lambda_\R^k\otimes \Lambda_\R^k\longrightarrow \Lambda_\R^{0}, \quad \text{and} \quad
{^\R c}_k:\Lambda_\R^{0}\longrightarrow\Lambda_\R^k\otimes \Lambda_\R^k,
\end{equation}
which will end up generating all morphisms in $\Fund$. We will introduce a graphical calculus\footnote{Precisely, we define a functor from a free web category to $\Fund$. To prevent confusion with the graphical calculus for the generators and relations category $\Web$, which is a quotient of the free web category, our convention is that the graphical calculus for the free web category is apricot, and the image of web diagrams in $\Fund$ are gray.} for morphisms in $\Fund$. In this graphical calculus the maps from equation \eqref{eqn:mec-maps} are
\begin{gather*} \VertoMultiRef, \quad \GCapKRef, \quad \text{and} \quad \GCupKRef
\end{gather*}
respectively.

After we introduce these special morphisms in $\Fund$, we then establish some basic relations they satisfy with one another, as well as their relation to the braiding for the defining vector representation.

The first map in equation \ref{eqn:mec-maps} is just a graded component of the multiplication map on the quantum exterior algebra, see Definition \ref{D:mij}. The second and third maps in equation \ref{eqn:mec-maps} are defined in a sightly more subtle way. In the case $k=1$, both are defined explicitly on a basis of $\Lambda_{\R}^1\otimes \Lambda_{\R}^1$, see Definition \ref{D:e1-and-c1} and Remark \ref{R:e1-and-c1}. The maps ${^\R e}_1$ and ${^\R c}_1$ are defined in a way which ensures that they satisfy the zig-zag relation \eqref{E:zigzag}. When $k>1$, the map ${^\R c}_k$ is defined inductively (in a way which ensures that the bigon relation holds, see equation \ref{E:grey-skein-bigon}). Finally Definition \ref{D:ek} describes ${^\R e}_k$ in a way which ensures that it will satisfy the zig-zag relation with ${^\R c}_k$, see Lemma \ref{l:eczigzag}.

\begin{remark}
We do not directly define comultiplication maps $\Lambda_\R^{i+j}\longrightarrow \Lambda_\R^i\otimes \Lambda_\R^j$. As far as we are aware, an explicit description for the comultiplication on the quantum exterior algebra for $U_q(\om)$, i.e. formulas analogous to those found in \cite[Lemma 3.1.2]{CKM}, has not appeared in the literature. Instead, in Definition \ref{D:comult} we specify the comultiplication by using the cup and cap morphisms, i.e. the second and third maps in equation \ref{eqn:mec-maps}, to rotate the multiplication.
\end{remark}

\begin{definition}\label{D:mij}
Let $\R\in \lbrace \kk, \A, \K\rbrace$ and define
\[
{^\R m}_{i,j}^{i+j}:\Lambda_\R^i\otimes \Lambda_\R^j\longrightarrow \Lambda_\R^{i+j}
\]
by $x\otimes y\mapsto xy$.
\end{definition}

\begin{remark}\label{R:m-is-non-zero}
The map ${^\R m}_{i,j}^{i+j}$ is a $U_{\R}(\om)$-linear transformation such that
\[
v_{\lbrace 1, \ldots, i\rbrace}\otimes v_{\lbrace i+1, \ldots, i+j\rbrace}\mapsto v_{\lbrace 1,\ldots,i, i+1, \ldots, {i+j}\rbrace}.
\]
In particular, when $i+j\le m$, the map is non-zero and therefore is surjective. Moreover, $\base_{\R}( {^{\A}m}_{i,j}^{i+j})= {^{\R}m}_{i,j}^{i+j}$.
\end{remark}

Let $\R\in\{\kk, \A, \K\}$ and let $X$ be a free $\R$-module with basis $B_X= \{b_1, \dots, b_d\}$. Then $X^*:=\Hom_\R(X, \R)$ has basis $\{b_1^*, \dots, b_d^*\}$, where $b_i^*(b_j)= \delta_{i,j}$. Consider the elements $C\in X\otimes X^*$ defined by
\[
C=\sum_{i=1}^d b_i\otimes b_i^*. 
\]
There are $\R$-linear maps
\[
\text{coev}:\R\rightarrow X\otimes X^*, \quad 1\mapsto C, \quad \text{and} \quad \text{ev}:X^*\otimes X\rightarrow \R,
\quad 
f\otimes x\mapsto f(x).
\]
It is easy to verify that
\begin{equation}\label{E:zigzag}
(\id_X\otimes \text{ev})\circ (\text{coev}\otimes \id_X) = \id_{X} \quad \text{and} \quad (\text{ev}\otimes \id_{X^*})\circ (\id_{X^*}\otimes \text{coev}) = \id_{X^*}.
\end{equation}

\begin{remark}
We can regard $\R$ as the trivial $U_\R(\om)$-module via the counit, denoted $\epsilon$ in Definition \ref{D:Uo}. Also, If $X$ is a $U_{\R}(\om)$-module, which is free over $\R$, then $X^*$ is as well, via the antipode, denoted $S$ in Definition \ref{D:Uo}. One easily checks that the maps $\text{ev}$ and $\text{coev}$ are $U_{\R}(\om)$-module maps, where $U_{\R}(\om)$ acts on $X\otimes X^*$ via the coproduct.
\end{remark}

\begin{definition}\label{D:e1-and-c1}
Let $\R\in \lbrace \kk, \A, \K\rbrace$. Define 
\[
{^\R\varphi}_1:V_\R\longrightarrow (V_\R)^*
\]
to be the unique $U_\R(\om)$-linear map such that 
\[
v_1\mapsto v_{m}^*.
\]
This is easily seen to be an isomorphism of $\R$-modules. Also, we have $\base_{\R}({^{\A}\varphi}_1) = {^{\R}\varphi}_1$.

We define 
\[
{^\R e}_1:V_\R\otimes V_\R\longrightarrow \R
\quad\text{by}\quad \text{ev}\circ ({^\R\varphi}_1\otimes \id),
\]
and
\[
{^\R c}_1:\R\longrightarrow V_\R\otimes V_\R
\quad \text{by} \quad (\id \otimes ({^\R\varphi}_1)^{-1})\circ \text{coev}.
\]
Note that for $\R\in \lbrace \kk, \A, \K\rbrace$, we have $\base_{\R}( {^{\A}c}_k) = {^\R c}_k$ and $\base_{\R}( {^{\A}e}_k)= {^\R e}_k$.
\end{definition}

\begin{remark}\label{R:e1-and-c1}
Let $\R\in \{\kk, \A, \K\}$. Recall that $V_\R$ is generated over $U_\R(\som)^{<0}$ by $a_1$ while $V_\R^*$ is similarly generated by $b_1^*$. It then follows from explicit calculation using: $U_\R(\om)$ equivariance of ${^\R\varphi_1}$, the formulas for the antipode in Definition \ref{defn:quantumgp}, and the description of the action on $V_\R$ in Definition \ref{D:action-on-V}, that if $m=2n+1$ is odd, then
\[
{^\R\varphi_1}(a_i) = (-q^2)^{i-1}b_i^*, \quad {^\R\varphi_1}(u) = (-q^2)^{n-1}[2]_q u^*, \quad \text{and} \quad {^\R\varphi_1}(b_i) = (-q^2)^{m-i}a_i^*,
\]
and if $m=2n$ is even, then
\[
{^\R\varphi_1}(a_i) = (-q^2)^{i-1}b_i^* \quad \text{and} \quad {^\R\varphi_1}(b_i) = (-q^2)^{m-i-1}a_i^*.
\]
\end{remark}

\begin{definition}\label{D:caps}
Let $\R\in \{\kk, \A, \K\}$. Define ${^\R{c_k}}\in \Hom_{U_\R(\om)}(\R, \Lambda_\R^k\otimes \Lambda_\R^k)$  inductively by 
\[
{^\R{c_k}}:=\frac{[2]}{[2k]}\cdot ({^\R{m_{k-1,1}^k}}\otimes {^\R{m_{1,k-1}^k}})\circ (\id_{\Lambda_\R^{k-1}}\otimes {^\R c}_1\otimes \id_{\Lambda_\R^{k-1}}) \circ {^\R{c_{k-1}}}.
\]
\end{definition}

\begin{lemma}\label{L:ckne0}
Let $\R \in \{\kk, \A, \K\}$. Then ${^{\R}c}_k \ne 0$.
\end{lemma}
\begin{proof}
Since ${^{\R}c}_k = \base_{\R}({^{\A}c}_k)$ for $\R\in \{\kk, \K\}$, it suffices to show ${^{\kk}c}_k\ne 0$. 

Let $\pi_k\in \Hom_{\kk}(\Lambda_{\kk}^k\otimes \Lambda_{\kk}^k,\kk)$ be the projection to $v_1\dots v_k\otimes v_{m-k+1}\dots v_m$, with respect to the basis $\{v_S \otimes v_T\}_{\substack{ S,T\subset \{1, \dots, m\}\\ |S|=|T| =k}}$. It suffices to show that $\pi_k\circ {^{\kk}c}_k \ne 0$.

Write $({^{\kk}\varphi}_1)^{-1}(v_{j}^*)= t_jv_{j'}$, where $j'= m-k+1$. Note that $t_j\in \kk^{\times}$ for $j=1, \dots, m$. Then
\begin{align*}
\pi_k\circ {^{\kk}c}_k(1) &= \frac{1}{k!}\sum_{w\in S_k}\pi_k\left(v_{w(1)}\dots v_{w(k)}\otimes {^{\kk}\varphi}_1^{-1}(v_{w(k)})\dots {^{\kk}\varphi}_1^{-1}(v_{w(1)})\right) \\
&=\frac{t_1\dots t_k}{k!} \cdot \sum_{w\in S_k}\pi_k\left(v_{w(1)}\dots v_{w(k)}\otimes v_{w(k)'}\dots v_{w(1)'}\right) \\
&=\frac{t_1\dots t_k}{k!} \cdot \sum_{w\in S_k}\pi_k\left((-1)^{\ell(w)}v_{1}\dots v_{k}\otimes (-1)^{\ell(w)}v_{m-{k-1}}\dots v_{m}\right) \\
&=t_1\dots t_k \ne 0.
\end{align*}
\end{proof}

\begin{definition}
Let
\[
{^\R\psi}_k:=(\text{ev}\otimes \id_{\Lambda_\R^k})\circ (\id_{(\Lambda_\R^k)^*}\otimes {^\R{c_k}})
\] 
Since $\text{ev}\in \Hom_{U_\R(\om)}((\Lambda_\R^k)^*\otimes \Lambda_\R^k, \R)$, it follows that ${^\R\psi}_k\in \Hom_{U_\R(\om)}((\Lambda_\R^k)^*, \Lambda_\R^k)$. 
\end{definition}

\begin{lemma}
We have the following equality in $\Hom_{U_\R(\om)}(\R, \Lambda_\R^k\otimes \Lambda_\R^k)$
\[
{^\R{c_k}} = (\id_{\Lambda_\R^k}\otimes {^\R\psi}_k) \circ \text{coev}
\]

\end{lemma}
\begin{proof}
Follows from Equation \ref{E:zigzag}. 
\end{proof}

\begin{lemma}\label{L:psi-an-iso}
Let $\R\in \{\kk, \A, \K\}$. Then ${^\R\psi}_k$ is an isomorphism. 
\end{lemma}
\begin{proof}
Thanks to Proposition \ref{L:ext-powers-irredicible}, for $\R\in \{\kk, \K\}$, ${^\R\psi}_k$ is an isomorphism if and only if ${^\R\psi}_k\ne 0$. We know from weight considerations that ${^{\A}\psi_k}(v_{\{m-k+1,\dots, m\}}^*)= \xi_k\cdot v_{\{1, \dots, k\}}$ for some $\xi_k\in \A$, and ${^{\A}\psi}_k$ is an isomorphism if and only if $\xi_k\in \A^{\times}$ if and only if $\xi_k$ does not map to zero under $\A\rightarrow \kk$. Since $\base_{\R}(^{\A}\psi_k) = {^{\R}\psi}_k$ and $\base_{\K}$ is faithful, the claim will follow if we show ${^{\kk}\psi_k}\ne 0$. This follows from Lemma \ref{L:ckne0} and Equation \ref{E:zigzag}.
\end{proof}

\begin{defn}\label{D:ek}
Let $\R\in \{\kk, \A, \K\}$. Define ${^\R\varphi}_k:=({^\R\psi}_k)^{-1}$ and
\[
{^\R e}_k:=\text{ev}\circ ({^\R\varphi}_k\otimes \id_{\Lambda_\R^k}) \in \Hom_{U_\R(\om)}(\Lambda_\R^k\otimes \Lambda_\R^k, \R).
\]
\end{defn}

\begin{remark}
When $k=1$, the previous definition agrees with Definition \ref{D:e1-and-c1}.
\end{remark}

\begin{lemma}\label{l:eczigzag}
Let $\R\in \{\kk, \A, \K\}$. The following equality of morphisms holds. 
\[
(\id_{\Lambda_\R^k}\otimes {^\R e}_k)\circ ({^\R c}_k\otimes \id_{\Lambda_\R^k}) = \id_{\Lambda_\R^k} \quad \text{and}\quad ({^\R e}_k\otimes \id_{\Lambda_\R^k})\circ (\id_{\Lambda_\R^k}\otimes {^\R c}_k) = \id_{\Lambda_\R^k}.
\]
\end{lemma}
\begin{proof}
Using Equation \eqref{E:zigzag} this follows from the definition of ${^\R e}_k$ and ${^\R c}_k$ along with the interchange law for monoidal categories\footnote{If $\mathcal{C}$ is a monoidal category, then for $\alpha \in \Hom_{\mathcal{C}}(X,Y)$ and $\alpha' \in \Hom_{\mathcal{C}}(X',Y')$, we have $(\alpha\otimes \id_{Y'})\circ (\id_{X}\otimes \alpha') = (\id_{Y}\otimes \alpha') \circ (\alpha \otimes \id_{X'})$.}
\end{proof}

We will now explain how the category $\Fund$ can be described by graphical calculus, in the spirit of \cite[Theorem 5.1]{MR1036112}. Intuitively, we mean that the morphism in $\Fund$ can be represented by isotopy classes of unoriented trivalent graphs. Precisely, we mean that $\Fund$ receives a pivotal functor from the following diagrammatically defined category.

\begin{definition}
    For $\R\in \{\A, \kk, \K\}$, let $\textbf{FreeWeb}_{\R}(\mathfrak{o}_m)$ denote the pivotal $\R$-linear category generated by self-dual objects $n\in \{1, \dots, m\}$, and morphisms generated by the following trivalent vertices: $\freewebgenone$ and $\freewebgentwo$.
\end{definition}

Morphisms in $\textbf{FreeWeb}_{\R}(\mathfrak{o}_m)$ are $\R$-linear combinations of planar diagrams which locally are one of the following: trivalent vertex, cap morphisms $ \CapKRef$, or cup morphism $\CupKRef$, modulo the relations implied by being pivotal, i.e. unoriented versions of \cite[Relations 1,2, and 13]{MR1036112}.

\def\KZigZag
{\begin{FGtric}
\draw [scale=0.7](0,0)..controls(0,1)and(1,1)..(1,0);
\draw [scale=0.7](1,0)..controls(1,-1)and(2,-1)..(2,0);
\draw [scale=0.7](0,0)--(0,-2)node[left,black,midway,scale=0.7]{$k$}
                 (2,0)--(2,2);
\end{FGtric}
}

\def\KZigZaga
{\begin{FGtric}
\draw [scale=0.7](0,-2)--(0,2)node[left,black,pos=0.3,scale=0.7]{$k$};
\end{FGtric}}

\def\KZigZagb
{\begin{FGtric}
\draw [scale=0.7](0,0)..controls(0,1)and(-1,1)..(-1,0);
\draw [scale=0.7](-1,0)..controls(-1,-1)and(-2,-1)..(-2,0);
\draw [scale=0.7](0,0)--(0,-2)node[left,black,midway,scale=0.7]{$k$}
(-2,0)--(-2,2);
\end{FGtric}
}

\def\MorphZigzag{
\begin{tric}
\draw [darkgreen]
(.4,0.4)--(.4,3)   node[above,scale=0.7,black]{$1$}
(0.9,0.4)--(0.9,3)  node[above,scale=0.7,black]{$1$}
(2.5,0.4)--(2.5,3)  node[above,scale=0.7,black]{$1$}

(.4,-0.4)--(.4,-3) node[below,scale=0.7,black]{$1$}
(0.9,-0.4)--(0.9,-3) node[below,scale=0.7,black]{$1$}
(2.5,-0.4)--(2.5,-3) node[below,scale=0.7,black]{$1$}     ;

\filldraw[black] (1.35,1.7) circle (1pt) (1.7,1.7) circle (1pt) (2.05,1.7) circle (1pt) 
(1.35,-1.7) circle (1pt) (1.7,-1.7) circle (1pt) (2.05,-1.7) circle (1pt);

\draw[darkred,thick] (-0.2,-0.4)rectangle(3.1,0.4)  
                     node[pos=0.5,scale=0.7,black]{$k+1$};
\end{tric}
}

\def\PivotalMorphA
{\begin{FGtric}
\draw  (3.7,2.5)--(3.7,2)..controls(3.7,-0.5)and(3.3,-1.65)..(1.8,-1.65)..controls(1.2,-1.65)and(0,-1)..(0,0)--(0,1.5)..controls(0,2)and(-0.7,2.3)..(-0.7,0)--(-0.7,-1)

 (3.2,2.5)--(3.2,2)..controls(3.2,-0.5)and(2.8,-1.2)..(1.8,-1.2)..controls(1.2,-1.2)and(0.5,-0.6)
 ..(0.5,0)--(0.5,1.5)..controls(0.5,2.5)and(-1.2,3)..(-1.2,0)--(-1.2,-1)
 
(2.2,2.5)--(2.2,2)..controls(2.2,0)and(2.1,-0.4)..(1.8,-0.4)..controls(1.7,-0.4)and(1.5,-0.3)..(1.5,0)--(1.5,1.5)..controls(1.5,2)and(1,3)..(0,3)..controls(-1.5,3)and(-2.2,2)..(-2.2,0)--(-2.2,-1);

\fill[white] (0.75,1) ellipse (1.2 and 0.3);
\draw[very thick,black](0.75,1) ellipse (1.2 and 0.3);
\draw(0.3,1.6)node[black,right]{$\dots$}
(0.3,0.4)node[black,right]{$\dots$}
(-2.38,-0.6)node[black,right]{$\dots$}
(2,2.1)node[black,right]{$\dots$};
\draw [very thick, decorate, decoration = {calligraphic brace}] 
(-0.7,-1.25) -- (-2.2,-1.25);
\draw [very thick, decorate, decoration = {calligraphic brace}] 
(2.2,2.7)--(3.7,2.7);
\draw (0.75,1)node[scale=0.7,black]{$f$}
      (-1.45,-1.4)node[scale=0.7,black,below]{$d$} 
     (2.95,2.85)node[scale=0.7,black,above]{$e$};
\end{FGtric}
}

\def\PivotalMorphB
{\begin{FGtric}
\begin{scope}[yscale=-1]
\draw  (3.7,2.5)--(3.7,2)..controls(3.7,-0.5)and(3.3,-1.65)..(1.8,-1.65)..controls(1.2,-1.65)and(0,-1)..(0,0)--(0,1.5)..controls(0,2)and(-0.7,2.3)..(-0.7,0)--(-0.7,-1)

 (3.2,2.5)--(3.2,2)..controls(3.2,-0.5)and(2.8,-1.2)..(1.8,-1.2)..controls(1.2,-1.2)and(0.5,-0.6)
 ..(0.5,0)--(0.5,1.5)..controls(0.5,2.5)and(-1.2,3)..(-1.2,0)--(-1.2,-1)
 
(2.2,2.5)--(2.2,2)..controls(2.2,0)and(2.1,-0.4)..(1.8,-0.4)..controls(1.7,-0.4)and(1.5,-0.3)..(1.5,0)--(1.5,1.5)..controls(1.5,2)and(1,3)..(0,3)..controls(-1.5,3)and(-2.2,2)..(-2.2,0)--(-2.2,-1);

\fill[white] (0.75,1) ellipse (1.2 and 0.3);
\draw[very thick,black](0.75,1) ellipse (1.2 and 0.3);
\draw(0.3,1.6)node[black,right]{$\dots$}
(0.3,0.4)node[black,right]{$\dots$}
(-2.38,-0.6)node[black,right]{$\dots$}
(2,2.1)node[black,right]{$\dots$};
\draw [very thick, decorate, decoration = {calligraphic brace}] 
(-2.2,-1.25)--(-0.7,-1.25);
\draw [very thick, decorate, decoration = {calligraphic brace}] 
(3.7,2.7)--(2.2,2.7);
\draw (0.75,1)node[scale=0.7,black]{$f$}
      (-1.45,-1.4)node[scale=0.7,black,above]{$e$} 
     (2.95,2.85)node[scale=0.7,black,below]{$d$};
\end{scope}     
\end{FGtric}
}

\begin{remark}
    The existence of cap and cup morphisms are implicit in $\textbf{FreeWeb}_{\R}(\mathfrak{o}_m)$ being pivotal, that is a pivotal category is a rigid category, so has left and right duals, i.e. there are cap and cup morphisms satisfying the following relation. 
    \begin{equation}\label{zigzig-graph}
    \KZigZag \quad = \quad \KZigZaga \quad = \quad \KZigZagb
    \end{equation}
    The key feature of a pivotal category is that the functor of left dual is equal to the functor of right dual. This is expressed in graphical calculus as follows.
    \begin{equation}\label{leftdualisrightgraph}\PivotalMorphA=\PivotalMorphB 
    \end{equation}
    We say that the 180 degree counterclockwise rotation of a morphism is equal to a 180 degree clockwise rotation of the same morphism
\end{remark}

Our proof of Theorem \ref{T:main-thm} requires certain base change arguments. These arguments require a graphical calculus for $\Fund$ when $\R\in \{\A, \K, \kk\}$, and compatibility of the graphical calculus descriptions with change of rings from $\A$ to $\R\in \{\K, \kk\}$. The following proposition uses known results about graphical calculus for $\FundK$ to deduce a graphical calculus for $\FundA$.

\begin{proposition}\label{P:greycalculuswelldefined}
There is a pivotal functor $\Phi_{\A}^{free}:\textbf{FreeWeb}_{\A}(\mathfrak{o}_m)\longrightarrow \FundA$ such that $k\mapsto \Lambda^k$, $\freewebgenone \mapsto {^{\A} m}_{k,1}^{k+1}$, and $\freewebgentwo\mapsto {^{\A} m}_{1,k}^{k+1}$. Moreover, we have $\CapKRef \mapsto {^{\A}e}_k$ and $\CupKRef \mapsto {^{\A}c}_k$.
\end{proposition}

\begin{proof}
Denote by $\textbf{Rep}(U_{\K}(\mathfrak{so}_m))$, the category of finite dimensional type-$\mathbf{1}$ $U_{\K}(\mathfrak{so}_m)$-modules, equipped with the pivotal structure from \cite{SnyderTingley}. This pivotal structure is such that all self-dual irreducible representations have Frobenius-Schur index $+1$. Viewing $\FundK$ as a subcategory of $\textbf{Rep}(U_{\K}(\mathfrak{so}_m))$, and recalling Lemma \ref{L:selfduality-of-O-reps}, it follows from the discussion in \cite[Section 2.2 and Theorem 5.1]{bodish2021type}, see also \cite{Sel2}, that there is a pivotal functor $\Phi_{\K}^{free}:\textbf{FreeWeb}_{\K}(\mathfrak{o}_m)\longrightarrow \FundK$, such that the generating trivalent vertices are sent to the morphisms ${^{\K} m}_{k,1}^{k+1}$ and ${^{\K} m}_{1,k}^{k+1}$.

Consider the morphisms $\text{cup}_k:=\Phi_{\K}^{free}\left(\CupKRef\right)$ and $\text{cap}_k = \Phi_{\K}^{free}\left(\CapKRef\right)$. Note that
\begin{equation}\label{eqn:cupk-capk-zigzag}
(\id_{\Lambda_\R^k}\otimes \text{cap}_k)\circ (\text{cup}_k\otimes \id_{\Lambda_\R^k}) = \id_{\Lambda_\R^k} \quad \text{and}\quad (\text{cap}_k\otimes \id_{\Lambda_\R^k})\circ (\id_{\Lambda_\R^k}\otimes \text{cup}_k) = \id_{\Lambda_\R^k}.
\end{equation}
If we re-scale $\text{cup}_k$ by any $\lambda_k\in \K^{\times}$, and also re-scale $\text{cap}_k$ by $\lambda_k^{-1}$, then this preserves the pivotal structure on $\FundK$. So we may assume that $\Phi_{\K}^{free}$ is such that $ \CapKRef \mapsto \lambda_k\cdot\text{cap}_k$ and $\CupKRef\mapsto \lambda_k^{-1}\cdot \text{cup}_k$.

We know that $\Hom_{U_{\K}(\om)}(\K, \Lambda_{\K}^k\otimes \Lambda_{\K}^k)$ is one dimensional. Also, both $\text{cup}_k$ and ${^{\K}c}_k$ are non-zero elements of $\Hom_{U_{\K}(\om)}(\K, \Lambda_{\K}^k\otimes \Lambda_{\K}^k)$. Therefore, there is $\lambda_k\in \K^{\times}$ such that ${^{\K}c}_k = \lambda_k\cdot \text{cup}_k$. Using equation \ref{eqn:cupk-capk-zigzag} and Lemma \ref{l:eczigzag} we find that
\begin{align*}
{^{\K}e_k} &= {^{\K}e_k}  \circ \left( \id_{\Lambda_{\K}^k}  \otimes \big( (\id_{\Lambda_{\K}^k}\otimes \lambda_{k}^{-1}\cdot\text{cap}_k)\circ (\lambda_{k}\cdot\text{cup}_k\otimes \id_{\Lambda_{\K}^k}) \big) \right)\\
&=\lambda_{k}^{-1}\cdot\text{cap}_k \circ \left( \big( ({^{\K}e_k}\otimes \id_{\Lambda_{\K}^k})\circ (\id_{\Lambda_{\K}^k}\otimes \lambda_k\cdot\text{cup}_k) \big)\otimes \id_{\Lambda_{\K}^k} \right) \\
&=\lambda_{k}^{-1}\cdot \text{cap}_k\circ \left( \big( ({^{\K}e_k}\otimes \id_{\Lambda_{\K}^k})\circ (\id_{\Lambda_{\K}^k}\otimes {^{\K}c_k}) \big) \otimes \id_{\Lambda_{\K}^k} \right)\\
&=\lambda_{k}^{-1}\cdot\text{cap}_k.
\end{align*}
So we can (and do) assume that $\Phi_{\K}^{free}$ sends the gray cap  labeled $k$ to ${^{\K}e}_k$ and the gray cup labeled $k$ to ${^{\K}c}_k$. 

Since $\Phi_{\K}^{free}$ is pivotal, it follows that left dual equals left dual \eqref{leftdualisrightgraph} in $\FundK$, when using caps ${^{\K}}e_k$ and cups ${^{\K}}c_k$. Note that $
\base_{\K}({^{\A}}e_k) = {^{\K}}e_k$ and $\base_{\K}({^{\A}c}_k) = {^{\K}c}_k$. From Lemma \ref{L:sufficestocheckoverK}, we see that left dual is equal to right dual in $\FundA$, when using ${^{\A}}e_k$ and ${^{\A}}c_k$. Also, we have $\base_{\K}({^{\A} m}_{i,j}^{i+j}) = {^{\K} m}_{i,j}^{i+j}$. Thus, there is a pivotal functor as in the statement of the proposition.

\end{proof}

\begin{definition}\label{D:comult}
Define ${^{\A}m}_{i+j}^{i,j}$ using the graphical calculus for morphisms as the $180$ degree rotation of ${^{\A}m}_{j,i}^{i+j}$. For $\R\in \{\kk, \K\}$, define ${^{\R}m}_{i+j}^{i,j}:=\base_{\R}( {^{\A}m}_{i+j}^{i,j})$. By Proposition \ref{P:greycalculuswelldefined} it does not matter whether we rotate clockwise or counterclockwise. 
\end{definition}

\begin{lemma}\label{L:graybigon}
Let $\R\in \{\kk, \A, \K\}$. For $k=1, \dots, m$,
\begin{equation*}
{^{\R}m}_{1,k-1}^{k}\circ {^{\R}m}_{k}^{1,k-1} = \frac{[2k]}{[2]}\id_{\Lambda_{\R}^k} \ .
\end{equation*}
\end{lemma}
\begin{proof}
For $\R= \A$, this follows by using the graphical calculus for $\FundA$ and Definition \ref{D:caps}. Then applying $\base_{\R}$ yields the result for $\R\in \{\kk, \K\}$. 
\end{proof}

\begin{lemma}\label{L:graycircle}
Let $\R\in \{\kk, \A, \K\}$. For $k=1, \dots, m$,
\begin{equation*}
{^{\R}e}_k\circ {^{\R}c}_k = \frac{[2m-4k][m]}{[m-2k][2m]}
\begin{bmatrix}
m \\
k 
\end{bmatrix} _{q^2}\cdot \id_{\Lambda_{\R}^0} \ .
\end{equation*}
\end{lemma}
\begin{proof}
We give a sketch, for more details see \cite[Section 2.2]{bodish2021type}. 

First, we observe that $\End_{\A}(\Lambda_{\A}^0)= \A\cdot \id_{\Lambda_{\A}^0}$, and every $\A$-linear endomorphism of the trivial module commutes with $U_{\A}(\om)$, so $\End_{U_{\A}(\om)}(\Lambda_{\A}^0)= \A\cdot \id_{\Lambda_{\A}^0}$. It follows that ${^{\A}e}_k\circ {^{\A}c_k}= d_k(m)\cdot \id_{\Lambda_{\A}^0}$, for some $d_k(m)\in \A$. Thus, 
\[
{^{\kk}e}_k\circ {^{\kk}c_k}= \overline{d_k(m)}\cdot \id_{\Lambda_{\kk}^0} \quad \text{and} \quad {^{\K}e}_k\circ {^{\K}c_k}= d_k(m)\cdot \id_{\Lambda_{\K}^0},
\]
where $\overline{d_k(m)}$ denotes the image of $d_k(m)$ under $\A\rightarrow \kk$. So it suffices to show the claim for $\R= \K$. 

Taking the trace of the identity of an object, with respect to our chosen pivotal structure, gives the quantum dimension. The quantum Weyl dimension formula states that for our chosen pivotal structure on $\textbf{Rep}(U_{q}(\som))$ 
\[
\mathrm{qdim}(L_{\K}(\wta)) := \tr_q\left(\id_{L_{\K}(\wta)}\right) =  (-1)^{(\rho^{\vee}, 2\wta)} \prod_{\alpha\in \Phi_+}\frac{[(\alpha^{\vee}, \wta+ \rho)]_{v_{\alpha}}}{[(\alpha^{\vee}, \wta)]_{v_{\alpha}}} \ ,
\]
where $v=q$, if $m$ is odd, and $v=q^2$, if $m$ is even. On the other hand, the trace of $\id_{\Lambda_{\K}^k}$, with respect to our chosen pivotal structure, is exactly the coefficient $d_k(m)$. 

We leave it as an exercise to use the quantum dimension formula, along with Remark \ref{R:L'Hospitals}, to derive the dimension formula in the statement of the Lemma. For a hint, look at proof of \cite[Proposition 2.2]{bodish2021type}.
\end{proof}

The last intertwiner we will consider is the braiding isomorphism $\brrepK:V_{\K}\otimes V_{\K}\rightarrow V_{\K}\otimes V_{\K}$. To this end, we observe that $V_{\K}\otimes V_{\K}$ is a direct sum of three non-isomorphic irreducible representations\footnote{Except when $m=1$, and we have $V_{\K}(2)=0 = V_{\K}(1,1)$. In this case $v_1\otimes v_1$ generates $V_{\K}(0)$.}: $L_{\K}(2\varpi_1, +1)$, which is characterized as containing $v_1\otimes v_1$, $L_{\K}(\varpi_2, +1)$, which is isomorphic to $\Lambda_{\K}^2$, and $V_{\K}(0, +1)$, the trivial module. Write $\pi_{(2)}$, $\pi_{(1,1)}$, and $\pi_0$ for the projections in $\End_{U_{\K}(\om)}(V_{\K}^{\otimes 2})$ with image $L_{\K}(2\varpi_1, +1)$, $L_{\K}(\varpi_2, +1)$, and $L_{\K}(0, +1)$ respectively. Then it follows from \cite[Equation 6.12]{LZ-stronglymultifree} that
\[
\brrepK = q^2\pi_{(2)} - q^{-2}\pi_{(1,1)} + q^{2-2m}\pi_{(0)} \quad \text{and} \quad \brrepKi = q^{-2}\pi_{(2)} - q^{2}\pi_{(1,1)} + q^{-2+2m}\pi_{(0)}.
\]

\begin{lemma}\label{L:brrep}
We can express the braiding and its inverse in terms of our previously defined morphisms as 
\[
\brrepK = q^2\cdot \id_{V_{\K}\otimes V_{\K}} - {^{\K}m}_{2}^{1,1}\circ {^{\K}m}_{1,1}^2 - \frac{[m-2]}{[2m-4]}(q^2-q^{-2})q^{-m+2}\cdot  {^{\K}c}_1\circ {^{\K}e}_1
\]
and
\[
\brrepKi = q^{-2}\cdot \id_{V_{\K}\otimes V_{\K}} - {^{\K}m}_{2}^{1,1}\circ {^{\K}m}_{1,1}^2 + \frac{[m-2]}{[2m-4]}(q^2-q^{-2})q^{m-2}\cdot  {^{\K}c}_1\circ {^{\K}e}_1.
\]
\end{lemma}
\begin{proof}
First, we note that 
\[
\pi_{(1,1)} = \frac{[2]}{[4]}\cdot {^{\K}m}_{2}^{1,1}\circ {^{\K}m}_{1,1}^2 \quad \text{and} \quad \pi_{(0)} = \frac{[m-2][2m]}{[2m-4][m][m]_{q^2}}\cdot {^{\K}c}_1\circ {^{\K}e}_1.
\]
Since $\id_{V_{\K}\otimes V_{\K}} = \pi_{(2)} + \pi_{(1,1)} + \pi_0$, it follows that
\[
\pi_{(2)} = \id_{V_{\K}\otimes V_{\K}} -\frac{[2]}{[4]}\cdot {^{\K}m}_{2}^{1,1}\circ {^{\K}m}_{1,1}^2 - \frac{[m-2][2m]}{[2m-4][m][m]_{q^2}}\cdot {^{\K}c}_1\circ {^{\K}e}_1.
\]
Thus, 
\begin{align*}
\brrepK &= q^2\id_{V_{\K}\otimes V_{\K}} + \left(-q^{2}\frac{[2]}{[4]} -q^{-2} \frac{[2]}{[4]}\right) \cdot {^{\K}m}_{2}^{1,1}\circ {^{\K}m}_{1,1}^2 \\
&+\left(- q^2\frac{[m-2][2m]}{[2m-4][m][m]_{q^2}} + q^{2-2m}\frac{[m-2][2m]}{[2m-4][m][m]_{q^2}} \right) \cdot {^{\K}c}_1\circ {^{\K}e}_1 \\
&=q^2\cdot \id_{V_{\K}\otimes V_{\K}} - {^{\K}m}_{2}^{1,1}\circ {^{\K}m}_{1,1}^2 - \frac{[m-2]}{[2m-4]}(q^2-q^{-2})q^{-m+2}\cdot  {^{\K}c}_1\circ {^{\K}e}_1. 
\end{align*}

The same argument is used to derive the formula for $\brrepKi$. 
\end{proof}

\begin{definition}\label{D:graybraiding}
Let $\R\in \{\kk, \A, \K\}$, we define
\[
\brrep:= q^2\cdot \id_{V_{\R}\otimes V_{\R}} - {^{\R}m}_{2}^{1,1}\circ {^{\R}m}_{1,1}^2 - \frac{[m-2]}{[2m-4]}(q^2-q^{-2})q^{-m+2}\cdot  {^{\R}c}_1\circ {^{\R}e}_1
\]
and
\[
\brrepi:= q^{-2}\cdot \id_{V_{\R}\otimes V_{\R}} - {^{\R}m}_{2}^{1,1}\circ {^{\R}m}_{1,1}^2 + \frac{[m-2]}{[2m-4]}(q^2-q^{-2})q^{m-2}\cdot  {^{\R}c}_1\circ {^{\R}e}_1. 
\]
\end{definition}

\begin{lemma}\label{gray-twistbraiding-is-inverse}
Let $\R\in \{\kk, \A, \K\}$. The $90$ degree rotation of $\brrep$ is $\brrepi$. 
\end{lemma}
\begin{proof}
The claim is equivalent to 
\[
(\id_{V_{\R}}\otimes {^{\R}e}_1)\circ (\brrep\otimes \id_{V_{\R}})\circ (\id_{V_{\R}}\otimes \brrep) = ({^{\R}e}_1\otimes \id_{V_{\R}}).
\]

By Lemma \ref{L:sufficestocheckoverK} it suffices to show this is true for $\brrepK$. The Hexagon equation implies that $\beta_{V_{\K}\otimes V_{\K}, V_{\K}} = (\brrepK\otimes \id_{V_{\K}})\circ (\id_{V_{\K}}\otimes \brrepK)$. Since we are in a strict braided monoidal category, we also have $\beta_{\K, V_{\K}} = \id_{V_{\K}}$. Also, by naturality of the braiding we have $(\id_{V_{\K}}\otimes \ {^{\K}e}_1)\circ \beta_{V_{\K}\otimes V_{\K}, V_{\K}}
= \beta_{\K, V_{\K}} \circ ({^{\K}e}_1\otimes \id_{V_{\K}})$. Thus,
\begin{align*}
(\id_{V_{\K}}\otimes {^{\K}e}_1)\circ (\brrepK\otimes \id_{V_{\K}})\circ (\id_{V_{\K}}\otimes \brrepK) &=(\id_{V_{\K}}\otimes {^{\K}e}_1)\circ \beta_{V_{\K}\otimes V_{\K}, V_{\K}} \\
&= \beta_{\K, V_{\K}} \circ ({^{\K}e}_1\otimes \id_{V_{\K}}) \\
&= \id_{V_{\K}}\circ ({^{\K}e}_1\otimes \id_{V_{\K}}) \\
&=({^{\K}e}_1\otimes \id_{V_{\K}}).
\end{align*}
\end{proof}

Unsurprisingly, the braiding when $q=1$ is just the tensor flip map.

\begin{lemma}\label{L:brrepC-is-flipmap}
The map $\brrepC$ acts on $V_{\kk}^{\otimes 2}$ by $v\otimes w\mapsto w\otimes v$.
\end{lemma}
\begin{proof}
Let $s$ denote the tensor flip map. Since $q-1=0$ in $\kk$, we see that Definition \ref{D:graybraiding} simplifies to $\brrepC = \id_{V_{\kk}\otimes V_{\kk}} - {^{\R}m}_2^{1,1}\circ {^{\R}m}_{1,1}^2$. Since ${^{\kk}m}_2^{1,1}\circ {^{\kk}m}_{1,1}^2$ factors through $\Lambda_{\kk}^2$ and squares to $2$, we have $\frac{1}{2}{^{\kk}m}_2^{1,1}\circ {^{\kk}m}_{1,1}^2 = \frac{1}{2}\left(\id_{V_{\kk}\otimes V_{\kk}} - s\right)$, the anti-symmetrizing idempotent. Thus, $\brrepC = \id_{V_{\kk}\otimes V_{\kk}} - \left(\id_{V_{\kk}\otimes V_{\kk}} -s\right) = s$. 
\end{proof}

%% file: proof.tex
\section{Existence of the functor}

Let $\R\in \lbrace \K, \A, \kk\rbrace$. We will prove that there is a pivotal functor $\Phi_{\R}:\Web\longrightarrow \Fund$. For later arguments, it is important for us to have canonical identifications between $\R\otimes \Phi_{\A}$ and $\Phi_\R$. To make this precise, we construct $\Phi_{\A}$ first, then define $\Phi_\R:=\base_{\R}\circ (\R\otimes \Phi_{\A})$. Since $\WebA$ is a generators and relations category, it suffices to define where generators go and check relations. We already defined where the generators go in Proposition \ref{P:greycalculuswelldefined}, when we established the existence of $\Phi_{\A}^{free}$. Thus, the majority of this section is devoted to deriving various relations among morphisms in $\Fund$. 

\subsection{Deriving relations}

In this section, all the graphical calculations take place in the category $\FundA$, as opposed to $\WebA$. To differentiate the graphical calculus for $\FundA$ from $\WebA$ and $\textbf{FreeWeb}_{\A}(\mathfrak{o}_m)$ we use gray diagrams. That this is valid is justified by Proposition \ref{P:greycalculuswelldefined}.

\begin{notation}
    We have the following gray graphical calculus for morphisms in $\FundA$. 
\begin{gather*}
    \GCapKRef:=\Phi_{\A}^{free}\left(\CapKRef\right), \quad \GCupKRef:=\Phi_{\A}^{free}\left(\CupKRef\right), \quad \text{and} \\
    \VertoMultiRef:={^{\A} m}_{i,j}^{i+j}
\end{gather*}
\end{notation}

\begin{remark}
    When we use gray diagrams in this section, the reader should interpret these diagrams as being $\Phi_{\A}^{free}$ applied to the apricot colored diagram in $\textbf{FreeWeb}_{\A}(\mathfrak{o}_m)$. The reader should also note that we will work in $\FundA$, unless we explicitly say that we are working over $\R\in \{\K, \kk\}$, in which case the gray diagram can be taken to represent $\Phi_{\R}$ applied to a apricot colored free web diagram.
\end{remark}

First, since $\Lambda_{\A}^{\bullet}$ is an associative graded algebra, we have the following.

\def\genSkeinassocia
{\begin{Gtric}
\draw[scale=0.8] (0,0)..controls(0,0.5)and(0.2,0.7)..(0.5,1)
      (1,0)..controls(1,0.5)and(0.8,0.7)..(0.5,1)
      (0.5,1)..controls(0.5,1.5)and(0.3,1.7)..(0,2) 
              node[right,black,midway, scale=0.7]{$k+l$}
      (-1,0)..controls(-1,1)and(-0.6,1.5)..(0,2)
      (0,2)--(0,3)
      (0,0)node[below,black,scale=0.7]{$k$}
      (1,0)node[below,black,scale=0.7]{$l$}
      (-1,0)node[below,black,scale=0.7]{$m$}
      (0,3)node[above,black,scale=0.7]{$k+l+m$};
\end{Gtric}
}

\def\genSkeinassoci
{\begin{Gtric}
\draw[scale=0.8] (0,0)..controls(0,0.5)and(-0.2,0.7)..(-0.5,1)
      (-1,0)..controls(-1,0.5)and(-0.8,0.7)..(-0.5,1)
      (-0.5,1)..controls(-0.5,1.5)and(-0.3,1.7)..(0,2) 
              node[left,black,midway, scale=0.7]{$k+m$}
      (1,0)..controls(1,1)and(0.6,1.5)..(0,2)
      (0,2)--(0,3)
      (0,0)node[below,black,scale=0.7]{$k$}
      (-1,0)node[below,black,scale=0.7]{$m$}
      (1,0)node[below,black,scale=0.7]{$l$}
      (0,3)node[above,black,scale=0.7]{$k+l+m$};
\end{Gtric}
}

\def\NextSkeinIH
{\begin{Gtric}
\draw[scale=0.7] (-1.732,-2)--(0,-1)--(1.732,-2) (1.732,2)--(0,1)--(-1.732,2) 
(-1.732,-2)node[below,black,scale=0.7]{$k+1$}
(1.732,-2)node[below,black,scale=0.7]{$1$}
(1.732,2)node[above,black,scale=0.7]{$1$}
(-1.732,2) node[above,black,scale=0.7]{$k+1$}
(0,-1)--(0,1)node[left,midway,black,scale=0.7]{$k+2$};
\end{Gtric}
}

\def\NextSkeinIHa
{\begin{Gtric}
\draw [scale=0.7]  (2,-1.74)--(1,0)--(2,1.74) 
(-3,1.74)--(0,0) node[above,pos=0.7,black,scale=0.7]{$1$}
(0,0)--(-3,-1.74) node[below,pos=0.3,black,scale=0.7]{$1$}
(0,0)--(1,0) node[below,midway,black,scale=0.7]{$2$}
(-3,-1.74) node[below,black,scale=0.7]{$k+1$}
(-3,1.74) node[above,black,scale=0.7]{$k+1$}
(2,-1.74) node[below,black,scale=0.7]{$1$}
(2,1.74) node[above,black,scale=0.7]{$1$}
(-2,1.16)--(-2,-1.16)node[left,midway,black,scale=0.7]{$k$};
\end{Gtric}
}

\def\NextSkeinIHb
{\begin{Gtric}
\draw[scale=0.7] (-1.732,-2)--(0,-1)--(1.732,-2) (1.732,2)--(0,1)--(-1.732,2) 
(-1.732,-2)node[below,black,scale=0.7]{$k+1$}
(1.732,-2)node[below,black,scale=0.7]{$1$}
(1.732,2)node[above,black,scale=0.7]{$1$}
(-1.732,2) node[above,black,scale=0.7]{$k+1$}
(0,-1)--(0,1)node[left,midway,black,scale=0.7]{$k$};
\end{Gtric}
}

\def\NextSkeinIHc
{\begin{Gtric}
\draw  [scale=0.7] 
(2,-2)..controls(1,-1)and(1,1)..(2,2) (2,-2)node[below,black, scale=0.7]{$1$}
(-2,2)node[above,black,scale=0.7]{$\phantom{k+1}$} 
(-2,2)..controls(-1,1)and(-1,-1)..(-2,-2) (-2,-2)node[below,black,scale=0.7]{$k+1$};
\end{Gtric}
}

\def\LASkeinIH
{\begin{Gtric}
\draw[scale=0.7] (-1.732,-2)--(-1,0) (1,0)--(2,-1.732)
(-1,0)..controls(-1,2)and(1,2)..(1,0) node[above,midway,black,scale=0.7]{$1$}
(-2,-1.732)node[below,black,scale=0.7]{$k+1$}
(2,-1.732) node[below,black,scale=0.7]{$k+1$}
(-1,0)--(1,0)node[below,midway,black,scale=0.7]{$k+2$};
\end{Gtric}
}

\def\LASkeinIHa
{\begin{Gtric}
\draw [scale=0.7]  (0,1)..controls(1,1)and(1,2)..(0,2)..controls(-1,2)and(-1,1)..(0,1)
(0.7,1.5) node[right,black,scale=0.7]{$1$}
(1.74,-3)--(0,0) node[right,pos=0.7,black,scale=0.7]{$1$}
(0,0)--(-1.74,-3) node[left,pos=0.3,black,scale=0.7]{$1$}
(0,0)--(0,1) node[left,midway,black,scale=0.7]{$2$}
(-1.74,-3) node[below,black,scale=0.7]{$k+1$}
(1.74,-3) node[below,black,scale=0.7]{$k+1$}
(1.16,-2)--(-1.16,-2)node[below,midway,black,scale=0.7]{$k$};
\end{Gtric}
}

\def\LASkeinIHb
{\begin{Gtric}
\draw[scale=0.7] (-2,-1.732)--(-1,0)  (1,0)--(2,-1.732) 
(-1,0)..controls(-1,2)and(1,2)..(1,0) node[above,midway,black,scale=0.7]{$1$}
(-2,-1.732)node[below,black,scale=0.7]{$k+1$}
(2,-1.732) node[below,black,scale=0.7]{$k+1$}
(-1,0)--(1,0)node[below,midway,black,scale=0.7]{$k$};
\end{Gtric}
}

\def\LASkeinIHc
{\begin{Gtric}
\draw  [scale=0.7] 
(0,0.5) circle (1) 
(0,1.5) node[above,black,scale=0.7]{$1$}
(2,-2)..controls(1,-1)and(-1,-1)..(-2,-2) node[below,black,midway,scale=0.7]{$k+1$};
\end{Gtric}
}

\def\LBSkeinIH
{\begin{Gtric}
\draw[scale=0.7] (-2,-1.732)--(-1,0)  (1,0)--(2,-1.732) 
(1,0)--(0,1.5) node[right,pos=0.6,black,scale=0.7]{$1$}
(0,1.5)--(-1,0) node[left,pos=0.4,black,scale=0.7]{$1$}
(0,1.5)--(0,3) node[left,midway,black,scale=0.7]{$2$}
(-2,-1.732)node[below,black,scale=0.7]{$k+1$}
(2,-1.732) node[below,black,scale=0.7]{$k+1$}
(-1,0)--(1,0)node[below,midway,black,scale=0.7]{$k+2$};
\end{Gtric}
}

\def\LBSkeinIHa
{\begin{Gtric}
\draw [scale=0.7] (0,1)..controls(1,1)and(1,2)..(0,2)..controls(-1,2)and(-1,1)..(0,1)
(-0.7,1.5) node[left,black,scale=0.7]{$1$}
(0.7,1.5) node[right,black,scale=0.7]{$1$}
(0,2)--(0,3)node[left,black,pos=0.6,scale=0.7]{$2$}
(1.74,-3)--(0,0) node[right,pos=0.7,black,scale=0.7]{$1$}
(0,0)--(-1.74,-3) node[left,pos=0.3,black,scale=0.7]{$1$}
(0,0)--(0,1) node[left,midway,black,scale=0.7]{$2$}
(-1.74,-3) node[below,black,scale=0.7]{$k+1$}
(1.74,-3) node[below,black,scale=0.7]{$k+1$}
(1.16,-2)--(-1.16,-2)node[below,midway,black,scale=0.7]{$k$};
\end{Gtric}
}

\def\LBSkeinIHb
{\begin{Gtric}
\draw[scale=0.7] (-2,-1.732)--(-1,0) (1,0)--(2,-1.732) 
(-2,-1.732)node[below,black,scale=0.7]{$k+1$}
(1,0)--(0,1.5) node[right,pos=0.6,black,scale=0.7]{$1$}
(0,1.5)--(-1,0) node[left,pos=0.4,black,scale=0.7]{$1$}
(0,1.5)--(0,3) node[left,midway,black,scale=0.7]{$2$}
(2,-1.732) node[below,black,scale=0.7]{$k+1$}
(-1,0)--(1,0)node[below,midway,black,scale=0.7]{$k$};
\end{Gtric}
}

\def\ZeroTriangleAtBound
{\begin{Gtric}
\draw[scale=0.7] (-2,-1.732)--(-1,0) (1,0)--(2,-1.732) 
(-2,-1.732)node[below,black,scale=0.7]{$m$}
(1,0)--(0,1.5) node[right,pos=0.6,black,scale=0.7]{$1$}
(0,1.5)--(-1,0) node[left,pos=0.4,black,scale=0.7]{$1$}
(0,1.5)--(0,3) node[left,midway,black,scale=0.7]{$2$}
(2,-1.732) node[below,black,scale=0.7]{$m$}
(-1,0)--(1,0)node[below,midway,black,scale=0.7]{$m-1$};
\end{Gtric}
}

\def\LBSkeinIHc
{\begin{Gtric}
\draw  [scale=0.7] 
(0,-1)..controls(1,-1)and(1,0)..(0,0)..controls(-1,0)and(-1,-1)..(0,-1)
(0.7,-0.5) node[right,black,scale=0.7]{$1$}
(0,0)--(0,1.5)node[left,black,pos=0.6,scale=0.7]{$2$}
(2,-3)..controls(1,-2)and(-1,-2)..(-2,-3) node[below,black,midway,scale=0.7]{$k+1$};
\end{Gtric}
}

\def\LCSkeinIH
{\begin{Gtric}
\draw[scale=0.7] (-2.5,-1.732)--(-1.5,0)--(-2.5,1.732) 
(1,0)..controls(1,1.5)and(2.5,1.5)..(2.5,0) node[above,midway,black,scale=0.7]{$1$}
(1,0)..controls(1,-1.5)and(2.5,-1.5)..(2.5,0) node[below,midway,black,scale=0.7]{$k+1$}
(2.5,0)--(5,0) node[above,midway,black,scale=0.7]{$k+2$}
(-2,-1.732)node[below,black,scale=0.7]{$k+1$}
(-2,1.732)node[above,black,scale=0.7]{$1$}
(-1.5,0)--(1,0)node[above,midway,black,scale=0.7]{$k+2$};
\end{Gtric}
}

\def\LCSkeinIHa
{\begin{Gtric}
\draw [scale=0.7]  (-1.74,2)--(0,1)
(1.16,-2)--(0,0) node[right,pos=0.7,black,scale=0.7]{$1$}
(0,0)--(-1.74,-3) node[left,pos=0.3,black,scale=0.7]{$1$}
(0,0)--(0,1) node[left,midway,black,scale=0.7]{$2$}
(-1.74,-3) node[below,black,scale=0.7]{$k+1$}
(-1.74,2) node[above,black,scale=0.7]{$1$}
(1.16,-2)--(-1.16,-2)node[below,midway,black,scale=0.7]{$k$}
(2,0)--(0,1) node[above,midway,black,scale=0.7]{$1$}
(2,0)--(1.16,-2) node[right,pos=0.7,black,scale=0.7]{$k+1$}
(2,0)--(4,0) node[above,midway,black,scale=0.7]{$k+2$};
\end{Gtric}
}

\def\LCSkeinIHb
{\begin{Gtric}
\draw[scale=0.7] (-2.5,-1.732)--(-1.5,0)--(-2.5,1.732) 
(1,0)..controls(1,1.5)and(2.5,1.5)..(2.5,0) node[above,midway,black,scale=0.7]{$1$}
(1,0)..controls(1,-1.5)and(2.5,-1.5)..(2.5,0) node[below,midway,black,scale=0.7]{$k+1$}
(2.5,0)--(5,0) node[above,midway,black,scale=0.7]{$k+2$}
(-2,-1.732)node[below,black,scale=0.7]{$k+1$}
(-2,1.732)node[above,black,scale=0.7]{$1$}
(-1.5,0)--(1,0)node[above,midway,black,scale=0.7]{$k$};
\end{Gtric}
}

\def\LCSkeinIHc
{\begin{Gtric}
\draw[scale=0.7] (-2.5,-1.732)--(-1.5,0)--(-2.5,1.732) 
(-2,-1.732)node[below,black,scale=0.7]{$k+1$}
(-2,1.732)node[above,black,scale=0.7]{$1$}
(-1.5,0)--(1,0)node[above,midway,black,scale=0.7]{$k+2$};
\end{Gtric}
}

\def\LCSkeinIHlem
{\begin{Gtric}
\draw [scale=0.8] 
(0,0)--(120:2) node[black,anchor=south,scale=0.7]{$1$}
(0,0)--(240:2) node[black,anchor=north,scale=0.7]{$k+1$}
(0,0)--(0:2) node[black,above,scale=0.7]{$k+2$};
\end{Gtric}
}

\def\TTcalc
{\begin{Gtric}
\draw[scale=0.7] (-2,-1.732)--(-1,0)  
(-2,-1.732)node[below,black,scale=0.7]{$k+1$}
(1,0)--(0,1.5) node[right,pos=0.6,black,scale=0.7]{$1$}
(0,1.5)--(-1,0) node[left,pos=0.4,black,scale=0.7]{$1$}
(-1,0)--(1,0)node[below,midway,black,scale=0.7]{$k+2$}

(0,1.5)..controls(0,2)and(4,2)..(4,1.5) 
node[above,midway,black,scale=0.7]{$2$}
(1,0)..controls(1.5,-1.5)and(2.5,-1.5)..(3,0) node[below,midway,black,scale=0.7]{$k+1$}

(6,-1.732)--(5,0)  
(6,-1.732)node[below,black,scale=0.7]{$k+1$}
(5,0)--(4,1.5) node[right,pos=0.6,black,scale=0.7]{$1$}
(4,1.5)--(3,0) node[left,pos=0.4,black,scale=0.7]{$1$}
(3,0)--(5,0)node[below,midway,black,scale=0.7]{$k$};
\end{Gtric}
}

\def\TTcalca
{\begin{Gtric}
\draw[scale=0.7] (-1,-1.732)--(0,0)  
(-1,-1.732)node[below,black,scale=0.7]{$k+1$}
(5,-1.732)--(4,0)  
(5,-1.732)node[below,black,scale=0.7]{$k+1$}
(0,0)--(4,0) node[below,black,scale=0.7,midway]{$k+2$}
(0,0)..controls(0,2)and(4,2)..(4,0) node[above,black,scale=0.7,midway]{$1$} ;
\end{Gtric}
}

\def\TTcalcb
{\begin{Gtric}
\draw[scale=0.7] 
(0,-1.732) .. controls(0,1)and(3,1) .. (3,-1.732)
 node[above,black,scale=0.7,midway]{$k+1$} ;
\end{Gtric}
}

\def\TTcalcc
{\begin{Gtric}
\draw[scale=0.7] (-2,-1.732)--(-1,0)  
(-2,-1.732)node[below,black,scale=0.7]{$k+1$}
(1,0)--(0,1.5) node[right,pos=0.6,black,scale=0.7]{$1$}
(0,1.5)--(-1,0) node[left,pos=0.4,black,scale=0.7]{$1$}
(-1,0)--(1,0)node[below,midway,black,scale=0.7]{$k$}

(0,1.5)..controls(0,2)and(4,2)..(4,1.5) 
node[above,midway,black,scale=0.7]{$2$}
(1,0)..controls(1.5,-1.5)and(2.5,-1.5)..(3,0) node[below,midway,black,scale=0.7]{$k+1$}

(6,-1.732)--(5,0)  
(6,-1.732)node[below,black,scale=0.7]{$k+1$}
(5,0)--(4,1.5) node[right,pos=0.6,black,scale=0.7]{$1$}
(4,1.5)--(3,0) node[left,pos=0.4,black,scale=0.7]{$1$}
(3,0)--(5,0)node[below,midway,black,scale=0.7]{$k$};
\end{Gtric}
}

\def\TTcalcd
{\begin{Gtric}
\draw[scale=0.7] (-2,-1.732)--(-1,0)  
(-2,-1.732)node[below,black,scale=0.7]{$k+1$}
(1,0)..controls(1,1.5)and(-1,1.5)..(-1,0) node[above,midway,black,scale=0.7]{$1$}
(-1,0)--(1,0)node[below,midway,black,scale=0.7]{$k$}

(1,0)..controls(1.5,-1.5)and(2.5,-1.5)..(3,0) node[below,midway,black,scale=0.7]{$k+1$}

(6,-1.732)--(5,0)  
(6,-1.732)node[below,black,scale=0.7]{$k+1$}
(5,0)..controls(5,1.5)and(3,1.5)..(3,0) node[above,midway,black,scale=0.7]{$1$}
(3,0)--(5,0)node[below,midway,black,scale=0.7]{$k$};
\end{Gtric}
}

\def\TTcalce
{\begin{Gtric}
\draw[scale=0.7] (-2,-1.732)--(-1,0)  
(-2,-1.732)node[below,black,scale=0.7]{$k+1$}
(-1,0)--(1,0)node[below,midway,black,scale=0.7]{$k$}

(1,0)..controls(1.5,-1.5)and(2.5,-1.5)..(3,0) node[below,midway,black,scale=0.7]{$k+1$}

(6,-1.732)--(5,0)  
(6,-1.732)node[below,black,scale=0.7]{$k+1$}
(3,0)--(5,0)node[below,midway,black,scale=0.7]{$k$}

(-1,0)..controls(0,2)and(1,2.5)..(2,2.5)
node[above,midway,black,scale=0.7]{$1$}
(5,0)..controls(4,2)and(3,2.5)..(2,2.5)
node[above,midway,black,scale=0.7]{$1$}
(2,2.5)--(2,1)node[left,midway,black,scale=0.7]{$2$}
(2,1)..controls(1.5,1)and(1,0.5)..(1,0)
node[left,pos=0.4,black,scale=0.7]{$1$}
(2,1)..controls(2.5,1)and(3,0.5)..(3,0)
node[right,pos=0.4,black,scale=0.7]{$1$};
\end{Gtric}
}

\def\TTcalcf
{\begin{Gtric}
\draw[scale=0.7] (-2,-1.732)--(-1,0)  
(-2,-1.732)node[below,black,scale=0.7]{$k+1$}
(-1,0)--(1,0)node[below,midway,black,scale=0.7]{$k$}

(1,0)..controls(1.5,-1.5)and(2.5,-1.5)..(3,0) node[below,midway,black,scale=0.7]{$k+1$}

(6,-1.732)--(5,0)  
(6,-1.732)node[below,black,scale=0.7]{$k+1$}
(3,0)--(5,0)node[below,midway,black,scale=0.7]{$k$}

(-1,0)..controls(0,2)and(1,2.5)..(2,2.5)
(5,0)..controls(4,2)and(3,2.5)..(2,2.5)
(2,2.5)node[above,black,scale=0.7]{$1$}
(3,0)..controls(3,0.5)and(2.5,1)..(2,1)..controls(1.5,1)and(1,0.5)..(1,0)
(2,1)node[above,black,scale=0.7]{$1$};
\end{Gtric}
}

\def\TTcalcg
{\begin{Gtric}
\draw[scale=0.7] (-0.5,-1.732)--(-0.5,0)  
(-0.5,-1.732)node[below,black,scale=0.7]{$k+1$}
(-0.5,0)--(1,0)node[below,midway,black,scale=0.7]{$k$}

(1,0)..controls(1.5,-1)and(2.5,-1)..(3,0) node[below,midway,black,scale=0.7]{$k-1$}

(4.5,-1.732)--(4.5,0)  
(4.5,-1.732)node[below,black,scale=0.7]{$k+1$}
(3,0)--(4.5,0)node[below,midway,black,scale=0.7]{$k$}

(-0.5,0)..controls(0,2)and(1,2.5)..(2,2.5)
node[above,midway,black,scale=0.7]{$1$}
(4.5,0)..controls(4,2)and(3,2.5)..(2,2.5)
node[above,midway,black,scale=0.7]{$1$}
(2,2.5)--(2,1)node[left,midway,black,scale=0.7]{$2$}
(2,1)..controls(1.5,1)and(1,0.5)..(1,0)
node[left,pos=0.4,black,scale=0.7]{$1$}
(2,1)..controls(2.5,1)and(3,0.5)..(3,0)
node[right,pos=0.4,black,scale=0.7]{$1$};
\end{Gtric}
}

\def\TTcalcgg
{\begin{Gtric}
\draw[scale=0.7] (-0.5,-1.732)--(-0.5,0)  
(-0.5,-1.732)node[below,black,scale=0.7]{$k+1$}
(-0.5,0)--(1,0)node[below,midway,black,scale=0.7]{$k$}

(1,0)..controls(1.5,-1)and(2.5,-1)..(3,0) node[below,midway,black,scale=0.7]{$k+1$}

(4.5,-1.732)--(4.5,0)  
(4.5,-1.732)node[below,black,scale=0.7]{$k+1$}
(3,0)--(4.5,0)node[below,midway,black,scale=0.7]{$k$}

(-0.5,0)..controls(0,2)and(1,2.5)..(2,2.5)
node[above,midway,black,scale=0.7]{$1$}
(4.5,0)..controls(4,2)and(3,2.5)..(2,2.5)
node[above,midway,black,scale=0.7]{$1$}
(2,2.5)--(2,1)node[left,midway,black,scale=0.7]{$2$}
(2,1)..controls(1.5,1)and(1,0.5)..(1,0)
node[left,pos=0.4,black,scale=0.7]{$1$}
(2,1)..controls(2.5,1)and(3,0.5)..(3,0)
node[right,pos=0.4,black,scale=0.7]{$1$};
\end{Gtric}
}

\def\TTcalch
{\begin{Gtric}
\draw[scale=0.7] (-0.5,-1.732)--(-0.5,0)  
(-0.5,-1.732)node[below,black,scale=0.7]{$k+1$}
(-0.5,0)--(1,0)node[below,midway,black,scale=0.7]{$k$}

(1,0)..controls(1.5,-1.5)and(2.5,-1.5)..(3,0) node[below,midway,black,scale=0.7]{$k-1$}

(4.5,-1.732)--(4.5,0)  
(4.5,-1.732)node[below,black,scale=0.7]{$k+1$}
(3,0)--(4.5,0)node[below,midway,black,scale=0.7]{$k$}

(-0.5,0)..controls(0,4)and(1,4)..(2,4)
node[above,midway,black,scale=0.7]{$1$}
(4.5,0)..controls(4,4)and(3,4)..(2,4)
node[above,midway,black,scale=0.7]{$1$}

(2,2)--(2,1)node[left,midway,black,scale=0.7]{$2$}
(2,2)..controls(1.5,2)and(1.5,3)..(2,3)
node[left,midway,black,scale=0.7]{$1$}
(2,2)..controls(2.5,2)and(2.5,3)..(2,3)
node[right,midway,black,scale=0.7]{$1$}
(2,3)--(2,4)node[left,midway,black,scale=0.7]{$2$}

(2,1)..controls(1.5,1)and(1,0.5)..(1,0)
node[left,pos=0.4,black,scale=0.7]{$1$}
(2,1)..controls(2.5,1)and(3,0.5)..(3,0)
node[right,pos=0.4,black,scale=0.7]{$1$};
\end{Gtric}
}

\def\TTcalci
{\begin{Gtric}
\draw[scale=0.7] (0.5,-1.732)--(0.5,0)  
(0.5,-1.732)node[below,black,scale=0.7]{$k+1$}
(0.5,0)--(3.5,0)node[below,midway,black,scale=0.7]{$k$}

(3.5,-1.732)--(3.5,0)  
(3.5,-1.732)node[below,black,scale=0.7]{$k+1$}

(0.5,0)..controls(0.5,4)..(2,4)
node[left,pos=0.3,black,scale=0.7]{$1$}
(3.5,0)..controls(3.5,4)..(2,4)
node[right,pos=0.3,black,scale=0.7]{$1$}

(2,1)..controls(1.5,1)and(1.5,2)..(2,2)
(2,1)..controls(2.5,1)and(2.5,2)..(2,2)
(2,2)node[right,midway,black,scale=0.7]{$1$}
(2,2)--(2,4)node[left,midway,black,scale=0.7]{$2$};
\end{Gtric}
}

\def\TTcalcj
{\begin{Gtric}
\draw[scale=0.7] (-0.5,-1.732)--(-0.5,0)  
(-0.5,-1.732)node[below,black,scale=0.7]{$k+1$}
(-0.5,0)--(3,0)node[below,midway,black,scale=0.7]{$k-1$}

(-0.5,0)--(-0.5,1)node[left,midway,black,scale=0.7]{$2$}

(4.5,-1.732)--(4.5,0)  
(4.5,-1.732)node[below,black,scale=0.7]{$k+1$}
(3,0)--(4.5,0)node[below,midway,black,scale=0.7]{$k$}

(-0.5,1)..controls(0,2)and(1,2.5)..(2,2.5)
node[above,midway,black,scale=0.7]{$1$}
(4.5,0)..controls(4,2)and(3,2.5)..(2,2.5)
node[above,midway,black,scale=0.7]{$1$}
(2,2.5)--(2,1)node[right,midway,black,scale=0.7]{$2$}
(2,1)--(-0.5,1) node[above,pos=0.4,black,scale=0.7]{$1$}
(2,1)..controls(2.5,1)and(3,0.5)..(3,0)
(2.65,0.75)node[right,black,scale=0.7]{$1$};
\end{Gtric}
}

\def\TTcalck
{\begin{Gtric}
\draw[scale=0.7] (0.5,-1.732)--(0.5,0)  
(0.5,-1.732)node[below,black,scale=0.7]{$k+1$}
(0.5,0)--(3,0)node[below,midway,black,scale=0.7]{$k-1$}

(0.5,0)--(0.5,1)node[left,midway,black,scale=0.7]{$2$}

(4.5,-1.732)--(4.5,0)  
(4.5,-1.732)node[below,black,scale=0.7]{$k+1$}
(3,0)--(4.5,0)node[below,midway,black,scale=0.7]{$k$}

(0.5,1)..controls(0.7,2)and(1.2,2.5)..(2,2.5)
node[above,black,scale=0.7]{$1$}
(4.5,0)..controls(4,2)and(3,2.5)..(2,2.5)

(2,1)--(0.5,1) node[above,pos=0.2,black,scale=0.7]{$1$}
(2,1)..controls(2.5,1)and(3,0.5)..(3,0);
\end{Gtric}
}

\def\TTcalcl
{\begin{Gtric}
\draw[scale=0.7] (-0.5,-1.732)--(-0.5,0)  
(-0.5,-1.732)node[below,black,scale=0.7]{$k+1$}
(0.7,0)--(3.3,0)node[below,midway,black,scale=0.7]{$k-1$}
(-0.5,0)--(0.7,0)node[below,midway,black,scale=0.7]{$k$}

(4.5,-1.732)--(4.5,0)  
(4.5,-1.732)node[below,black,scale=0.7]{$k+1$}
(3.3,0)--(4.5,0)node[below,midway,black,scale=0.7]{$k$}

(-0.5,0)..controls(0,2)and(1,2.5)..(2,2.5)
node[above,black,scale=0.7]{$1$}
(4.5,0)..controls(4,2)and(3,2.5)..(2,2.5)

(2,1)..controls(1.5,1)and(0.7,0.5)..(0.7,0) (2,1)node[above,black,scale=0.7]{$1$}
(2,1)..controls(2.5,1)and(3.3,0.5)..(3.3,0);
\end{Gtric}
}

\def\GrayQuaVertexa
{\begin{Gtric}
\draw[scale=0.7] (-1.732,-2)--(0,-1)--(1.732,-2) (1.732,2)--(0,1)--(-1.732,2) 
(-1.732,-2)node[below,black,scale=0.7]{$1$}
(1.732,-2)node[below,black,scale=0.7]{$1$}
(1.732,2)node[above,black,scale=0.7]{$1$}
(-1.732,2) node[above,black,scale=0.7]{$1$}
(0,-1)--(0,1)node[left,midway,black,scale=0.7]{$2$};
\end{Gtric}
}

\def\GrayQuaVertexb
{\begin{Gtric}
\draw  [scale=0.7] 
(-2,2)..controls(-1,1)and(1,1)..(2,2) node[above,black,midway,scale=0.7]{$1$}
(2,-2)..controls(1,-1)and(-1,-1)..(-2,-2) node[below,black,midway,scale=0.7]{$1$};
\end{Gtric}
}

\def\GrayQuaVertexc
{\begin{Gtric}
\draw[scale=0.7] (-2,-1.732)--(-1,0)--(-2,1.732) (2,1.732)--(1,0)--(2,-1.732) 
(-2,-1.732)node[below,black,scale=0.7]{$1$}
(-2,1.732)node[above,black,scale=0.7]{$1$}
(2,1.732)node[above,black,scale=0.7]{$1$}
(2,-1.732) node[below,black,scale=0.7]{$1$}
(-1,0)--(1,0)node[below,midway,black,scale=0.7]{$2$};
\end{Gtric}
}

\def\GrayQuaVertexd
{\begin{Gtric}
\draw  [scale=0.7] 
(-2,2)..controls(-1,1)and(-1,-1)..(-2,-2) node[left,black,midway,scale=0.7]{$1$}
(2,-2)..controls(1,-1)and(1,1)..(2,2) node[right,black,midway,scale=0.7]{$1$};
\end{Gtric}
}

\begin{equation}\label{E:gray-Assoc}
    \genSkeinassocia = \genSkeinassoci  .
\end{equation}

\def\GraySkeinak
{\begin{Gtric}
\draw (0.7,0) circle (0.7);
\draw (0,0)node[left,black,scale=0.7]{$k$}; 
\end{Gtric}
}

\def\GraySkeingg
{\begin{Gtric}
\draw[scale=0.8] (0,-1)--(0,0) (0,0)..controls(0.7,0.5)and(0.7,1.5)..(0,1.5)..controls(-0.7,1.5)and(-0.7,0.5)..(0,0);
\draw (0,-0.8)node[below,black,scale=0.7]{$2$};
\draw (-0.4,0.7)node[left,black,scale=0.7]{$1$};
\end{Gtric}
}

\def\GraySkeinbigon
{\begin{Gtric}
\draw[scale= 0.8] (0,0.7)--(0,1.5) (0,-0.7)--(0,-1.5)
        (0,1.5)node[above,black,scale=0.7]{$k$}
        (0,-1.5)node[below,black,scale=0.7]{$k$}
      (0,0.7)..controls(-0.5,0.7)and(-0.5,-0.7)..(0,-0.7) node[left,midway,black,scale=0.7]{$1$} (0,0.7)..controls(0.5,0.7)and(0.5,-0.7)..(0,-0.7)
      node[right,midway,black,scale=0.7]{$k-1$} ;
\end{Gtric}
}

\def\GraySkeinbigona
{\begin{Gtric}
\draw [scale=0.8] (0,1.5)--(0,-1.5)  node[below,black,scale=0.7]{$k$} ;
\end{Gtric}
}

\def\GrayAltbigon
{\begin{Gtric}
\draw[scale= 0.8] (0,0.7)--(0,1.5) (0,-0.7)--(0,-1.5)
        (0,1.5)node[above,black,scale=0.7]{$k$}
        (0,-1.5)node[below,black,scale=0.7]{$k$}
      (0,0.7)..controls(-0.5,0.7)and(-0.5,-0.7)..(0,-0.7) node[left,midway,black,scale=0.7]{$1$} (0,0.7)..controls(0.5,0.7)and(0.5,-0.7)..(0,-0.7)
      node[right,midway,black,scale=0.7]{$k+1$} ;
\end{Gtric}
}

\def\GrayAltbigona
{\begin{Gtric}
\draw [scale=0.8] (0,1.5)--(0,-1.5)  node[below,black,scale=0.7]{$k$} ;
\end{Gtric}
}

\def\GrayAltbigonb
{\begin{Gtric}
\draw[scale= 0.8] 
        (0,0.7)..controls(0,1.5)and(-2,1.5)..(-2,0) 
        (0,-0.7)..controls(0,-1.5)and(-2,-1.5)..(-2,0)
        (-2,0)node[left,black,scale=0.7]{$k$}
      (0,0.7)..controls(-0.5,0.7)and(-0.5,-0.7)..(0,-0.7) node[left,midway,black,scale=0.7]{$1$} (0,0.7)..controls(0.5,0.7)and(0.5,-0.7)..(0,-0.7)
      node[right,midway,black,scale=0.7]{$k+1$} ;
\end{Gtric}
}

\def\GrayZeroBigon
{\begin{Gtric}
   \draw[scale= 0.8] (0,0.7)--(0,1.5) (0,-0.7)--(0,-1.5)
        (0,1.5)node[above,black,scale=0.7]{$k$}
        (0,-1.5)node[below,black,scale=0.7]{$k+2$}
      (0,0.7)..controls(-0.5,0.7)and(-0.5,-0.7)..(0,-0.7) node[left,midway,black,scale=0.7]{$1$} (0,0.7)..controls(0.5,0.7)and(0.5,-0.7)..(0,-0.7)
      node[right,midway,black,scale=0.7]{$k+1$} ;
\end{Gtric}
}

\def\GrayCirK
{\begin{Gtric}
\draw (0.7,0) circle (0.7);
\draw (0.7,0.7)node[above,black,scale=0.7]{$k$}
      (0.7,-0.7)node[below,black,scale=0.7]{\phantom{x}}; 
\end{Gtric}
}

\def\GrayCirKplus
{\begin{Gtric}
\draw (0.7,0) circle (0.7);
\draw (0.7,0.7)node[above,black,scale=0.7]{$k+1$}
      (0.7,-0.7)node[below,black,scale=0.7]{\phantom{x}}; 
\end{Gtric}
}

\noindent Lemma \ref{L:graycircle} says that

\begin{equation}\label{E:gray-circle}
 \GraySkeinak = \frac{[2m-4k][m]}{[m-2k][2m]}
\begin{bmatrix}
m \\
k 
\end{bmatrix} _{q^2} , 
\end{equation}
and from Lemma \ref{L:graybigon}  we find, 

\begin{equation}\label{E:grey-skein-bigon}
\GraySkeinbigon  \ = \ \frac{[2k]}{[2]} \ \GraySkeinbigona \ \ .   
\end{equation}
Since $\Hom_{U_{\A}(\om)}(\Lambda^{k+2}_{\A}, \Lambda^k_{\A}) = 0$, we also have
\begin{equation}\label{E:grey-skein-zerobigon}
 \GrayZeroBigon = 0 .
\end{equation}
Let $k=0$ in \EQ \eqref{E:grey-skein-zerobigon}, we have 
\begin{equation}\label{E:grey-skein-monogon}
 \GraySkeingg = 0 .
\end{equation}

\begin{lemma} 
    \begin{align}   \GrayAltbigon =   \frac{[2m-2k][2m-4k-4][m-2k]}{[2][m-2k-2][2m-4k]} \GrayAltbigona \label{Grayreversebigon}
     \end{align} 
\end{lemma}

\begin{proof}
We temporarily work over $\K$. Since  $\Hom_{U_{\K}(\om)}(\Lambda_{\K}^k,  \Lambda_{\K}^k)$ is 1-dimensional, there exists $\alpha \in\ \K$ such that
\[
\GrayAltbigon =   \alpha \GrayAltbigona.
\]
If we show \EQ \eqref{Grayreversebigon} is true over $\K$, then in particular $\alpha\in \A$, so Lemma \ref{L:sufficestocheckoverK} implies the equation also holds in $\FundA$.

Observe that 
$$\GrayAltbigonb = \alpha \ \GrayCirK
  = \alpha \ \frac{[2m-4k][m]}{[m-2k][2m]}
\begin{bmatrix}
m \\
k 
\end{bmatrix} _{q^2} .$$ On the other hand, 

$$\GrayAltbigonb = \frac{[2k+2]}{[2]} \  \GrayCirKplus
  = \frac{[2k+2]}{[2]}  \ \frac{[2m-4k-4][m]}{[m-2k-2][2m]}
\begin{bmatrix}
m \\
k+1 
\end{bmatrix} _{q^2} .$$ Thus, 
\[
\alpha \ \frac{[2m-4k][m]}{[m-2k][2m]}
\begin{bmatrix}
m \\
k 
\end{bmatrix} _{q^2}\cdot \id_{\Lambda^0_{\K}}  =
\frac{[2k+2]}{[2]}  \ \frac{[2m-4k-4][m]}{[m-2k-2][2m]}
\begin{bmatrix}
m \\
k+1 
\end{bmatrix} _{q^2}\cdot \id_{\Lambda^0_{\K}},
\]
and since $\id_{\Lambda^0_{\K}}\ne 0$, we can compare coefficients and solve for $\alpha$ from
\[
\alpha \ \frac{[2m-4k][m]}{[m-2k][2m]}
\begin{bmatrix}
m \\
k 
\end{bmatrix} _{q^2}  =
\frac{[2k+2]}{[2]}  \ \frac{[2m-4k-4][m]}{[m-2k-2][2m]}
\begin{bmatrix}
m \\
k+1 
\end{bmatrix} _{q^2}.
\]
  
\end{proof}

\def\GrayBraidTriv
{\begin{Gtric}
\draw (-1.5,1.5)--(1.5,-1.5) ;
\draw (1.5,1.5)--(-1.5,-1.5);
\end{Gtric}
}

\def\GrayBraida
{\begin{Gtric}
\draw [scale=0.7] 
(-2,2)..controls(-1,1)and(-1,-1)..(-2,-2) node[left,black,midway,scale=0.7]{$1$}
(2,-2)..controls(1,-1)and(1,1)..(2,2) node[right,black,midway,scale=0.7]{$1$};
\end{Gtric}
}

\def\GrayBraid
{\begin{Gtric}
\draw (-1.5,1.5)--(1.5,-1.5) ;
\draw[double=darkwhite,ultra thick,white,double distance=0.8pt,line width=3pt] (1.5,1.5)--(-1.5,-1.5);
\end{Gtric}
}

\def\GrayBraidb
{\begin{Gtric}
\draw[scale=0.7] (-1.732,-2)--(0,-1)--(1.732,-2) (1.732,2)--(0,1)--(-1.732,2) 
(-1.732,-2)node[below,black,scale=0.7]{$1$}
(1.732,-2)node[below,black,scale=0.7]{$1$}
(1.732,2)node[above,black,scale=0.7]{$1$}
(-1.732,2) node[above,black,scale=0.7]{$1$}
(0,-1)--(0,1)node[left,midway,black,scale=0.7]{$2$};
\end{Gtric}
}

\def\GrayBraidc
{\begin{Gtric}
\draw [scale=0.7] 
(-2,2)..controls(-1,1)and(1,1)..(2,2) node[above,black,midway,scale=0.7]{$1$}
(2,-2)..controls(1,-1)and(-1,-1)..(-2,-2) node[below,black,midway,scale=0.7]{$1$};
\end{Gtric}
}

\def\GrayNegBraid
{\begin{Gtric}
\draw  (1.5,1.5)--(-1.5,-1.5);
\draw[double=darkwhite,ultra thick,white,double distance=0.8pt,line width=3pt] (-1.5,1.5)--(1.5,-1.5);
\end{Gtric}
}

\def\GraySkeinIH
{\begin{Gtric}
\draw[scale=0.7] (-2,-1.732)--(-1,0)--(-2,1.732) (2,1.732)--(1,0)--(2,-1.732) 
(-2,-1.732)node[below,black,scale=0.7]{$k$}
(-2,1.732)node[above,black,scale=0.7]{$1$}
(2,1.732)node[above,black,scale=0.7]{$1$}
(2,-1.732) node[below,black,scale=0.7]{$k$}
(-1,0)--(1,0)node[below,midway,black,scale=0.7]{$k+1$};
\end{Gtric}
}

\def\GraySkeinIHa
{\begin{Gtric}
\draw [scale=0.7]  (-1.74,2)--(0,1)--(1.74,2) 
(1.74,-3)--(0,0) node[right,pos=0.7,black,scale=0.7]{$1$}
(0,0)--(-1.74,-3) node[left,pos=0.3,black,scale=0.7]{$1$}
(0,0)--(0,1) node[left,midway,black,scale=0.7]{$2$}
(-1.74,-3) node[below,black,scale=0.7]{$k$}
(1.74,-3) node[below,black,scale=0.7]{$k$}
(-1.74,2) node[above,black,scale=0.7]{$1$}
(1.74,2) node[above,black,scale=0.7]{$1$}
(1.16,-2)--(-1.16,-2)node[below,midway,black,scale=0.7]{$k-1$};
\end{Gtric}
}

\def\GraySkeinIHb
{\begin{Gtric}
\draw[scale=0.7] (-2,-1.732)--(-1,0)--(-2,1.732) (2,1.732)--(1,0)--(2,-1.732) 
(-2,-1.732)node[below,black,scale=0.7]{$k$}
(-2,1.732)node[above,black,scale=0.7]{$1$}
(2,1.732)node[above,black,scale=0.7]{$1$}
(2,-1.732) node[below,black,scale=0.7]{$k$}
(-1,0)--(1,0)node[below,midway,black,scale=0.7]{$k-1$};
\end{Gtric}
}

\def\GraySkeinIHbm
{\begin{Gtric}
\draw[scale=0.7] (-2,-1.732)--(-1,0)--(-2,1.732) (2,1.732)--(1,0)--(2,-1.732) 
(-2,-1.732)node[below,black,scale=0.7]{$m$}
(-2,1.732)node[above,black,scale=0.7]{$1$}
(2,1.732)node[above,black,scale=0.7]{$1$}
(2,-1.732) node[below,black,scale=0.7]{$m$}
(-1,0)--(1,0)node[below,midway,black,scale=0.7]{$m-1$};
\end{Gtric}
}

\def\GraySkeinIHc
{\begin{Gtric}
\draw  [scale=0.7] 
(-2,2)..controls(-1,1)and(1,1)..(2,2) node[above,black,midway,scale=0.7]{$1$}
(2,-2)..controls(1,-1)and(-1,-1)..(-2,-2) node[below,black,midway,scale=0.7]{$k$};
\end{Gtric}
}

\def\GraySkeinIHcm
{\begin{Gtric}
\draw  [scale=0.7] 
(-2,2)..controls(-1,1)and(1,1)..(2,2) node[above,black,midway,scale=0.7]{$1$}
(2,-2)..controls(1,-1)and(-1,-1)..(-2,-2) node[below,black,midway,scale=0.7]{$m$};
\end{Gtric}
}

\begin{lemma}
 
\begin{align}
   \GrayQuaVertexa 
   + \frac{[2m-8][m-2]}{[m-4][2m-4]}\GrayQuaVertexb   = \GrayQuaVertexc
   +\frac{[2m-8][m-2]}{[m-4][2m-4]}   \GrayQuaVertexd \label{grayIHbase}
\end{align}   

\end{lemma}

\begin{proof}

By Definition \ref{D:graybraiding} and Lemma \ref{gray-twistbraiding-is-inverse}, we know that the $90$ degree rotation of $\brrep$ is the same as $\brrepi$, so
\begin{align*}
 q^2  \GrayBraidc \ &- \GrayQuaVertexc
- \frac{[m-2]}{[2m-4]}(q^2-q^{-2})\cdot q^{-m+2} \ \GrayBraida \\ 
&= q^{-2} \ \GrayBraida \ - \GrayBraidb 
+ \frac{[m-2]}{[2m-4]}(q^2-q^{-2})\cdot q^{m-2} \GrayBraidc \quad .
\end{align*}
We obtain Equation \eqref{grayIHbase} by combining terms and using the identities in Remark \ref{R:identities}.  
\end{proof}

\def\SmallTriangle 
{\begin{Gtric}
\draw [scale=0.7]
(1.74,-3)--(0,0) node[right,pos=0.7,black,scale=0.7]{$1$}
(0,0)--(-1.74,-3) node[left,pos=0.3,black,scale=0.7]{$1$}
(0,0)--(0,1.5) node[left,midway,black,scale=0.7]{$2$}
(-1.74,-3) node[below,black,scale=0.7]{$1$}
(1.74,-3) node[below,black,scale=0.7]{$1$}
(1.16,-2)--(-1.16,-2)node[below,midway,black,scale=0.7]{$2$};
\end{Gtric}
}
 
\def\SmallTrianglea
{\begin{Gtric}
\draw [scale=0.7] 
(1.16,-2) node[below,black,scale=0.7]{$1$}
(1.16,-2)--(0,0) 
(-1.16,-2) node[below,black,scale=0.7]{$1$}
(0,0)--(-1.16,-2)
(0,0)--(0,2) node[left,midway,black,scale=0.7]{$2$};
\end{Gtric}
}

\def\TbyS
{\begin{Gtric}
\draw
(-90:1)--(-150:2) node[below,midway,black,scale=0.7]{$2$} 
(-90:1)--(-30:2) node[below,midway,black,scale=0.7]{$1$} 
(-150:2)--(150:1) node[left,black,scale=0.7]{$1$}
(150:1)--(0,2)
(-30:2)--(30:1) node[right,midway,black,scale=0.7]{$k$} 
(30:1)--(0,2)   node[right,midway,black,scale=0.7]{$k+1$} 
(0,2)--(0,3.5) node[left,midway,black,scale=0.7]{$k+2$}
(-150:2)--(-150:3.5) node[below,black,scale=0.7]{$1$}
(-30:2)--(-30:3.5) node[below,black,scale=0.7]{$k+1$}
(30:1)--(-90:1)node[left,midway,black,scale=0.7]{$1$};
\end{Gtric}
}

\def\TbySa
{\begin{Gtric}
\draw
(-90:1)--(-150:2) node[below,midway,black,scale=0.7]{$2$} 
(-90:1)--(-30:2) node[below,midway,black,scale=0.7]{$1$} 
(-150:2)--(150:1) node[left,midway,black,scale=0.7]{$1$}
(150:1)--(0,2)  node[left,midway,black,scale=0.7]{$2$}
(-30:2)--(30:1) node[right,black,scale=0.7]{$k$} 
(30:1)--(0,2)    
(0,2)--(0,3.5) node[left,midway,black,scale=0.7]{$k+2$}
(-150:2)--(-150:3.5) node[below,black,scale=0.7]{$1$}
(-30:2)--(-30:3.5) node[below,black,scale=0.7]{$k+1$}
(-90:1)--(150:1)node[right,midway,black,scale=0.7]{$1$} ;
\end{Gtric}
}

\def\TbySb
{\begin{Gtric}
\draw [scale=0.5]
(-150:2)--(-90:1) node[below,black,scale=0.7]{$1$} 
(-90:1)--(-30:2) 
(-150:2)--(150:1) node[left,black,scale=0.7]{$2$}
(150:1)--(0,2)  
(-30:2)--(30:1) node[right,black,scale=0.7]{$k$} 
(30:1)--(0,2)    
(0,2)--(0,4) node[above,black,scale=0.7]{$k+2$}
(-150:2)--(-150:4) node[below,black,scale=0.7]{$1$}
(-30:2)--(-30:4) node[below,black,scale=0.7]{$k+1$} ;
\end{Gtric}
}

\def\TbySc
{\begin{Gtric}
\draw [scale=0.5]
(-30:2)..controls(0.4,-0.2)and(0.2,0)..(0.5,2) node[left,midway,black,scale=0.7]{$1$} 
(0,4)--(-2,-2) node[below,black,scale=0.7]{$1$}
(0,4)--(0,6) node[above,black,scale=0.7]{$k+2$}

(-30:2)--(0.5,2)   node[right,midway,black,scale=0.7]{$k$}   
(0.5,2)--(0,4) node[right,midway,black,scale=0.7]{$k+1$}
(-30:2)--(-30:4) node[below,black,scale=0.7]{$k+1$} ;
\end{Gtric}
}

\def\TbySd
{\begin{Gtric}
\draw [scale=0.7] 
(1.16,-2) node[below,black,scale=0.7]{$k+1$}
(1.16,-2)--(0,0) 
(-1.16,-2) node[below,black,scale=0.7]{$1$}
(0,0)--(-1.16,-2)
(0,0)--(0,2) node[above,black,scale=0.7]{$k+2$};
\end{Gtric}
}

\def\LASkeinIHbRotate
{\begin{Gtric}
\draw[scale=0.7] (-1.732,-2)--(0,-1)  (0,1)--(-1.732,2) 
(0,-1)..controls(2,-1)and(2,1)..(0,1) node[right,midway,black,scale=0.7]{$1$}
(-1.732,-2)node[below,black,scale=0.7]{$k+1$}
(-1.732,2) node[above,black,scale=0.7]{$k+1$}
(0,-1)--(0,1)node[left,midway,black,scale=0.7]{$k$};
\end{Gtric}
}

\def\LASkeinIHcRotate
{\begin{Gtric}
\draw  [scale=0.7] 
(0.5,0) circle (1) 
(1.5,0) node[right,black,scale=0.7]{$1$}
(-2,2)node[above,black,scale=0.7]{$\phantom{k+1}$} 
(-2,2)..controls(-1,1)and(-1,-1)..(-2,-2) node[below,black,scale=0.7]{$k+1$};
\end{Gtric}
}

\begin{lemma}\label{smalltriangle}
\begin{align} \label{twooneone}
   \SmallTriangle=\frac{[2m][m-2]}{[m][2m-4]} \SmallTrianglea 
\end{align} 
\end{lemma}
\begin{proof}
This is immediate from Equation \eqref{grayIHbase}, Equation \eqref{E:grey-skein-bigon}, and Equation \eqref{E:grey-skein-monogon}.
\end{proof}

\begin{lemma}
\begin{equation}\label{TrbySq}
\LCSkeinIHa = \frac{[2k+2][2m][m-2]}{[2][m][2m-4]} \LCSkeinIHlem 
\end{equation}
\end{lemma}
\begin{proof}
\begin{align*}
  \TbyS &\stackrel{\eqref{E:gray-Assoc}}{=} \TbySa 
  \stackrel{\eqref{twooneone}}{=} \
  \frac{[2m][m-2]}{[m][2m-4]} \TbySb \\
  &\stackrel{\eqref{E:gray-Assoc}}{=}  \ \frac{[2m][m-2]}{[m][2m-4]}\TbySc  
  \stackrel{\eqref{E:grey-skein-bigon}}{=} \ 
  \frac{[2k+2][2m][m-2]}{[2][m][2m-4]} \TbySd
\end{align*} 
\end{proof}

\begin{lemma}\label{L:ne-to-zero}
For $k+1 < m$, 
\begin{equation*}
\LBSkeinIHb \neq 0.
\end{equation*}
\end{lemma}
\begin{proof}
Since this triangle is part of the left hand side of \EQ \eqref{TrbySq}, and the right hand side of \EQ \eqref{TrbySq} is a non-zero scalar multiple of the map $m_{1,k+1}^{k+2}$, which, as long as $k+2\le m$, is non-zero by Remark \ref{R:m-is-non-zero}. 
\end{proof}

\begin{lemma}\label{L:e-to-zero}
\begin{equation}\label{E:mm-triangle}
\ZeroTriangleAtBound = 0
\end{equation}
\end{lemma}
\begin{proof}
We know that $\Lambda_{\K}^m\otimes \Lambda_{\K}^m\cong \det_{\K}\otimes \det_{\K} \cong \Lambda_{\K}^0$. Also, Proposition \ref{L:ext-powers-irredicible} implies that
\[
\dim_{\K}\Hom_{U_{\K}(\om)}(\Lambda_{\K}^0, \Lambda_{\K}^2) = 0.
\]
This means that \EQ \eqref{E:mm-triangle} holds after applying $\base_{\K}$, and the claim follows from Lemma \ref{L:sufficestocheckoverK}.
\end{proof}

\begin{proposition}\label{P:gray-IH}
Suppose that $k\le m$. Then the following equation holds in $\FundA$.  \begin{align}
\GraySkeinIHa =  \GraySkeinIH + \frac{[2m-4k-4][m-2k]}{[m-2k-2][2m-4k]}  \GraySkeinIHb
  -  \frac{[2m-4k-4][m-2]}{[m-2k-2][2m-4]} \GraySkeinIHc \label{grayIH}
  \end{align}
\end{proposition}

\begin{proof}
We prove Equation \eqref{grayIH}  by induction. We verified the base case, that \EQ \eqref{grayIH} holds when $k=1$, in Lemma \ref{grayIHbase}. Suppose \EQ \eqref{grayIH} holds in $\FundA$ for $k\ge 1$, then we will show that it is true for $k+1$. Since $\Lambda_{\A}^{k+1}=0$ when $k+1>m$, we may assume $k+1\le m$, in order for \EQ \eqref{grayIH} to be non-trivial in $\FundA$. Suppose that $k+1 < m$. We temporarily work over $\K$. It follows from Lemma \ref{L:pieri} that $\Lambda_{\K}^{k+1}\otimes V_{\K} \cong \Lambda_{\K}^{k+2}\oplus \Lambda_{\K}^{k}\oplus L_{\K}(\varpi_1+\varpi_k, +1)$. Therefore, the merge-split to $k+2$, the merge-split to $k$, and $\id_{\Lambda^{k+1}_{\K}\otimes V_{\K}}$ form a basis for $\End_{U_{\K}(\om)}(\Lambda_{\K}^i\otimes V_{\K})$. In particular, there are $x, y, z\in \K$ such that 
\begin{equation}\label{E:HI-xyz}
\NextSkeinIHa  =  x \NextSkeinIH + y \NextSkeinIHb + z \NextSkeinIHc \ \ .
\end{equation}
Now, in order to find a linear system of equations about $x$,$y$,$z$, we first attach caps or trivalent vertices, in three different ways, to each term of \EQ \eqref{E:HI-xyz}.

$$\LASkeinIHa =  x \LASkeinIH + y \LASkeinIHb + z \LASkeinIHc $$

$$\LBSkeinIHa =  x \LBSkeinIH  + y \LBSkeinIHb + z \LBSkeinIHc $$

$$\LCSkeinIHa =  x \LCSkeinIH  + y \LCSkeinIHb + z \LCSkeinIHc $$

In order to simplify the second equation, we need to relate the two triangles on the right hand side. Using results in \cite{KoikeTerada}, one can check that since $k+1< m$, $\dim_{\K}\Hom_{U_{\K}(\om)}(\Lambda^{k+1}_{\K}\otimes \Lambda^{2}_{\K}, \Lambda_{\K}^{k+1})=1$. By Lemma \ref{L:ne-to-zero} $$ \LBSkeinIHb \neq 0 \  . $$ So there exists $\tau\in \K$ such that  
 \begin{align}
     \LBSkeinIH = \tau \LBSkeinIHb. \label{TrickyTriangle}
 \end{align} 

On one hand, we know: 
\begin{align*}
\TTcalc      
             &\stackrel{\eqref{TrbySq}}{=} \
              \frac{[2k+2][2m][m-2]}{[2][m][2m-4]} \TTcalca \\ 
             &\stackrel{\eqref{Grayreversebigon}}{=} \
              \frac{[2k+2][2m][m-2][2m-2k-2][2m-4k-8][m-2k-2]}
                   {[2][m][2m-4][2][m-2k-4][2m-4k-4]} \TTcalcb
\end{align*}

On the other hand, we know that
\begin{align*}
\TTcalcc\ \stackrel{\eqref{grayIHbase}}{=} \
& -\frac{[2m-8][m-2]}{[m-4][2m-4]}\TTcalcd + \TTcalce \\ 
& + \frac{[2m-8][m-2]}{[m-4][2m-4]}\TTcalcf ,
\end{align*}
where 
\begin{align*}
   \TTcalcd &\stackrel{\eqref{E:grey-skein-bigon}}{=} \ 
   \left(\frac{[2k+2]}{[2]}\right)^2  \TTcalcb \\
   \TTcalcf 
&\stackrel{\substack{\eqref{Grayreversebigon} \\ \eqref{E:grey-skein-bigon}}}{=} \ 
\frac{[2m-2k][2m-4k-4][m-2k]}{[2][m-2k-2][2m-4k]}\frac{[2k+2]}{[2]} \TTcalcb\ {\ }_{.} 
\end{align*}
By the inductive hypothesis,
\begin{align*}
&\TTcalcgg \stackrel{\eqref{grayIH}}{=} \ 
\TTcalch- \frac{[2m-4k-4][m-2k]}{[m-2k-2][2m-4k]}\TTcalcg  
 \\ &+ \frac{[2m-4k-4][m-2]}{[m-2k-2][2m-4]} \TTcalci
\qquad \quad  \stackrel{\substack{\eqref{E:grey-skein-bigon}\\ \eqref{E:grey-skein-monogon}}}{=}  
\frac{[2m-4k+4][m-2k]}{[m-2k+2][2m-4k]} \TTcalcg \\
& \stackrel{\eqref{E:gray-Assoc}}{=} 
\frac{[2m-4k+4][m-2k]}{[m-2k+2][2m-4k]} \TTcalcj
\stackrel{\eqref{twooneone}}{=}  
\frac{[2m-4k+4][m-2k]}{[m-2k+2][2m-4k]} \frac{[2m][m-2]}{[m][2m-4]} \TTcalck \\
&\stackrel{\eqref{E:gray-Assoc}}{=} 
\frac{[2m-4k+4][m-2k][2m][m-2]}{[m-2k+2][2m-4k][m][2m-4]} \TTcalcl \\
&\stackrel{\eqref{E:grey-skein-bigon}}{=}  
\frac{[2m-4k+4][m-2k][2m][m-2]}{[m-2k+2][2m-4k][m][2m-4]} \frac{[2k]}{[2]} \frac{[2k+2]}{[2]} \ \TTcalcb\ {\ }_{.} 
\end{align*}

In conclusion, 
$$  \tau = \frac{[2m-4k][m-2k-2]}{[m-2k][2m-4k-4]} . $$ 

So by applying Equations \eqref{E:gray-circle}, \eqref{E:grey-skein-bigon}, \eqref{E:grey-skein-zerobigon}, \eqref{E:grey-skein-monogon}, \eqref{Grayreversebigon}, \eqref{TrbySq}, and \eqref{TrickyTriangle}, we have the following system of linear equations.

\begin{align*}
  & 0 = \frac{[2m-2k-2][2m-4k-8][m-2k-2]}{[2][m-2k-4][2m-4k-4]} x
   + \frac{[2k+2]}{[2]}y + \frac{[2m-4][m]}{[m-2][2]} z,\\
  & \frac{[4]}{[2]}  
  =  \frac{[2m-4k][m-2k-2]}{[m-2k][2m-4k-4]} x + y, \quad \text{and} \\
  & \frac{[2k+2][2m][m-2]}{[2][m][2m-4]} =  \frac{[2k+4]}{[2]} x + z.
\end{align*}
The values  
\begin{align*}
 x&=1\\
 y&= \frac{[2m-4k-8][m-2k-2]}{[m-2k-4][2m-4k-4]}\\
 z&=  - \frac{[2m-4k-8][m-2]}{[m-2k-4][2m-4]}
\end{align*}
satisfy the equations, and therefore are a unique set of solutions. Since $x,y,z\in \A$, it follows from Lemma \ref{L:sufficestocheckoverK} that Equation \eqref{grayIH} holds in $\FundA$ when $k+1 < m$.

Now, suppose that $k+1=m$.  Lemma \ref{L:pieri} implies that $\Lambda_{\K}^{k+1}\otimes V_{\K} \cong \Lambda_{\K}^{k}$. Thus, there exists $\gamma \in \K $ such that 
\begin{equation} \label{boundaryIH}
    \NextSkeinIHc= \gamma \NextSkeinIHb,
\end{equation} \ so
$$\LASkeinIHcRotate
= \gamma \LASkeinIHbRotate.$$
Using Equations \eqref{E:gray-circle} and \eqref{E:grey-skein-bigon}, we get 
$\cfrac{[2m-4][m]}{[m-2][2]} \cdot \id_{\Lambda^{k+1}_{\K}}
=\gamma \cdot \cfrac{[2k+2]}{[2]} \cdot \id_{\Lambda^{k+1}_{\K}}$, so $\gamma= \cfrac{[2m-4][m]}{[m-2][2m]}\in \A$. It follows from Lemma \ref{L:sufficestocheckoverK} that \EQ \eqref{boundaryIH} also holds in $\FundA$.

On the other hand, consider \EQ \eqref{grayIH} when $k=m$.  Lemma \ref{L:e-to-zero} implies that the left hand side of \EQ \eqref{grayIH} is zero. Also, the first term on the right hand side of \EQ \eqref{grayIH} has a label $m+1$, so is also zero. Thus, \EQ \eqref{grayIH} becomes 
\begin{equation*}
0 = 0 + \frac{[2m-4m-4][m-2m]}{[m-2m-2][2m-4m]}  \GraySkeinIHbm
  -  \frac{[2m-4m-4][m-2]}{[m-2m-2][2m-4]} \GraySkeinIHcm ,
\end{equation*}
which agrees with \EQ \eqref{boundaryIH}. 
\end{proof}

\def\VertoMulti
{\begin{tric}
\draw [scale=0.8] (0,0)--(90:1) (0,0)--(210:1) (0,0)--(330:1);
\draw (90:1)node[black,anchor=south,scale=0.7]{$i+1$}
      (210:1)node[black,anchor=north,scale=0.7]{$i$}
      (330:1)node[black,anchor=north,scale=0.7]{$1$}; 
\end{tric}
}

\def\VertoMultia
{\begin{tric}
\draw [scale=0.8] (0,0)--(90:1) (0,0)--(210:1) (0,0)--(330:1);
\draw (90:1)node[black,anchor=south,scale=0.7]{$i+1$}
      (210:1)node[black,anchor=north,scale=0.7]{$1$}
      (330:1)node[black,anchor=north,scale=0.7]{$i$}; 
\end{tric}
}

\begin{thm}\label{T:existence-of-functor}
Let $\R\in \{\kk, \A, \K\}$. There is a pivotal functor 
\[
\Phi_{\R}: \Web\rightarrow \Fund,
\]
such that
\[
\VertoMulti \mapsto {^\R m}_{i,1}^{i+1}, \quad \VertoMultia \mapsto {^\R m}_{1,i}^{i+1},
 \quad \text{and} \quad \Phi_{\R }(\brweb) = \brrep.
\]
Moreover, we have canonical identifications $\Phi_{\R } = \base_{\R}\circ (\R \otimes \Phi_{\A})$. 
\end{thm}

\begin{proof}
Suppose we have $\Phi_{\A}$ as in the statement of the theorem. Then for $\R\in \{\kk, \K\}$, we define $\Phi_{\R}$ as the composition 
\[
\Web=\R\otimes \WebA\xrightarrow{\R\otimes \Phi_{\A}} \R\otimes \FundA\xrightarrow{\base_{\R}}\Fund.
\]
It is then easy to see that it suffices to prove the result over $\A$. To this end, we just need to check the defining relations in $\WebA$, see \EQ \eqref{defskein}, are satisfied in $\FundA$. This follows from \EQS \eqref{E:gray-circle}, \eqref{E:grey-skein-bigon}, \eqref{E:grey-skein-monogon}, \eqref{E:gray-Assoc}, and \eqref{grayIH}.

Having established everything else in the statement of the theorem, the equality $\Phi_{\R}(\brweb) = \brrep$ follows from comparing Definition \ref{D:graybraiding} and Definition \ref{D:brweb}.
\end{proof}

\subsection{Compatability with classical invariant theory.}

Let $\kk^m$ be a vector space with basis $\{v_1, \dots, v_m\}$ and bilinear form $(v_i, v_j) = \delta_{i,j}$. We write $O(\kk^m)$ for the subgroup of $GL(\kk^m)$ preserving $(-,-)$.

We want to identify $\StdRepC$ with the full monoidal subcategory of $\textbf{Rep}(O(\kk^m))$ generated by $\kk^m$. To this end, consider the $\kk$-basis for $V_{\kk}$:
\[
\begin{cases}
(\sqrt{-1})^{i-1}\left(\frac{a_i + b_i}{2}\right), \quad (\sqrt{-1})^{i-1}\left(\frac{a_i - b_i}{2}\right), \quad i=1, \dots, n, \quad \text{if $m=2n$} \\
(\sqrt{-1})^{i-1}\left(\frac{a_i + b_i}{2}\right), \quad (\sqrt{-1})^{n-1}\frac{u}{\sqrt{2}}, \quad \text{and}\quad (\sqrt{-1})^{i-1}\left(\frac{a_i - b_i}{2}\right), \quad i=1, \dots, n, \quad \text{if $m=2n+1$}.
\end{cases}
\]
This basis gives an identification $\kk^m = V_{\kk}$ under which the form $(v_i,v_j)= \delta_{i,j}$ on $\kk^m$ agrees with the form $(v,w) = {^{\kk}e}_{1}(v\otimes w)$ on $V_{\kk}$.

Let $\mathcal{B}(m)$ be the Brauer category as defined in \cite[Definition 2.4 and Theorem 2,6]{LZbrauercat}. Comparing Lehrer-Zhang's definition of $\mathcal{B}(m)$ with Definition \ref{def:BMW}, it follows that there is an identification $\BMWC=\mathcal{B}(m)$. Lehrer-Zhang prove there is a unique monoidal functor $F:\mathcal{B}(m) \rightarrow \textbf{Rep}(O(\kk^m)) = \RepC$ such that the crossing diagram maps to the tensor flip map, and the cup and cap diagrams map to to the natural homomorphisms constructed with $(-,-)$ \cite[Theorem 3.4]{LZbrauercat}. 

On the other hand, we have constructed a monoidal functor 
\[
\Phi_{\kk}|_{\StdWebC}:\StdWebC\rightarrow \StdRepC.
\]
Using the identification $\BMWC=\mathcal{B}(m)$, Proposition \ref{P:functor-bmw-to-web} gives a monoidal functor $\eta_{\kk}:\mathcal{B}(m)\rightarrow \StdWebC$. By definition, $F\circ \eta_{\kk}$ sends $\brwebC$ to the tensor flip map, and by Lemma \ref{L:brrepC-is-flipmap} $\Phi_{\kk}$ acts the same way. Thus, after identifying $\StdRepC$ with the full monoidal subcategory of $\textbf{Rep}(O(\kk^m))$ generated by $V_{\kk}$, we have $F = \Phi_{\kk}|_{\StdWebC}\circ \eta_{\kk}$. 

\section{Proof of the equivalence}

\subsection{Reduction to standard subcategories}

Let $\R = \K$ or $\kk$. In Theorem \ref{T:existence-of-functor} we showed the existence of a pivotal functor $\Phi_\R :\Web\rightarrow \Fund$. Our goal is to show that the functor $\Phi_\R $ is an equivalence. Essential sujectivity is immediate from the definitions, but we need to work to show $\Phi_\R $ is full and faithful. The first step is to reduce to showing that $\Phi_\R |_{\StdWeb}:\StdWeb\rightarrow \StdRep$ is full and faithful. 

\begin{lemma}\label{L:reduction-to-allones}
Let $\R \in \{\kk, \K\}$. If $\Phi_\R |_{\StdWeb}$ is full, then $\Phi_\R $ is full. If $\Phi_\R |_{\StdWeb}$ is faithful, then $\Phi_\R $ is faithful. 
\end{lemma}
\begin{proof}
Using the merge and split trivalent vertices, it is easy to see that each generating object $k$ in $\Web$ is a direct summand of $1^{\otimes k}$. The claim then follows from \cite[Lemma 5.5]{bodish2021type}. 
\end{proof}


\subsection{Fullness}

It is well known that if $\R = \kk$ or $\K$, then the Brauer algebra, respectively the BMW algebra, is Schur-Weyl dual to $U_\R (\om)$ acting on tensor powers of $V_\R $. An adjunction argument, see \cite[Theorem 5.8]{bodish2021type}, then yields fullness of $\Phi_{\R}|_{\StdWeb}$. We give precise citations in the proof below.

\begin{thm}\label{T:full}
Let $\R \in \{\kk, \K\}$. The functor $\Phi_\R |_{\StdWeb}$ is full.
\end{thm}
\begin{proof}
We first argue for $\R = \kk$. We know that $F$ is full \cite[Theorem 4.8]{LZbrauercat}. Also, $F= \Phi_{\kk}|_{\StdWebC}\circ \eta_{\kk}$, so it follows that $\Phi_{\kk}|_{\StdWebC}$ is full.

Suppose $\R  = \K$. Then by \cite[Theorem 8.5]{LZ-stronglymultifree} the operators $\id\otimes \brrepK\otimes \id$ generate $\End_{U_{\K}(\om)}(V_\R ^{\otimes d})$, for all $d\in \mathbb{Z}_{\ge 0}$. Since $\Phi_{\K}$ is monoidal and $\Phi_{\K}(\brwebK) = \brrepK$, it follows that the map
\[
\Phi_{\K}:\End_{\StdWebK}(1^{\otimes d})\rightarrow \End_{U_{\K}(\om)}(V_{\K}^{\otimes d})
\]
is surjective for all $d\in \mathbb{Z}_{\ge 0}$. Since $\Phi_{\K}$ is pivotal, we can use adjunction in $\StdWebK$ and $\StdRepK$ to deduce that 
\[
\Phi_{\K}:\Hom_{\StdWebK}(1^{\otimes b}, 1^{\otimes c})\rightarrow \Hom_{U_{\K}(\om)}(V_{\K}^{\otimes b}, V_{\K}^{\otimes c})
\]
is surjective for all $b,c\in \mathbb{Z}_{\ge 0}$ such that $b+c$ is even. Since the homomorphism spaces are zero when $b+c$ is odd, it follows that $\Phi_{\K}|_{\StdWebK}$ is full. 
\end{proof}

\subsection{Faithfulness}

Our goal is to show that a functor is faithful, so we will necessarily have to analyze the kernel of a functor. The kernel of a monoidal functor is a monoidal ideal, so we recall some Lemmas about the interactions between monoidal functors and monoidal ideals.

\begin{notation}
Let $\mathcal{C}$ be an $\R$-linear monoidal category and let $x$ be a homomorphism in $\mathcal{C}$. Write $\langle x\rangle$ to denote the monoidal ideal generated by $x$ in $\mathcal{C}$. Given a morphism $y\in \Hom_{\mathcal{C}}(X,Y)$, we write $y\in \langle x\rangle$ if $y\in \Hom_{\langle x\rangle}(X,Y)$. 
\end{notation}

\begin{lemma}\label{L:monoidalfunctor/ideal-compatibility}
Suppose that $G:\mathcal{C}\rightarrow \mathcal{D}$ is an $\R$-linear monoidal functor and $x$ is a morphism in $\mathcal{C}$. Then if $y\in \langle x\rangle$,
then $G(y)\in \langle G(x)\rangle$. 
\end{lemma}
\begin{proof}
Follows from observing that $G$ preserves linear combinations, tensor products, and compositions of morphisms. 
\end{proof}

Classical invariant theory gives us a description of the kernel of the functor $F:\mathcal{B}(m)\rightarrow \textbf{Rep}(O(\kk^m))$. 

\begin{defn}
Given an element $w\in S_k$, we naturally get an element $w\in\End_{\mathcal{B}(m)}(1^{\otimes k})$. Let $a_k:=\frac{1}{k!}\sum_{w\in S_k}(-1)^{\ell(w)}w\in \End_{\mathcal{B}(m)}(1^{\otimes k})$. We will represent the elements $a_k$ and $\eta_{\kk}(a_k)$ graphically by a box labelled by $k$.
\end{defn}

\def\SumTwist{
\begin{tric}
\draw [darkgreen]
(.4,0.4)--(.4,3)   node[above,scale=0.7,black]{$1$}
(0.9,0.4)--(0.9,3)  node[above,scale=0.7,black]{$1$}
(2.5,0.4)--(2.5,3)  node[above,scale=0.7,black]{$1$}

(.4,-0.4)--(.4,-3) node[below,scale=0.7,black]{$1$}
(0.9,-0.4)--(0.9,-3) node[below,scale=0.7,black]{$1$}
(2.5,-0.4)--(2.5,-3) node[below,scale=0.7,black]{$1$}     ;

\filldraw[black] (1.35,1.7) circle (1pt) (1.7,1.7) circle (1pt) (2.05,1.7) circle (1pt) 
(1.35,-1.7) circle (1pt) (1.7,-1.7) circle (1pt) (2.05,-1.7) circle (1pt);

\draw[darkred,thick] (-0.2,-0.4)rectangle(3.1,0.4)  
                     node[pos=0.5,scale=0.7,black]{$k$};
\end{tric}
}

\def\BigIk{
\begin{tric}
\draw [darkgreen](0,0)--(0,1) (0,1)..controls(-1,1)and(-1,1.5)..(-1,2) 
       node[midway,left,scale=0.7,black]{$k-1$}
       (0,1)..controls(1,2)..(1,7)node[above,scale=0.7,black]{$1$}
       (-1,2)..controls(-2,2)and(-2,2.5)..(-2,3) 
        node[midway,left,scale=0.7,black]{$k-2$}
       (-1,2)..controls(0,3)..(0,7)node[above,scale=0.7,black]{$1$};
\filldraw[black] (-2.25,3.25) circle (1pt) 
                 (-2.5,3.5) circle (1pt) (-2.75,3.75) circle (1pt)
      (-0.9,6.5) circle (1pt) (-1.25,6.5) circle (1pt) (-1.6,6.5) circle (1pt); 
\draw [darkgreen]
(-3,4)..controls(-3.5,4)and(-3.5,4.5)..(-3.5,5)node[midway,left,scale=0.7,black]{$3$}
(-3.5,5)..controls(-4.5,5)and(-4.5,5.5)..(-4.5,6)node[midway,left,scale=0.7,black]{$2$}
(-4.5,6)..controls(-5.5,6)and(-5.5,6.5)..(-5.5,7)node[above,scale=0.7,black]{$1$}
(-3.5,5)..controls(-2.5,5.5)..(-2.5,7) node[above,scale=0.7,black]{$1$}
(-4.5,6)..controls(-3.5,6)and(-3.5,6.5)..(-3.5,7)node[above,scale=0.7,black]{$1$};
\draw (0,0)node[left,scale=0.7,black]{$k$};
\draw [darkgreen] (0,0)--(0,-1) (0,-1)..controls(-1,-1)and(-1,-1.5)..(-1,-2) 
       node[midway,left,scale=0.7,black]{$k-1$}
       (0,-1)..controls(1,-2)..(1,-7)node[below,scale=0.7,black]{$1$}
       (-1,-2)..controls(-2,-2)and(-2,-2.5)..(-2,-3) 
        node[midway,left,scale=0.7,black]{$k-2$}
       (-1,-2)..controls(0,-3)..(0,-7)node[below,scale=0.7,black]{$1$};
\filldraw[black] (-2.25,-3.25) circle (1pt) 
                 (-2.5,-3.5) circle (1pt) (-2.75,-3.75) circle (1pt)
      (-0.9,-6.5) circle (1pt) (-1.25,-6.5) circle (1pt) (-1.6,-6.5) circle (1pt); 
\draw [darkgreen]
(-3,-4)..controls(-3.5,-4)and(-3.5,-4.5)..(-3.5,-5)node[midway,left,scale=0.7,black]{$3$}
(-3.5,-5)..controls(-4.5,-5)and(-4.5,-5.5)..(-4.5,-6)node[midway,left,scale=0.7,black]{$2$}
(-4.5,-6)..controls(-5.5,-6)and(-5.5,-6.5)..(-5.5,-7)node[below,scale=0.7,black]{$1$}
(-3.5,-5)..controls(-2.5,-5.5)..(-2.5,-7) node[below,scale=0.7,black]{$1$}
(-4.5,-6)..controls(-3.5,-6)and(-3.5,-6.5)..(-3.5,-7)node[below,scale=0.7,black]{$1$};
\end{tric}
}

\def\SumTwista{
\begin{tric}
\draw [darkgreen]
(.4,0.4)--(.4,3)   node[above,scale=0.7,black]{$1$}
(0.9,0.4)--(0.9,3)  node[above,scale=0.7,black]{$1$}
(2.5,0.4)--(2.5,3)  node[above,scale=0.7,black]{$1$}

(.4,-0.4)--(.4,-3) node[below,scale=0.7,black]{$1$}
(0.9,-0.4)--(0.9,-3) node[below,scale=0.7,black]{$1$}
(2.5,-0.4)--(2.5,-3) node[below,scale=0.7,black]{$1$}     ;

\filldraw[black] (1.35,1.7) circle (1pt) (1.7,1.7) circle (1pt) (2.05,1.7) circle (1pt) 
(1.35,-1.7) circle (1pt) (1.7,-1.7) circle (1pt) (2.05,-1.7) circle (1pt);

\draw[darkred,thick] (-0.2,-0.4)rectangle(3.1,0.4)  
                     node[pos=0.5,scale=0.7,black]{$k+1$};
\end{tric}
}

\def\SumTwistb{
\begin{tric}
\draw [darkgreen]
 (4,3)--(4,-3) node[below,scale=0.7,black]{$1$}
 
(.4,0.4)--(.4,3)   node[above,scale=0.7,black]{$1$}
(0.9,0.4)--(0.9,3)  node[above,scale=0.7,black]{$1$}
(2.5,0.4)--(2.5,3)  node[above,scale=0.7,black]{$1$}

(.4,-0.4)--(.4,-3) node[below,scale=0.7,black]{$1$}
(0.9,-0.4)--(0.9,-3) node[below,scale=0.7,black]{$1$}
(2.5,-0.4)--(2.5,-3) node[below,scale=0.7,black]{$1$}     ;

\filldraw[black] (1.35,1.7) circle (1pt) (1.7,1.7) circle (1pt) (2.05,1.7) circle (1pt) 
(1.35,-1.7) circle (1pt) (1.7,-1.7) circle (1pt) (2.05,-1.7) circle (1pt);

\draw[darkred,thick] (-0.2,-0.4)rectangle(3.1,0.4)  
                     node[pos=0.5,scale=0.7,black]{$k$};
\end{tric}
}

\def\SumTwistc{
\begin{tric}
\draw [darkgreen]  

(.4,0.4)--(.4,2.2)   
(0.9,0.4)--(0.9,2.2)  
(2.5,0.4)--(2.5,2.2)  

(.4,-0.4)--(.4,-2)   node[below,scale=0.7,black]{$1$}
(0.9,-0.4)--(0.9,-2)  node[below,scale=0.7,black]{$1$}
(2.5,-0.4)--(2.5,-2)  node[below,scale=0.7,black]{$1$}
(3,-0.4)--(3,-2)  node[below,scale=0.7,black]{$1$}

(.4,3)--(.4,4.6)   node[above,scale=0.7,black]{$1$}
(0.9,3)--(0.9,4.6)  node[above,scale=0.7,black]{$1$}
(2.5,3)--(2.5,4.6)  node[above,scale=0.7,black]{$1$}
(3,3)--(3,4.6)  node[above,scale=0.7,black]{$1$};

\filldraw[black] (1.35,3.8) circle (1pt) (1.7,3.8) circle (1pt) (2.05,3.8) circle (1pt) 

(1.35,-1.2) circle (1pt) (1.7,-1.2) circle (1pt) (2.05,-1.2) circle (1pt)

(1.35,1.3) circle (1pt) (1.7,1.3) circle (1pt) (2.05,1.3) circle (1pt) ;

\draw [darkgreen] 
(3,0.4)..controls(3,1.3)and(4.5,1.3)..(4.5,4.6)
node[above,scale=0.7,black]{$1$}
      (3,2.2)..controls(3,1.3)and(4.5,1.3)..(4.5,-2)
      node[below,scale=0.7,black]{$1$};

\draw[darkred,thick] (-0.2,-0.4)rectangle(3.6,0.4)  
    node[pos=0.5,scale=0.7,black]{$k$}
                     
(-0.2,2.2)rectangle(3.6,3)  
    node[pos=0.5,scale=0.7,black]{$k$};
\end{tric}
}

\def\SumTwistd{
\begin{tricc}
\draw (2,7)--(2,-7)node[below,scale=0.7,black]{$1$};
\draw (0,0)--(0,1) (0,1)..controls(-1,1)and(-1,1.5)..(-1,2) 
       node[midway,left,scale=0.7,black]{$k-1$}
       (0,1)..controls(1,2)..(1,7)node[above,scale=0.7,black]{$1$}
       (-1,2)..controls(-2,2)and(-2,2.5)..(-2,3) 
        node[midway,left,scale=0.7,black]{$k-2$}
       (-1,2)..controls(0,3)..(0,7)node[above,scale=0.7,black]{$1$};
\filldraw[black] (-2.25,3.25) circle (1pt) 
                 (-2.5,3.5) circle (1pt) (-2.75,3.75) circle (1pt)
      (-0.9,6.5) circle (1pt) (-1.25,6.5) circle (1pt) (-1.6,6.5) circle (1pt); 
\draw 
(-3,4)..controls(-3.5,4)and(-3.5,4.5)..(-3.5,5)node[midway,left,scale=0.7,black]{$3$}
(-3.5,5)..controls(-4.5,5)and(-4.5,5.5)..(-4.5,6)node[midway,left,scale=0.7,black]{$2$}
(-4.5,6)..controls(-5.5,6)and(-5.5,6.5)..(-5.5,7)node[above,scale=0.7,black]{$1$}
(-3.5,5)..controls(-2.5,5.5)..(-2.5,7) node[above,scale=0.7,black]{$1$}
(-4.5,6)..controls(-3.5,6)and(-3.5,6.5)..(-3.5,7)node[above,scale=0.7,black]{$1$};
\draw (0,0)node[left,scale=0.7,black]{$k$};
\draw (0,0)--(0,-1) (0,-1)..controls(-1,-1)and(-1,-1.5)..(-1,-2) 
       node[midway,left,scale=0.7,black]{$k-1$}
       (0,-1)..controls(1,-2)..(1,-7)node[below,scale=0.7,black]{$1$}
       (-1,-2)..controls(-2,-2)and(-2,-2.5)..(-2,-3) 
        node[midway,left,scale=0.7,black]{$k-2$}
       (-1,-2)..controls(0,-3)..(0,-7)node[below,scale=0.7,black]{$1$};
\filldraw[black] (-2.25,-3.25) circle (1pt) 
                 (-2.5,-3.5) circle (1pt) (-2.75,-3.75) circle (1pt)
      (-0.9,-6.5) circle (1pt) (-1.25,-6.5) circle (1pt) (-1.6,-6.5) circle (1pt); 
\draw 
(-3,-4)..controls(-3.5,-4)and(-3.5,-4.5)..(-3.5,-5)node[midway,left,scale=0.7,black]{$3$}
(-3.5,-5)..controls(-4.5,-5)and(-4.5,-5.5)..(-4.5,-6)node[midway,left,scale=0.7,black]{$2$}
(-4.5,-6)..controls(-5.5,-6)and(-5.5,-6.5)..(-5.5,-7)node[below,scale=0.7,black]{$1$}
(-3.5,-5)..controls(-2.5,-5.5)..(-2.5,-7) node[below,scale=0.7,black]{$1$}
(-4.5,-6)..controls(-3.5,-6)and(-3.5,-6.5)..(-3.5,-7)node[below,scale=0.7,black]{$1$};
\end{tricc}
}

\def\SumTwiste{
\begin{tricc}
\draw (0,0)--(0,1) (0,1)..controls(-1,1)and(-1,1.5)..(-1,2) 
       node[midway,left,scale=0.7,black]{$k-1$}
       (0,1)..controls(1,2)..(1,7)node[above,scale=0.7,black]{$1$}
       (-1,2)..controls(-2,2)and(-2,2.5)..(-2,3) 
        node[midway,left,scale=0.7,black]{$k-2$}
       (-1,2)..controls(0,3)..(0,7)node[above,scale=0.7,black]{$1$};
\filldraw[black] (-2.25,3.25) circle (1pt) 
                 (-2.5,3.5) circle (1pt) (-2.75,3.75) circle (1pt)
      (-0.9,6.5) circle (1pt) (-1.25,6.5) circle (1pt) (-1.6,6.5) circle (1pt); 
\draw 
(-3,4)..controls(-3.5,4)and(-3.5,4.5)..(-3.5,5)node[midway,left,scale=0.7,black]{$3$}
(-3.5,5)..controls(-4.5,5)and(-4.5,5.5)..(-4.5,6)node[midway,left,scale=0.7,black]{$2$}
(-4.5,6)..controls(-5.5,6)and(-5.5,6.5)..(-5.5,7)node[above,scale=0.7,black]{$1$}
(-3.5,5)..controls(-2.5,5.5)..(-2.5,7) node[above,scale=0.7,black]{$1$}
(-4.5,6)..controls(-3.5,6)and(-3.5,6.5)..(-3.5,7)node[above,scale=0.7,black]{$1$};
\draw (0,0)node[left,scale=0.7,black]{$k$};
\draw (0,0)--(0,-1) (0,-1)..controls(-1,-1)and(-1,-1.5)..(-1,-2) 
       node[midway,left,scale=0.7,black]{$k-1$}
       (0,-1)..controls(1,-2)..(2,-21)node[below,scale=0.7,black]{$1$}
       (-1,-2)..controls(-2,-2)and(-2,-2.5)..(-2,-3) 
        node[midway,left,scale=0.7,black]{$k-2$}
       (-1,-2)..controls(0,-3)..(0,-7)node[left,scale=0.7,black]{$1$};
\filldraw[black] (-2.25,-3.25) circle (1pt) 
                 (-2.5,-3.5) circle (1pt) (-2.75,-3.75) circle (1pt)
      (-1,-7) circle (1pt) (-1.35,-7) circle (1pt) (-1.7,-7) circle (1pt); 
\draw 
(-3,-4)..controls(-3.5,-4)and(-3.5,-4.5)..(-3.5,-5)node[midway,left,scale=0.7,black]{$3$}
(-3.5,-5)..controls(-4.5,-5)and(-4.5,-5.5)..(-4.5,-6)node[midway,left,scale=0.7,black]{$2$}
(-4.5,-6)..controls(-5.5,-6)and(-5.5,-6.5)..(-5.5,-7)node[left,scale=0.7,black]{$1$}
(-3.5,-5)..controls(-2.5,-5.5)..(-2.5,-7) node[left,scale=0.7,black]{$1$}
(-4.5,-6)..controls(-3.5,-6)and(-3.5,-6.5)..(-3.5,-7)node[left,scale=0.7,black]{$1$};
\draw (0,-14)--(0,-13) (0,-13)..controls(-1,-13)and(-1,1.5-14)..(-1,2-14) 
       node[midway,left,scale=0.7,black]{$k-1$}
       (0,1-14)..controls(1,2-14)..(2,7)node[above,scale=0.7,black]{$1$}
       (-1,2-14)..controls(-2,2-14)and(-2,2.5-14)..(-2,3-14) 
        node[midway,left,scale=0.7,black]{$k-2$}
       (-1,2-14)..controls(0,3-14)..(0,7-14);
\filldraw[black] (-2.25,3.25-14) circle (1pt) 
                 (-2.5,3.5-14) circle (1pt) (-2.75,3.75-14) circle (1pt); 
\draw 
(-3,4-14)..controls(-3.5,4-14)and(-3.5,4.5-14)..(-3.5,5-14)node[midway,left,scale=0.7,black]{$3$}
(-3.5,5-14)..controls(-4.5,5-14)and(-4.5,5.5-14)..(-4.5,6-14)node[midway,left,scale=0.7,black]{$2$}
(-4.5,6-14)..controls(-5.5,6-14)and(-5.5,6.5-14)..(-5.5,7-14)
(-3.5,5-14)..controls(-2.5,5.5-14)..(-2.5,7-14)
(-4.5,6-14)..controls(-3.5,6-14)and(-3.5,6.5-14)..(-3.5,7-14);
\draw (0,0-14)node[left,scale=0.7,black]{$k$};
\draw (0,0-14)--(0,-1-14) (0,-1-14)..controls(-1,-1-14)and(-1,-1.5-14)..(-1,-2-14) 
       node[midway,left,scale=0.7,black]{$k-1$}
       (0,-1-14)..controls(1,-2-14)..(1,-7-14)node[below,scale=0.7,black]{$1$}
       (-1,-2-14)..controls(-2,-2-14)and(-2,-2.5-14)..(-2,-3-14) 
        node[midway,left,scale=0.7,black]{$k-2$}
       (-1,-2-14)..controls(0,-3-14)..(0,-7-14)node[below,scale=0.7,black]{$1$};
\filldraw[black] (-2.25,-3.25-14) circle (1pt) 
                 (-2.5,-3.5-14) circle (1pt) (-2.75,-3.75-14) circle (1pt)
      (-0.9,-6.5-14) circle (1pt) (-1.25,-6.5-14) circle (1pt) (-1.6,-6.5-14) circle (1pt); 
\draw 
(-3,-4-14)..controls(-3.5,-4-14)and(-3.5,-4.5-14)..(-3.5,-5-14)node[midway,left,scale=0.7,black]{$3$}
(-3.5,-5-14)..controls(-4.5,-5-14)and(-4.5,-5.5-14)..(-4.5,-6-14)node[midway,left,scale=0.7,black]{$2$}
(-4.5,-6-14)..controls(-5.5,-6-14)and(-5.5,-6.5-14)..(-5.5,-7-14)node[below,scale=0.7,black]{$1$}
(-3.5,-5-14)..controls(-2.5,-5.5-14)..(-2.5,-7-14) node[below,scale=0.7,black]{$1$}
(-4.5,-6-14)..controls(-3.5,-6-14)and(-3.5,-6.5-14)..(-3.5,-7-14)node[below,scale=0.7,black]{$1$};
\end{tricc}
}

\def\SumTwistf{
\begin{tricc}
\draw (2,7)--(2,-7);
\draw (0,0)--(0,1) (0,1)..controls(-1,1)and(-1,1.5)..(-1,2) 
       node[midway,left,scale=0.7,black]{$k-1$}
       (0,1)..controls(1,2)..(1,7)node[above,scale=0.7,black]{$1$}
       (-1,2)..controls(-2,2)and(-2,2.5)..(-2,3) 
        node[midway,left,scale=0.7,black]{$k-2$}
       (-1,2)..controls(0,3)..(0,7)node[above,scale=0.7,black]{$1$};
\filldraw[black] (-2.25,3.25) circle (1pt) 
                 (-2.5,3.5) circle (1pt) (-2.75,3.75) circle (1pt)
      (-0.9,6.5) circle (1pt) (-1.25,6.5) circle (1pt) (-1.6,6.5) circle (1pt); 
\draw 
(-3,4)..controls(-3.5,4)and(-3.5,4.5)..(-3.5,5)node[midway,left,scale=0.7,black]{$3$}
(-3.5,5)..controls(-4.5,5)and(-4.5,5.5)..(-4.5,6)node[midway,left,scale=0.7,black]{$2$}
(-4.5,6)..controls(-5.5,6)and(-5.5,6.5)..(-5.5,7)node[above,scale=0.7,black]{$1$}
(-3.5,5)..controls(-2.5,5.5)..(-2.5,7) node[above,scale=0.7,black]{$1$}
(-4.5,6)..controls(-3.5,6)and(-3.5,6.5)..(-3.5,7)node[above,scale=0.7,black]{$1$};
\draw (0,0)node[left,scale=0.7,black]{$k$};
\draw (0,0)--(0,-1) (0,-1)..controls(-1,-1)and(-1,-1.5)..(-1,-2) 
       node[midway,left,scale=0.7,black]{$k-1$}
       (0,-1)..controls(1,-2)..(1,-7)node[below,scale=0.7,black]{$1$}
       (-1,-2)..controls(-2,-2)and(-2,-2.5)..(-2,-3) 
        node[midway,left,scale=0.7,black]{$k-2$}
       (-1,-2)..controls(0,-3)..(0,-7)node[below,scale=0.7,black]{$1$};
\filldraw[black] (-2.25,-3.25) circle (1pt) 
                 (-2.5,-3.5) circle (1pt) (-2.75,-3.75) circle (1pt)
      (-0.9,-6.5) circle (1pt) (-1.25,-6.5) circle (1pt) (-1.6,-6.5) circle (1pt); 
\draw 
(-3,-4)..controls(-3.5,-4)and(-3.5,-4.5)..(-3.5,-5)node[midway,left,scale=0.7,black]{$3$}
(-3.5,-5)..controls(-4.5,-5)and(-4.5,-5.5)..(-4.5,-6)node[midway,left,scale=0.7,black]{$2$}
(-4.5,-6)..controls(-5.5,-6)and(-5.5,-6.5)..(-5.5,-7)node[below,scale=0.7,black]{$1$}
(-3.5,-5)..controls(-2.5,-5.5)..(-2.5,-7) node[below,scale=0.7,black]{$1$}
(-4.5,-6)..controls(-3.5,-6)and(-3.5,-6.5)..(-3.5,-7)node[below,scale=0.7,black]{$1$};
\end{tricc}
}

\def\SumTwistg{
\begin{tricc}
\draw (0,0)--(0,1) (0,1)..controls(-1,1)and(-1,1.5)..(-1,2) 
       node[midway,left,scale=0.7,black]{$k-1$}
       (0,1)..controls(1,2)..(1,7)node[above,scale=0.7,black]{$1$}
       (-1,2)..controls(-2,2)and(-2,2.5)..(-2,3) 
        node[midway,left,scale=0.7,black]{$k-2$}
       (-1,2)..controls(0,3)..(0,7)node[above,scale=0.7,black]{$1$};
\filldraw[black] (-2.25,3.25) circle (1pt) 
                 (-2.5,3.5) circle (1pt) (-2.75,3.75) circle (1pt)
      (-0.9,6.5) circle (1pt) (-1.25,6.5) circle (1pt) (-1.6,6.5) circle (1pt); 
\draw 
(-3,4)..controls(-3.5,4)and(-3.5,4.5)..(-3.5,5)node[midway,left,scale=0.7,black]{$3$}
(-3.5,5)..controls(-4.5,5)and(-4.5,5.5)..(-4.5,6)node[midway,left,scale=0.7,black]{$2$}
(-4.5,6)..controls(-5,6)and(-5,6.5)..(-5,7)node[above,scale=0.7,black]{$1$}
(-3.5,5)..controls(-2.5,5.5)..(-2.5,7) node[above,scale=0.7,black]{$1$}
(-4.5,6)..controls(-3.5,6)and(-3.5,6.5)..(-3.5,7)node[above,scale=0.7,black]{$1$};
\draw (0,0)node[left,scale=0.7,black]{$k$};
\draw (0,0)--(0,-1) (0,-1)..controls(-1,-1)and(-1,-1.5)..(-1,-2) 
       node[midway,left,scale=0.7,black]{$k-1$}
       (0,-1)..controls(1,-2)..(2,-21)node[below,scale=0.7,black]{$1$}
       (-1,-2)..controls(-2,-2)and(-2,-2.5)..(-2,-3) 
        node[midway,left,scale=0.7,black]{$k-2$}
       (-1,-2)..controls(0,-3)..(0,-7)node[left,scale=0.7,black]{$1$};
\filldraw[black] (-2.25,-3.25) circle (1pt) 
                 (-2.5,-3.5) circle (1pt) (-2.75,-3.75) circle (1pt)
      (-1,-7) circle (1pt) (-1.35,-7) circle (1pt) (-1.7,-7) circle (1pt); 
\draw 
(-3,-4)..controls(-3.5,-4)and(-3.5,-4.5)..(-3.5,-5)node[midway,left,scale=0.7,black]{$3$}
(-3.5,-5)..controls(-4.5,-5)and(-4.5,-5.5)..(-4.5,-7)node[right,scale=0.7,black]{$2$}
(-3.5,-5)..controls(-2.5,-5.5)..(-2.5,-7) node[left,scale=0.7,black]{$1$};
\draw (0,-14)--(0,-13) (0,-13)..controls(-1,-13)and(-1,1.5-14)..(-1,2-14) 
       node[midway,left,scale=0.7,black]{$k-1$}
       (0,1-14)..controls(1,2-14)..(2,7)node[above,scale=0.7,black]{$1$}
       (-1,2-14)..controls(-2,2-14)and(-2,2.5-14)..(-2,3-14) 
        node[midway,left,scale=0.7,black]{$k-2$}
       (-1,2-14)..controls(0,3-14)..(0,7-14);
\filldraw[black] (-2.25,3.25-14) circle (1pt) 
                 (-2.5,3.5-14) circle (1pt) (-2.75,3.75-14) circle (1pt); 
\draw 
(-3,4-14)..controls(-3.5,4-14)and(-3.5,4.5-14)..(-3.5,5-14)node[midway,left,scale=0.7,black]{$3$}
(-3.5,5-14)..controls(-4.5,5-14)and(-4.5,5.5-14)..(-4.5,7-14)
(-3.5,5-14)..controls(-2.5,5.5-14)..(-2.5,7-14);
\draw (0,0-14)node[left,scale=0.7,black]{$k$};
\draw (0,0-14)--(0,-1-14) (0,-1-14)..controls(-1,-1-14)and(-1,-1.5-14)..(-1,-2-14) 
       node[midway,left,scale=0.7,black]{$k-1$}
       (0,-1-14)..controls(1,-2-14)..(1,-7-14)node[below,scale=0.7,black]{$1$}
       (-1,-2-14)..controls(-2,-2-14)and(-2,-2.5-14)..(-2,-3-14) 
        node[midway,left,scale=0.7,black]{$k-2$}
       (-1,-2-14)..controls(0,-3-14)..(0,-7-14)node[below,scale=0.7,black]{$1$};
\filldraw[black] (-2.25,-3.25-14) circle (1pt) 
                 (-2.5,-3.5-14) circle (1pt) (-2.75,-3.75-14) circle (1pt)
      (-0.9,-6.5-14) circle (1pt) (-1.25,-6.5-14) circle (1pt) (-1.6,-6.5-14) circle (1pt); 
\draw 
(-3,-4-14)..controls(-3.5,-4-14)and(-3.5,-4.5-14)..(-3.5,-5-14)node[midway,left,scale=0.7,black]{$3$}
(-3.5,-5-14)..controls(-4.5,-5-14)and(-4.5,-5.5-14)..(-4.5,-6-14)node[midway,left,scale=0.7,black]{$2$}
(-4.5,-6-14)..controls(-5,-6-14)and(-5,-6.5-14)..(-5,-7-14)node[below,scale=0.7,black]{$1$}
(-3.5,-5-14)..controls(-2.5,-5.5-14)..(-2.5,-7-14) node[below,scale=0.7,black]{$1$}
(-4.5,-6-14)..controls(-3.5,-6-14)and(-3.5,-6.5-14)..(-3.5,-7-14)node[below,scale=0.7,black]{$1$};
\end{tricc}
}

\def\SumTwisth{
\begin{tricc}
\draw (2,7)--(2,-7);
\draw (0,0)--(0,1) (0,1)..controls(-1,1)and(-1,1.5)..(-1,2) 
       node[midway,left,scale=0.7,black]{$k-1$}
       (0,1)..controls(1,2)..(1,7)node[above,scale=0.7,black]{$1$}
       (-1,2)..controls(-2,2)and(-2,2.5)..(-2,3) 
        node[midway,left,scale=0.7,black]{$k-2$}
       (-1,2)..controls(0,3)..(0,7)node[above,scale=0.7,black]{$1$};
\filldraw[black] (-2.25,3.25) circle (1pt) 
                 (-2.5,3.5) circle (1pt) (-2.75,3.75) circle (1pt)
      (-0.9,6.5) circle (1pt) (-1.25,6.5) circle (1pt) (-1.6,6.5) circle (1pt); 
\draw 
(-3,4)..controls(-3.5,4)and(-3.5,4.5)..(-3.5,5)node[midway,left,scale=0.7,black]{$3$}
(-3.5,5)..controls(-4.5,5)and(-4.5,5.5)..(-4.5,6)node[midway,left,scale=0.7,black]{$2$}
(-4.5,6)..controls(-5.5,6)and(-5.5,6.5)..(-5.5,7)node[above,scale=0.7,black]{$1$}
(-3.5,5)..controls(-2.5,5.5)..(-2.5,7) node[above,scale=0.7,black]{$1$}
(-4.5,6)..controls(-3.5,6)and(-3.5,6.5)..(-3.5,7)node[above,scale=0.7,black]{$1$};
\draw (0,0)node[left,scale=0.7,black]{$k$};
\draw (0,0)--(0,-1) (0,-1)..controls(-1,-1)and(-1,-1.5)..(-1,-2) 
       node[midway,left,scale=0.7,black]{$k-1$}
       (0,-1)..controls(1,-2)..(1,-7)node[below,scale=0.7,black]{$1$}
       (-1,-2)..controls(-2,-2)and(-2,-2.5)..(-2,-3) 
        node[midway,left,scale=0.7,black]{$k-2$}
       (-1,-2)..controls(0,-3)..(0,-7)node[below,scale=0.7,black]{$1$};
\filldraw[black] (-2.25,-3.25) circle (1pt) 
                 (-2.5,-3.5) circle (1pt) (-2.75,-3.75) circle (1pt)
      (-0.9,-6.5) circle (1pt) (-1.25,-6.5) circle (1pt) (-1.6,-6.5) circle (1pt); 
\draw 
(-3,-4)..controls(-3.5,-4)and(-3.5,-4.5)..(-3.5,-5)node[midway,left,scale=0.7,black]{$3$}
(-3.5,-5)..controls(-4.5,-5)and(-4.5,-5.5)..(-4.5,-6)node[midway,left,scale=0.7,black]{$2$}
(-4.5,-6)..controls(-5.5,-6)and(-5.5,-6.5)..(-5.5,-7)node[below,scale=0.7,black]{$1$}
(-3.5,-5)..controls(-2.5,-5.5)..(-2.5,-7) node[below,scale=0.7,black]{$1$}
(-4.5,-6)..controls(-3.5,-6)and(-3.5,-6.5)..(-3.5,-7)node[below,scale=0.7,black]{$1$};
\end{tricc}
}

\def\SumTwisti{
\begin{tricc}
\draw (0,0)--(0,1) (0,1)..controls(-1,1)and(-1,1.5)..(-1,2) 
       node[midway,left,scale=0.7,black]{$k-1$}
       (0,1)..controls(1,2)..(1,7)node[above,scale=0.7,black]{$1$}
       (-1,2)..controls(-2,2)and(-2,2.5)..(-2,3) 
        node[midway,left,scale=0.7,black]{$k-2$}
       (-1,2)..controls(0,3)..(0,7)node[above,scale=0.7,black]{$1$};
\filldraw[black] (-2.25,3.25) circle (1pt) 
                 (-2.5,3.5) circle (1pt) (-2.75,3.75) circle (1pt)
      (-0.9,6.5) circle (1pt) (-1.25,6.5) circle (1pt) (-1.6,6.5) circle (1pt); 
\draw 
(-3,4)..controls(-3.5,4)and(-3.5,4.5)..(-3.5,5)node[midway,left,scale=0.7,black]{$3$}
(-3.5,5)..controls(-4.5,5)and(-4.5,5.5)..(-4.5,6)node[midway,left,scale=0.7,black]{$2$}
(-4.5,6)..controls(-5,6)and(-5,6.5)..(-5,7)node[above,scale=0.7,black]{$1$}
(-3.5,5)..controls(-2.5,5.5)..(-2.5,7) node[above,scale=0.7,black]{$1$}
(-4.5,6)..controls(-3.5,6)and(-3.5,6.5)..(-3.5,7)node[above,scale=0.7,black]{$1$};
\draw (0,0)node[left,scale=0.7,black]{$k$};
\draw (0,0)--(0,-1) (0,-1)..controls(-1,-1)and(-1,-1.5)..(-1,-2) 
       node[midway,left,scale=0.7,black]{$k-1$}
       (0,-1)..controls(1,-2)..(2,-17)node[below,scale=0.7,black]{$1$}
       (-1,-2)..controls(-2,-2)and(-2,-2.5)..(-2,-3) 
        node[midway,left,scale=0.7,black]{$k-2$}
       (-1,-2)..controls(0,-3)..(0,-5)node[left,scale=0.7,black]{$1$};
\filldraw[black] (-2.25,-3.25) circle (1pt) 
                 (-2.5,-3.5) circle (1pt) (-2.75,-3.75) circle (1pt)
      (-1.35,-5) circle (1pt) (-1.7,-5) circle (1pt) (-2.05,-5) circle (1pt); 
\draw 
(-3,-4)..controls(-3.5,-4)and(-3.5,-4.5)..(-3.5,-5)node[left,scale=0.7,black]{$3$} ;
\draw (0,-11)--(0,-9) (0,-9)..controls(-1,-9)and(-1,1.5-10)..(-1,2-10) 
       node[midway,left,scale=0.7,black]{$k-1$}
       (0,1-10)..controls(1,2-10)..(2,7)node[above,scale=0.7,black]{$1$}
       (-1,2-10)..controls(-2,2-10)and(-2,2.5-10)..(-2,3-10) 
        node[midway,left,scale=0.7,black]{$k-2$}
       (-1,2-10)..controls(0,3-10)..(0,-5);
\filldraw[black] (-2.25,3.25-10) circle (1pt) 
                 (-2.5,3.5-10) circle (1pt) (-2.75,3.75-10) circle (1pt); 
\draw 
(-3,4-10)..controls(-3.5,4-10)and(-3.5,4.5-10)..(-3.5,-5);
\draw (0,0-10)node[left,scale=0.7,black]{$k$};
\draw (0,0-10)--(0,-1-10) (0,-1-10)..controls(-1,-1-10)and(-1,-1.5-10)..(-1,-2-10) 
       node[midway,left,scale=0.7,black]{$k-1$}
       (0,-1-10)..controls(1,-2-10)..(1,-7-10)node[below,scale=0.7,black]{$1$}
       (-1,-2-10)..controls(-2,-2-10)and(-2,-2.5-10)..(-2,-3-10) 
        node[midway,left,scale=0.7,black]{$k-2$}
       (-1,-2-10)..controls(0,-3-10)..(0,-7-10)node[below,scale=0.7,black]{$1$};
\filldraw[black] (-2.25,-3.25-10) circle (1pt) 
                 (-2.5,-3.5-10) circle (1pt) (-2.75,-3.75-10) circle (1pt)
      (-0.9,-6.5-10) circle (1pt) (-1.25,-6.5-10) circle (1pt) (-1.6,-6.5-10) circle (1pt); 
\draw 
(-3,-4-10)..controls(-3.5,-4-10)and(-3.5,-4.5-10)..(-3.5,-5-10)node[midway,left,scale=0.7,black]{$3$}
(-3.5,-5-10)..controls(-4.5,-5-10)and(-4.5,-5.5-10)..(-4.5,-6-10)node[midway,left,scale=0.7,black]{$2$}
(-4.5,-6-10)..controls(-5,-6-10)and(-5,-6.5-10)..(-5,-7-10)node[below,scale=0.7,black]{$1$}
(-3.5,-5-10)..controls(-2.5,-5.5-10)..(-2.5,-7-10) node[below,scale=0.7,black]{$1$}
(-4.5,-6-10)..controls(-3.5,-6-10)and(-3.5,-6.5-10)..(-3.5,-7-10)node[below,scale=0.7,black]{$1$};
\end{tricc}
}

\def\SumTwistj{
\begin{tricc}
\draw (2,7)--(2,-7)node[below,scale=0.7,black]{$1$};
\draw (0,0)--(0,1) (0,1)..controls(-1,1)and(-1,1.5)..(-1,2) 
       node[midway,left,scale=0.7,black]{$k-1$}
       (0,1)..controls(1,2)..(1,7)node[above,scale=0.7,black]{$1$}
       (-1,2)..controls(-2,2)and(-2,2.5)..(-2,3) 
        node[midway,left,scale=0.7,black]{$k-2$}
       (-1,2)..controls(0,3)..(0,7)node[above,scale=0.7,black]{$1$};
\filldraw[black] (-2.25,3.25) circle (1pt) 
                 (-2.5,3.5) circle (1pt) (-2.75,3.75) circle (1pt)
      (-0.9,6.5) circle (1pt) (-1.25,6.5) circle (1pt) (-1.6,6.5) circle (1pt); 
\draw 
(-3,4)..controls(-3.5,4)and(-3.5,4.5)..(-3.5,5)node[midway,left,scale=0.7,black]{$3$}
(-3.5,5)..controls(-4.5,5)and(-4.5,5.5)..(-4.5,6)node[midway,left,scale=0.7,black]{$2$}
(-4.5,6)..controls(-5.5,6)and(-5.5,6.5)..(-5.5,7)node[above,scale=0.7,black]{$1$}
(-3.5,5)..controls(-2.5,5.5)..(-2.5,7) node[above,scale=0.7,black]{$1$}
(-4.5,6)..controls(-3.5,6)and(-3.5,6.5)..(-3.5,7)node[above,scale=0.7,black]{$1$};
\draw (0,0)node[left,scale=0.7,black]{$k$};
\draw (0,0)--(0,-1) (0,-1)..controls(-1,-1)and(-1,-1.5)..(-1,-2) 
       node[midway,left,scale=0.7,black]{$k-1$}
       (0,-1)..controls(1,-2)..(1,-7)node[below,scale=0.7,black]{$1$}
       (-1,-2)..controls(-2,-2)and(-2,-2.5)..(-2,-3) 
        node[midway,left,scale=0.7,black]{$k-2$}
       (-1,-2)..controls(0,-3)..(0,-7)node[below,scale=0.7,black]{$1$};
\filldraw[black] (-2.25,-3.25) circle (1pt) 
                 (-2.5,-3.5) circle (1pt) (-2.75,-3.75) circle (1pt)
      (-0.9,-6.5) circle (1pt) (-1.25,-6.5) circle (1pt) (-1.6,-6.5) circle (1pt); 
\draw 
(-3,-4)..controls(-3.5,-4)and(-3.5,-4.5)..(-3.5,-5)node[midway,left,scale=0.7,black]{$3$}
(-3.5,-5)..controls(-4.5,-5)and(-4.5,-5.5)..(-4.5,-6)node[midway,left,scale=0.7,black]{$2$}
(-4.5,-6)..controls(-5.5,-6)and(-5.5,-6.5)..(-5.5,-7)node[below,scale=0.7,black]{$1$}
(-3.5,-5)..controls(-2.5,-5.5)..(-2.5,-7) node[below,scale=0.7,black]{$1$}
(-4.5,-6)..controls(-3.5,-6)and(-3.5,-6.5)..(-3.5,-7)node[below,scale=0.7,black]{$1$};
\end{tricc}
}

\def\SumTwistk{
\begin{tricc}
\draw (0,0)--(0,1) (0,1)..controls(-1,1)and(-1,1.5)..(-1,2) 
       node[midway,left,scale=0.7,black]{$k-1$}
       (0,1)..controls(1,2)..(1,7)node[above,scale=0.7,black]{$1$}
       (-1,2)..controls(-2,2)and(-2,2.5)..(-2,3) 
        node[midway,left,scale=0.7,black]{$k-2$}
       (-1,2)..controls(0,3)..(0,7)node[above,scale=0.7,black]{$1$};
\filldraw[black] (-2.25,3.25) circle (1pt) 
                 (-2.5,3.5) circle (1pt) (-2.75,3.75) circle (1pt)
      (-0.9,6.5) circle (1pt) (-1.25,6.5) circle (1pt) (-1.6,6.5) circle (1pt); 
\draw 
(-3,4)..controls(-3.5,4)and(-3.5,4.5)..(-3.5,5)node[midway,left,scale=0.7,black]{$3$}
(-3.5,5)..controls(-4.5,5)and(-4.5,5.5)..(-4.5,6)node[midway,left,scale=0.7,black]{$2$}
(-4.5,6)..controls(-5,6)and(-5,6.5)..(-5,7)node[above,scale=0.7,black]{$1$}
(-3.5,5)..controls(-2.5,5.5)..(-2.5,7) node[above,scale=0.7,black]{$1$}
(-4.5,6)..controls(-3.5,6)and(-3.5,6.5)..(-3.5,7)node[above,scale=0.7,black]{$1$};
\draw (0,0)node[left,scale=0.7,black]{$k$};
\draw (0,0)--(0,-1) (0,-1)..controls(-1,-1)and(-1,-2)..(-1,-3) 
       node[left,scale=0.7,black]{$k-1$}
       (0,-1)..controls(1,-2)..(2,-13)node[below,scale=0.7,black]{$1$};
\draw   (0,-6)--(0,-5) (0,-5)..controls(-1,-5)and(-1,-4)..(-1,-3) 
       (0,1-6)..controls(1,2-6)..(2,7)node[above,scale=0.7,black]{$1$};
\draw (0,0-6)node[left,scale=0.7,black]{$k$};
\draw (0,0-6)--(0,-1-6) (0,-1-6)..controls(-1,-1-6)and(-1,-1.5-6)..(-1,-2-6) 
       node[midway,left,scale=0.7,black]{$k-1$}
       (0,-1-6)..controls(1,-2-6)..(1,-7-6)node[below,scale=0.7,black]{$1$}
       (-1,-2-6)..controls(-2,-2-6)and(-2,-2.5-6)..(-2,-3-6) 
        node[midway,left,scale=0.7,black]{$k-2$}
       (-1,-2-6)..controls(0,-3-6)..(0,-7-6)node[below,scale=0.7,black]{$1$};
\filldraw[black] (-2.25,-3.25-6) circle (1pt) 
                 (-2.5,-3.5-6) circle (1pt) (-2.75,-3.75-6) circle (1pt)
      (-0.9,-6.5-6) circle (1pt) (-1.25,-6.5-6) circle (1pt) (-1.6,-6.5-6) circle (1pt); 
\draw 
(-3,-4-6)..controls(-3.5,-4-6)and(-3.5,-4.5-6)..(-3.5,-5-6)node[midway,left,scale=0.7,black]{$3$}
(-3.5,-5-6)..controls(-4.5,-5-6)and(-4.5,-5.5-6)..(-4.5,-6-6)node[midway,left,scale=0.7,black]{$2$}
(-4.5,-6-6)..controls(-5,-6-6)and(-5,-6.5-6)..(-5,-7-6)node[below,scale=0.7,black]{$1$}
(-3.5,-5-6)..controls(-2.5,-5.5-6)..(-2.5,-7-6) node[below,scale=0.7,black]{$1$}
(-4.5,-6-6)..controls(-3.5,-6-6)and(-3.5,-6.5-6)..(-3.5,-7-6)node[below,scale=0.7,black]{$1$};
\end{tricc}
}

\def\SumTwistl{
\begin{tricc}
\draw  (0,1)..controls(-1,1)and(-1,1.5)..(-1,2) 
       node[midway,left,scale=0.7,black]{$k-1$}
       (0,1)..controls(1,2)..(1,7)node[above,scale=0.7,black]{$1$}
       (-1,2)..controls(-2,2)and(-2,2.5)..(-2,3) 
        node[midway,left,scale=0.7,black]{$k-2$}
       (-1,2)..controls(0,3)..(0,7)node[above,scale=0.7,black]{$1$};

\draw   (0,-1)--(0,1) node[midway,left,scale=0.7,black]{$k$}
       (0,-5)--(0,-3) node[midway,left,scale=0.7,black]{$k$}
        (0,-1)..controls(2,0)..(2,7)node[above,scale=0.7,black]{$1$}
       (0,-3)..controls(2,-4)..(2,-11)node[below,scale=0.7,black]{$1$}
       (0,-1)--(0,-3) node[midway,left,scale=0.7,black]{$k+1$};
       
\filldraw[black] (-2.25,3.25) circle (1pt) 
                 (-2.5,3.5) circle (1pt) (-2.75,3.75) circle (1pt)
      (-0.9,6.5) circle (1pt) (-1.25,6.5) circle (1pt) (-1.6,6.5) circle (1pt); 
\draw 
(-3,4)..controls(-3.5,4)and(-3.5,4.5)..(-3.5,5)node[midway,left,scale=0.7,black]{$3$}
(-3.5,5)..controls(-4.5,5)and(-4.5,5.5)..(-4.5,6)node[midway,left,scale=0.7,black]{$2$}
(-4.5,6)..controls(-5,6)and(-5,6.5)..(-5,7)node[above,scale=0.7,black]{$1$}
(-3.5,5)..controls(-2.5,5.5)..(-2.5,7) node[above,scale=0.7,black]{$1$}
(-4.5,6)..controls(-3.5,6)and(-3.5,6.5)..(-3.5,7)node[above,scale=0.7,black]{$1$};
\draw  (0,-1-4)..controls(-1,-1-4)and(-1,-1.5-4)..(-1,-2-4) 
       node[midway,left,scale=0.7,black]{$k-1$}
       (0,-1-4)..controls(1,-2-4)..(1,-7-4) node[below,scale=0.7,black]{$1$}
    (-1,-2-4)..controls(-2,-2-4)and(-2,-2.5-4)..(-2,-3-4) 
        node[midway,left,scale=0.7,black]{$k-2$}
       (-1,-2-4)..controls(0,-3-4)..(0,-7-4) node[below,scale=0.7,black]{$1$};
\filldraw[black] (-2.25,-3.25-4) circle (1pt) 
                 (-2.5,-3.5-4) circle (1pt) (-2.75,-3.75-4) circle (1pt)
      (-0.9,-6.5-4) circle (1pt) (-1.25,-6.5-4) circle (1pt) (-1.6,-6.5-4) circle (1pt); 
\draw 
(-3,-4-4)..controls(-3.5,-4-4)and(-3.5,-4.5-4)..(-3.5,-5-4)node[midway,left,scale=0.7,black]{$3$}
(-3.5,-5-4)..controls(-4.5,-5-4)and(-4.5,-5.5-4)..(-4.5,-6-4)node[midway,left,scale=0.7,black]{$2$}
(-4.5,-6-4)..controls(-5,-6-4)and(-5,-6.5-4)..(-5,-7-4)node[below,scale=0.7,black]{$1$}
(-3.5,-5-4)..controls(-2.5,-5.5-4)..(-2.5,-7-4) node[below,scale=0.7,black]{$1$}
(-4.5,-6-4)..controls(-3.5,-6-4)and(-3.5,-6.5-4)..(-3.5,-7-4)node[below,scale=0.7,black]{$1$};
\end{tricc}
}

\begin{thm}\label{thm:LZ-kernel}
The kernel of the functor 
\[
F:\mathcal{B}(m)\longrightarrow \StdRepC
\]
is the monoidal ideal generated by $a_{m+1}$.
\end{thm}
\begin{proof}
This is \cite[Theorem 4.8(ii)]{LZbrauercat}.
\end{proof}

\begin{lemma}\label{L:kernel-classical-stdwebfunctor}
The kernel of the functor $\Phi_{\kk}|_{\StdWebC}$ is the monoidal ideal of the category $\StdWebC$ generated (as a monoidal ideal) by $\eta_{\kk}(a_{m+1})$. 
\end{lemma}
\begin{proof}
Since $\Phi_{\kk}(\eta_{\kk}(a_{m+1})) = F(a_{m+1}) = 0$, we have $\langle \eta_{\kk}(a_{m+1})\rangle \subset \ker \Phi_{\kk}$. To show the reverse inclusion, let $f\in \ker\Phi_{\kk}$. Since $\BMWC=\mathcal{B}(m)$ it follows from Proposition \ref{P:functor-bmw-to-web} that $\eta_{\kk}$ is full. Thus, there is some $\tilde{f}$ in $\mathcal{B}(m)$ such that $\eta_{\kk}(\tilde{f}) = f$. Thus, $F(\tilde{f}) = \Phi_{\kk}\circ\eta_{\kk}(\tilde{f}) = \Phi_{\kk}(f) = 0$, so $\tilde{f}$ is in the kernel of $F$. Theorem \ref{thm:LZ-kernel} implies $\tilde{f}\in \langle a_{m+1}\rangle$, and by Lemma \ref{L:monoidalfunctor/ideal-compatibility} we have $f = \eta_{\kk}(\tilde{f})\in \langle \eta_{\kk}(a_{m+1})\rangle$.
\end{proof}

We will show that $\Phi_{\kk}|_{\StdWebC}$ is faithful. Since this is equivalent to the kernel of $\Phi_{\kk}$ being $\langle 0\rangle$, we want to argue that $\eta_{\kk}(a_{m+1})$ is already equal to zero in $\StdWebC$. This is implied by a diagrammatic calculation in $\WebC$, which relies on the following Lemma.

\begin{lemma}
The following equation holds in $\StdWebC$. 
\begin{equation}\label{E:anti-symmetrizer-double-coset-recursion}
\SumTwista  \  \ = \ \  \SumTwistb \ \  -\frac{1}{(k-1)!} \ \ \SumTwistc
\end{equation}
\end{lemma}
\begin{proof}
Apply $\eta_{\kk}$ to the analogous equation in $\mathcal{B}(m)$, which holds by \cite[Lemma 2.11(1)]{LZbrauercat}.
\end{proof}

Next, we perform the calculation in $\WebC$ required to deduce $\eta_{\kk}(a_{m+1})=0$.

\begin{proposition}\label{P:verification-of-classicalkernel}
Let $k\in \mathbb{Z}_{\ge 0}$. The following equality holds in $\StdWebC$.

  \begin{align}
         \SumTwist = \BigIk \label{claspVSsym}
  \end{align}
\end{proposition}

\begin{proof}
We use proof by induction. The base case, $k=1$, is trivial. Suppose \eqref{claspVSsym} holds for $k \in \mathbb{Z}_{\ge 1}$. The following graphical calculation proves that Equation \eqref{claspVSsym} holds for $k+1$.
\begin{align*}
\SumTwista \stackrel{\eqref{E:anti-symmetrizer-double-coset-recursion}}{=} \ \  \SumTwistb \ \  -\frac{1}{(k-1)!} \ \ \SumTwistc \\
    \stackrel{\eqref{claspVSsym}}{=}   \SumTwistd -\frac{1}{(k-1)!}  \SumTwiste \\
\end{align*}

\begin{align*}    
\stackrel{(\ref{defskein}c)}{=} . . . \stackrel{(\ref{defskein}c)}{=}\SumTwistj -\frac{(k-1)!}{(k-1)!}  \SumTwistk \stackrel{\eqref{EQ:classicalcrossing}}{=}  \SumTwistl
\end{align*}
\end{proof}

\begin{corollary}\label{C:PhiC-std-is-faithful}
The functor 
\[
\Phi_{\kk}|_{\StdWebC}:\StdWebC\longrightarrow \StdRepC
\]
is faithful.
\end{corollary}
\begin{proof}
Since strands labelled by $m+1$ are equal to zero in $\StdWebC$, Proposition \ref{P:verification-of-classicalkernel} implies that $\eta_{\kk}(a_{m+1})=0$. We then deduce from Lemma \ref{L:kernel-classical-stdwebfunctor} that $\ker \Phi_{\kk} = \langle 0 \rangle$, i.e. $\Phi_{\kk}$ is faithful.
\end{proof}

Before we prove that $\Phi_{\K}|_{\StdWebK}$ is faithful, we state two technical lemmas.

\begin{lemma}\label{L:technical-base-change-lemma}
Let $W, F$ be $\A$-modules and assume that $W$ is finitely generated over $\A$. Let $f:W\longrightarrow F$ be an $\A$-module homomorphism. Suppose that $\R \otimes f:\R \otimes W\longrightarrow \R \otimes F$ is surjective for $\R \in\{ \K,\kk\}$, and that $\dim_{\kk}(\kk\otimes F)= \dim_{\K}(\K\otimes F)$. If $\kk\otimes f$ is injective, then $\K\otimes f$ is injective.
\end{lemma}
\begin{proof}
Since $\A$ is a principal ideal domain and $W$ is finitely generated, it follows that $W\cong \A^{\dim_{\K}(\K\otimes W)} \oplus T$ where $\K\otimes T=0$. Since $\A$ is also a local ring with maximal ideal $M$ and residue field $\A/M=\kk$, it follows that there are $r_1, \dots, r_d\in\mathbb{Z}_{\ge 1}$ such that $T\cong \oplus_{i=1}^d(\A/M)^{r_i}$ and $d = \dim_{\kk}(\kk\otimes W) - \dim_{\K}(\K\otimes W)$. In particular, 
\[
\dim_{\kk}(\kk\otimes W) - \dim_{\K}(\K\otimes W) \ge 0.
\]

Suppose $\kk\otimes f$ is injective. Then $\kk\otimes f$ is an isomorphism, so 
\[
\dim_{\kk}(\kk\otimes W) = \dim_{\kk}(\kk\otimes F),
\]
and $\K\otimes f$ is surjective, so 
\[
\dim_{\K}(\K\otimes W) \ge \dim_{\K}(\K\otimes F).
\]
We further assumed that $\dim_{\kk}(\kk\otimes F)= \dim_{\K}(\K\otimes F)$. It follows that
\[
\dim_{\kk}(\kk\otimes W) - \dim_{\K}(\K\otimes W) \le \dim_{\kk}(\kk\otimes F) - \dim_{\K}(\K\otimes F) = 0.
\]
Hence, 
\[
\dim_{\K}(\K\otimes W) = \dim_{\kk}(\kk\otimes W) = \dim_{\kk}(\kk\otimes F) = \dim_{\K}(\K\otimes F).
\]
Surjectivity of $\K\otimes f$ implies that $\K\otimes f$ is an isomorphism. 
\end{proof}

\begin{lemma}\label{L:tech-base-change-lemma-2}
Let $W,F$ be $\A$-modules and assume that $W$ is finitely generated over $\A$ and $F$ is free and finitely generated over $\A$. Let $f:W\rightarrow F$ be an $\A$-module homomorphism. Assume further that for $\R\in \{\kk, \K\}$, there are vector spaces $F_{\R}$ and linear maps $b_{\R}:\R\otimes F\rightarrow F_{\R}$, such that $b_{\K}$ is injective. Suppose that $b_{\R}\circ (\R\otimes f):\R\otimes W\rightarrow F_{\R}$ is surjective for $\R\in \{\kk, \K\}$, and that $\dim_{\kk}F_{\kk} = \dim_{\K}F_{\K}$. If $b_{\kk}\circ (\kk\otimes f)$ is injective, then $b_{\K}\circ (\K\otimes f)$ is injective. 
\end{lemma}
\begin{proof}
Let $\R\in \{\kk, \K\}$. Since $b_{\R}\circ (\R\otimes f)$ is surjective, it follows that $b_{\R}$ is surjective so $\dim_{\R} F_{\R}\le \dim_{\R}(\R\otimes F)$. Since $b_{\K}$ is injective, it follows that $\dim_{\K}(\K\otimes F)\le \dim_{\K}F_{\K}$. Thus, $b_{\K}$ is an isomorphism and $\dim_{\K}(\K\otimes F) = \dim_{\K}F_{\K}$. 

Using that $F$ is free over $\A$, we find
\[
\dim_{\kk}(\kk\otimes F) = \rk_{\A}F = \dim_{\K}(\K\otimes F)= \dim_{\K}F_{\K} = \dim_{\kk}F_{\kk}.
\]
Thus, surjectivity of $b_{\kk}$ implies $b_{\kk}$ is an isomorphism. 

Suppose $b_{\kk}\circ (\kk\otimes f)$ is injective. Since $b_{\R}$ is an isomorphism for $\R\in\{\kk, \K\}$, the claim follows from Lemma \ref{L:technical-base-change-lemma}.
\end{proof}

\begin{thm}\label{T:PhiK-std-is-faithful}
The functor 
\[
\Phi_{\K}|_{\StdWebK}:\StdWebK\longrightarrow \StdRepK
\]
is faithful.
\end{thm}
\begin{proof}
For $\R\in \{\kk, \K\}$ and $d,e\in \mathbb{Z}_{\ge 0}$ we have induced maps
\[
\Phi_{\R }:\Hom_{\StdWeb}(1^{\otimes d}, 1^{\otimes e})\rightarrow \Hom_{\StdRep}(V_{\R }^{\otimes d}, V_{\R }^{\otimes e}).
\]
It suffices to show that $\Phi_{\K}$ is injective for all $d,e\in \mathbb{Z}_{\ge 0}$. Recall from Theorem \ref{T:existence-of-functor} that $\Phi_\R  = \base_{\R}\circ (\R \otimes \Phi_{\A}) $. So we will use Lemma \ref{L:tech-base-change-lemma-2} when $W = \Hom_{\StdWebA}(1^{\otimes d}, 1^{\otimes e})$, $F= \Hom_{\StdRepA}(V_{\A}^{\otimes d}, V_{\A}^{\otimes e})$, $f = \Phi_{\A}$, $F_{\R}= \Hom_{\StdRep}(V_{\R}^{\otimes d}, V_{\R}^{\otimes e})$, and $b_{\R}=\base_{\R}$.

Proposition \ref{P:StdWeb-finitely-generated} says that $\Hom_{\StdWebA}(1^{\otimes d}, 1^{\otimes e})$ is a finitely generated $\A$-module, and Lemma \ref{L:FundA-homs-free-fin-gen} implies that $\Hom_{\StdRepA}(V_{\A}^{\otimes d}, V_{\A}^{\otimes e})$ is a free and finitely generated $\A$-module. Lemma \ref{L:sufficestocheckoverK} implies that $\base_{\K}$ is injective.

Theorem \ref{T:full} implies that $\Phi_{\R} = \base_{\R}\circ (\R\otimes \Phi_{\A})$ is surjective for $\R\in \{\kk, \K\}$. Proposition \ref{P:dimhomrep} says that $\dim_{\K} \Hom_{\StdRepK}(V_{\K}^{\otimes d}, V_{\K}^{\otimes e}) = \dim_{\kk} \Hom_{\StdRepC}(V_{\kk}^{\otimes d}, V_{\kk}^{\otimes e})$. 

By Corollary \ref{C:PhiC-std-is-faithful} we know that $\Phi_{\kk}$ is injective. Thus, Lemma \ref{L:tech-base-change-lemma-2} implies that $\Phi_{\K}$ is injective.
\end{proof}

\subsection{Main theorem}

We now prove Theorem \ref{T:main-thm}

\begin{proof}
Thanks to Lemma \ref{L:reduction-to-allones}, it  follows from Theorem \ref{T:full}, Corollary \ref{C:PhiC-std-is-faithful}, and Theorem \ref{T:PhiK-std-is-faithful} that $\Phi_\R$ is full and faithful. Since the objects of $\Fund$ are tensor products of $\Lambda_\R^k$, for $k\in \{0, 1, \dots, m\}$, and $\Phi_\R(k) = \Lambda_\R^k$, it follows that $\Phi_\R$ is essentially surjective. Hence, $\Phi_\R$ is an equivalence of $\R$-linear pivotal categories. 
\end{proof}

\begin{remark}
Since $\Fund$ is a ribbon category, with braiding $\brrep$, we can use the equivalence $\Phi_\R$ to define a braiding on $\Web$. We know that $\Phi_\R(\brweb) = \brrep$ and $\Phi_\R(\brwebi) = \brrepi$. Naturality of the braiding on $\Fund$ allows us to define a braiding on $\Web$, as in \cite[Section 5.9]{bodish2021type}. The functor $\Phi_\R$ can then be treated as an equivalence of braided pivotal (in fact ribbon) categories.
\end{remark}